\documentclass[moor]{informs1}              

\usepackage[sort]{natbib}
\NatBibNumeric
\def\bibfont{\small}%
\bibpunct[, ]{[}{]}{,}{n}{}{,}%

\usepackage[colorlinks=true,breaklinks=true,bookmarks=true,urlcolor=blue,
citecolor=blue,linkcolor=blue,bookmarksopen=false,draft=false]{hyperref}

\def\EMAIL#1{\href{mailto:#1}{#1}}
\def\URL#1{\href{#1}{#1}}         

\TheoremsNumberedBySection
\EquationsNumberedBySection 

\usepackage{lipsum}
\usepackage{amsfonts}
\usepackage{graphicx}
\usepackage{epstopdf}
\usepackage{algorithm}
\usepackage{algpseudocode}
\usepackage{amsopn}
\usepackage[caption=false]{subfig}

\usepackage{enumerate}
\usepackage[OT1]{fontenc}
\usepackage{mathrsfs}
\usepackage{hyperref}
\usepackage{float}
\usepackage{amsfonts,amsmath,amssymb,url,xspace}
\usepackage{verbatim}
\usepackage{mathtools}
\usepackage[bottom]{footmisc}
\usepackage{enumitem}
\usepackage{arydshln,leftidx,mathtools}
\usepackage{lscape}
\usepackage{chngcntr}
\usepackage{soul}
\usepackage{multirow}

\allowdisplaybreaks[1]

\newcommand{\LD}{\langle}
\newcommand{\RD}{\rangle}

\newcommand{\mR}{\mathbb{R}}

\newcommand{\mC}{\mathcal{C}}
\newcommand{\mL}{\mathcal{L}}
\newcommand{\mM}{\mathcal{M}}

\newcommand{\mB}{\mathcal{B}}

\newcommand{\mD}{\mathcal{D}}

\def\I{\mathcal{I}}
\def\J{\mathcal{J}}
\def\K{\mathcal{K}}
\def\P{{\mathcal P}}

\def\L{{\mathcal L}}
\def\T{{\mathcal T}}

\def\0{{\boldsymbol 0}}
\def\ba{{\boldsymbol{a}}}
\def\bb{{\boldsymbol{b}}}

\def\LP{{\L\P}}
\def\H{{\mathcal H}}

\def\bv{{\boldsymbol{v}}}

\def\bx{{\boldsymbol{x}}}
\def\bu{{\boldsymbol{u}}}
\def\bd{{\boldsymbol{d}}}
\def\bp{{\boldsymbol{p}}}
\def\bq{{\boldsymbol{q}}}
\def\bw{{\boldsymbol{w}}}
\def\bzeta{{\boldsymbol{\zeta}}}

\def\bz{{\boldsymbol{z}}}
\def\blambda{{\boldsymbol{\lambda}}}
\def\tlambda{{\boldsymbol{\lambda}}^\star}

\def\tx{{\boldsymbol{x}}^\star}
\def\tu{{\boldsymbol{u}}^\star}
\def\tz{{\boldsymbol{z}}^\star}
\def\tw{{\boldsymbol{w}}^\star}
\def\tp{{\boldsymbol{p}}^\star}
\def\tq{{\boldsymbol{q}}^\star}
\def\tzeta{{\boldsymbol{\zeta}}^\star}

\def\bblambda{\bar{\boldsymbol{\lambda}}}

\def\tbq{{\tilde{\boldsymbol{q}}}}
\def\tbp{{\tilde{\boldsymbol{p}}}}

\def\tbx{{\tilde{\boldsymbol{x}}}}
\def\tbu{{\tilde{\boldsymbol{u}}}}
\def\bbx{\bar{\boldsymbol{x}}}
\def\bbu{\bar{\boldsymbol{u}}}

\def\tg{{\tilde{g}}}

\def\tf{{\tilde{f}}}
\def\tbz{{\tilde{\boldsymbol{z}}}}
\def\tblambda{{\tilde{\blambda}}}

\def\tbw{{\tilde{\boldsymbol{w}}}}
\def\tbp{{\tilde{\boldsymbol{p}}}}
\def\tbq{{\tilde{\boldsymbol{q}}}}
\def\tbzeta{{\tilde{\boldsymbol{\zeta}}}}

\def\hH{{\hat{H}}}
\newcommand{\mZ}{\mathcal{Z}}
\def\hQ{{\hat{Q}}}
\def\hS{{\hat{S}}}
\def\hR{{\hat{R}}}
\def\bomega{{\boldsymbol{\omega}}}

\def\bbp{\bar{\boldsymbol{p}}}
\def\bbq{\bar{\boldsymbol{q}}}
\def\bbzeta{\bar{\boldsymbol{\zeta}}}
\def\tDelta{{\tilde{\Delta}}}
\newcommand{\rbr}[1]{\left(#1\right)}

\newcommand{\cbr}[1]{\left\{#1\right\}}
\newcommand{\nbr}[1]{\left\|#1\right\|}

\DeclareMathOperator{\diag}{diag}

\def\tblambdai{\tblambda_i^{\star}}
\def\tbzi{\tbz_i^{\star}}

\def\tbwi{\tbw_i^{\star}}

\def\tbzetai{\tbzeta_i^{\star}}

\usepackage{xargs}
\usepackage[colorinlistoftodos,prependcaption,textsize=tiny]{todonotes}
\newcommandx{\unsure}[2][1=]{\todo[linecolor=red,backgroundcolor=red!25,bordercolor=red,#1]{#2}}
\newcommandx{\change}[2][1=]{\todo[linecolor=blue,backgroundcolor=blue!25,bordercolor=blue,#1]{#2}}
\newcommandx{\info}[2][1=]{\todo[linecolor=OliveGreen,backgroundcolor=OliveGreen!25,bordercolor=OliveGreen,#1]{#2}}
\newcommandx{\improvement}[2][1=]{\todo[linecolor=Plum,backgroundcolor=Plum!25,bordercolor=Plum,#1]{#2}}

\RequirePackage[capitalize,nameinlink]{cleveref}[0.19]
\crefname{section}{section}{sections}

\begin{document}

\RUNAUTHOR{Na, Anitescu, and Kolar}

\RUNTITLE{A Fast Temporal Decomposition Procedure for NLDP}

\TITLE{A Fast Temporal Decomposition Procedure for Long-horizon Nonlinear Dynamic Programming}

\ARTICLEAUTHORS{%
\AUTHOR{Sen Na}
\AFF{ICSI and Department of Statistics, University of California, Berkeley, \EMAIL{senna@berkeley.edu} \URL{}}
\AUTHOR{Mihai Anitescu}
\AFF{Mathematics and Computer Science Division, Argonne National Laboratory, \EMAIL{anitescu@mcs.anl.gov} \URL{}}
\AUTHOR{Mladen Kolar}
\AFF{Booth School of Business, University of Chicago, \EMAIL{mladen.kolar@chicagobooth.edu} \URL{}}
}

\ABSTRACT{We propose a \textit{fast} temporal decomposition procedure for solving long-horizon nonlinear dynamic programs. The core of the procedure is sequential quadratic programming (SQP) that utilizes a differentiable exact augmented Lagrangian as the merit function. Within each SQP iteration, we approximately solve the Newton system using an overlapping temporal decomposition strategy. We show that the approximate search~direction is still a descent direction of the augmented Lagrangian, provided the overlap size and penalty parameters are suitably chosen, which allows us to establish the global convergence. Moreover, we show that a unit~stepsize is accepted locally for the approximate search direction, and further establish a uniform,~local~\mbox{linear}~convergence over stages. This local convergence rate matches the rate of the recent Schwarz scheme \cite{Na2022Convergence}. However, the Schwarz scheme has to solve nonlinear subproblems to optimality in~each~iteration, while we only perform a single Newton step instead. Numerical experiments validate our theories and demonstrate the superiority of our method.
}


\maketitle

\section{Introduction}\label{sec:1}
We consider nonlinear equality-constrained dynamic programs (NLDPs):
\vskip-10pt
\begin{subequations}\label{pro:1}
\begin{align}
\min_{\bx, \bu}\;\;\; & \sum_{k=0}^{N-1} g_k(\bx_k, \bu_k) + g_N(\bx_N), \label{pro:1a}\\
\text{s.t. }\; & \bx_{k+1}   = f_k(\bx_k, \bu_k), \quad k = 0, 1, \ldots, N-1, \label{pro:1b}\\
& \bx_0 = \bbx_0, \label{pro:1c}
\end{align}
\end{subequations}
where $\bx_k \in\mR^{n_x}$ is the state variable, $\bu_k\in\mR^{n_u}$ is the control variable, $g_k: \mR^{n_x}\times \mR^{n_u} \rightarrow \mR$ ($g_N: \mR^{n_x}\rightarrow \mR$) is the cost function, $f_k: \mR^{n_x}\times \mR^{n_u}\rightarrow \mR^{n_x}$ is the dynamical constraint function, $\bbx_0$ is the given initial state, and $N$ is the temporal horizon length. In the control literature, Problem \eqref{pro:1}~is also called nonlinear optimal control problem. This paper focuses on solving \eqref{pro:1} with a large $N$.

Large-scale nonlinear problems pose significant computational challenges due to the nonlinearity of the problems and the large number of variables that need to be optimized. Numerous solvers have been developed to address the scalability issue from different aspects. For example, IPOPT~\citep{Waechter2005implementation} is a primal-dual interior point method that employs the logarithmic barrier function with the filter line-search step. It enjoys both global and local superlinear convergence \citep{Waechter2005Line, Waechter2005Linea}.~As~another~example, Knitro \citep{Byrd2006Knitro} integrates the advantages of two complementary methods---the interior point method and the active-set method---in a unified package, and achieves~the robust~performance. We refer~to \cite{Petra2014Augmented, Chiang2014Structured, Wan2016Structured} for the other scalable high-performance solvers for nonlinear programs.

The long-horizon NLDPs are of special interest as they appear in a variety of applications~including portfolio management \citep{Topaloglou2008dynamic}, autonomous vehicle \citep{Frasch2013auto, Zanon2014Model}, and power planning \citep{Dias2013Parallel}. The aforementioned centralized solvers, while exploiting the problem-dependent sparse structures to accelerate the computation, are not particularly designed for NLDPs. As these solvers do not take advantage of generic properties of dynamical systems, they are deficient when directly implemented~on~Problem \eqref{pro:1}, especially when we have limited computing resources. The generic properties of dynamical systems, such as sensitivity and controllability, play a key role in developing various efficient~algorithms \citep{Xu2018Exponentially, Na2020Exponential, Shin2019Parallel, Shin2020Decentralized, Shin2021Diffusing}. Further, the centralized solvers are not flexible enough to be implemented on different types of computing hardware. Their performance heavily relies on a single high-speed processor, which sometimes is inaccessible. These limitations motivate us to design a parallel,~DP-oriented procedure for \eqref{pro:1}.

In contrast to solving NLDPs in real time \citep{Diehl2002Real, Diehl2005Nominal, Diehl2009Efficient, Zanelli2020Stability}, the dominant class of distributed~offline methods is the decomposition-based methods, where the full problem is decomposed into multiple subproblems, which are then solved in a parallel environment \citep{Wright1990Solution}. For example, temporal~decomposition \citep{Barrows2014Time, Laine2019Parallelizing}, Lagrangian dual decomposition \citep{Lemarechal2001Lagrangian, Beccuti2004Temporal}, and alternating direction method of multipliers (ADMM) \citep{Boyd2010Distributed, ODonoghue2013splitting} are widely adopted in practice. The decomposition-based methods are preferable when the problem scale (i.e., $N$ in our case) is so large that solving \eqref{pro:1} on a single processor is computationally prohibitive, but multiple processors can be used to solve subproblems in parallel.

This paper contributes to the literature on the overlapping temporal decomposition (OTD) methods \cite{Barrows2014Time, Xu2018Exponentially, Shin2019Parallel, Na2022Convergence}. A TD method decomposes the full temporal horizon $[0, N]$ into several short horizons $[0, N]\subseteq \cup_{i}[n_i, n_{i+1}]$; constructs subproblems $\{\P^i\}_i$ associated with short horizons~$\{[n_i, n_{i+1}]\}_i$; solves all subproblems in parallel; and retrieves the full-horizon solution by composing the subproblem solutions sequentially. Different from the exclusive decomposition in TD, OTD extends each short horizon $[n_i, n_{i+1}]$ by $b$ stages on two ends to encourage information exchange between two adjacent subproblems. The composition is performed by using only the exclusive part, $[n_i, n_{i+1}]$, of each subproblem's solution. OTD was first empirically studied in \cite{Barrows2014Time} on power planning problems, and rigorously analyzed for linear-quadratic convex DPs in \cite{Xu2018Exponentially}. In this work, the authors showed under the controllability and boundedness conditions that $\max_k\|(\hat{\bx}_k - \tx_k; \hat{\bu}_k-\tu_k)\| \leq C\rho^b$~for~some constants $C>0$ and $\rho\in(0, 1)$. Here, $(\hat{\bx}_k, \hat{\bu}_k)$ is the OTD output at the stage $k$ and $(\tx_k, \tu_k)$ is the optimal solution. For general NLDPs, \cite{Shin2019Parallel, Na2022Convergence} recently proposed an overlapping Schwarz scheme~as an application of OTD on nonlinear problems. We will review the Schwarz scheme in \cref{sec:2} but briefly introduce it here to motivate our study.

Following the same spirit as OTD, the Schwarz scheme decomposes $[0, N]$ by $[0, N]\subseteq~\cup_{i}~[m_1^i, m_2^i]$, where $m_1^i = n_i-b$ and $m_2^i = n_{i+1}+b$ are two boundaries of the extended \mbox{intervals}.~The~\mbox{$i$-th}~subproblem $\P^i = \P^i(\bd_i)$ is parameterized by some boundary variables $\bd_i = (\bd_{m_1^i}; \bd_{m_2^i})$.~In~the~\mbox{$\tau$-th}~\mbox{iteration}, one first specifies the boundary variables $\bd_i^\tau$ by the solutions of adjacent \mbox{subproblems}~(\mbox{noting}~that $m_1^i, m_2^i\notin [n_i, n_{i+1}]$); then solves the subproblems $\{\P^i(\bd_i^\tau)\}_i$ to \textit{optimality} in parallel and updates~the boundary variables to $\bd_i^{\tau+1}$. The solutions to subproblems are finally composed to derive a~full-horizon solution. The Schwarz scheme is demonstrated in Figure \ref{fig:1}. Clearly,~the~OTD step corresponds to one~\mbox{iteration} of the Schwarz scheme. \cite{Na2022Convergence} empirically showed that the Schwarz~scheme~significantly~improves~the~\mbox{efficiency}~of~ADMM;~and~may~be~as~\mbox{efficient}~as~the~\mbox{centralized}~solver~IPOPT, but provide significant flexibility on different computing environments. By adjusting the overlap size $b$ of the decomposition, the scheme adapts to centralized or to decentralized environments. Furthermore, unlike ADMM whose convergence for nonlinear \mbox{problems}~is~\mbox{established}~for~some~setups~that may not directly apply to \eqref{pro:1} (see \cite{Hong2016Convergence, Wang2018Global}), the Schwarz scheme exhibits a uniform local~linear~convergence. Under the same conditions of OTD on linear-quadratic convex DPs in \cite{Xu2018Exponentially},~\cite{Na2022Convergence}~showed $\max_k\|(\bx_k^\tau - \tx_k; \bu_k^\tau-\tu_k)\|  \leq (C\rho^b)^\tau\max_k\|(\bx_k^0 - \tx_k; \bu_k^0-\tu_k)\|$, where $(\bx_k^\tau, \bu_k^\tau)$ is the $\tau$-th iterate of the Schwarz scheme. Despite this promising result, the Schwarz scheme has two limitations.

\begin{figure}[tbp]
\centering
\includegraphics[width=13cm]{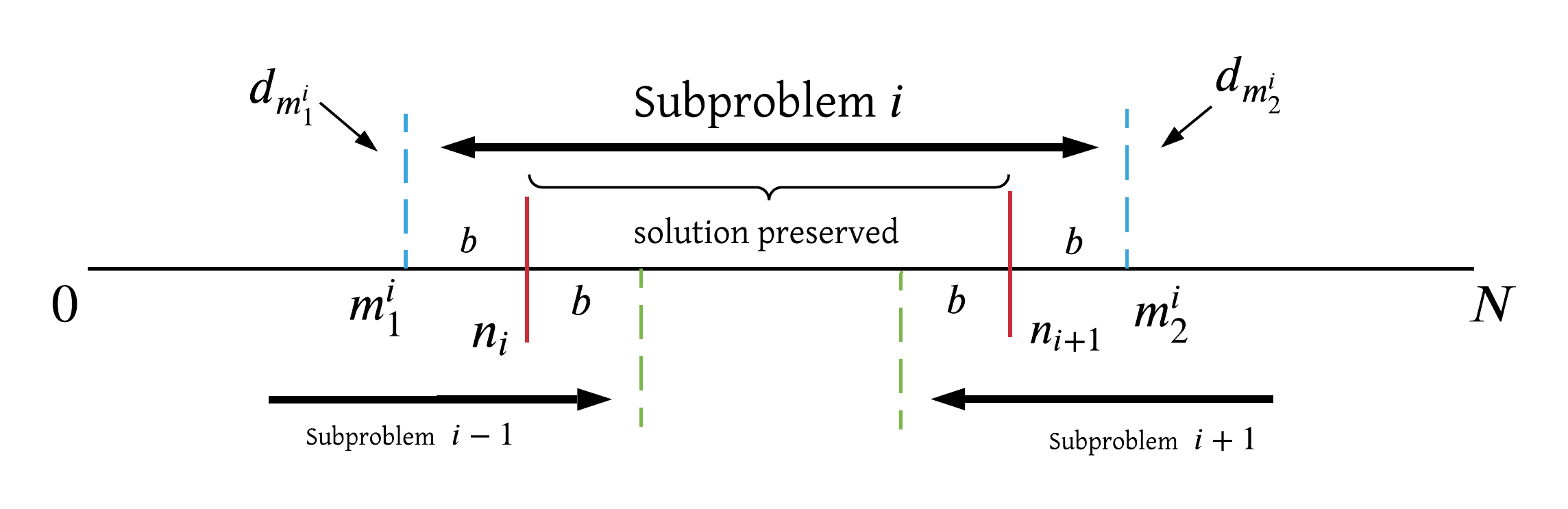}
\vskip-0.5cm
\caption{Demonstration of the Schwarz scheme. The horizontal line is the full horizon $[0, N]$. The red vertical~lines are knots of the exclusive intervals; the blue vertical lines are boundaries of the extended intervals. The subproblem $i$ parameterized by $\bd_i = (\bd_{m_1^i}, \bd_{m_2^i})$ is solved, but only the exclusive part of the solution is used in the composition.}
\label{fig:1}
\end{figure}

First, only local convergence guarantee is established while global convergence is still unknown. In fact, as we will explain later, the Schwarz scheme does not converge globally in general. There is no guarantee that the composed solution of the subproblems~is~the~solution to the full problem, even if we correctly specify the boundary variables for each subproblem. Thus, we arise the question:
\vskip4pt
{\em Q1: How to design an OTD-based procedure that converges globally?}
\vskip4pt
Second, the Schwarz scheme is computationally expensive. In each iteration, the scheme solves~all nonlinear subproblems $\{\P^i(\bd_i)\}_i$ to \textit{optimality}. However, we want to know:
\vskip4pt
{\em Q2: Is it necessary to solve subproblems to optimality in each iteration to enjoy (uniform) local linear convergence?}
\vskip4pt

We answer Q1 and Q2 by designing a \textit{fast} OTD (FOTD) procedure. Our procedure is inspired by \cite{Wright1990Solution}, where the author applied the sequential quadratic programming (SQP) to solve \eqref{pro:1}, and a parallel block-banded linear solver to solve the exact Newton system. By using the unit stepsize,~\cite{Wright1990Solution} conducted a local analysis. In this paper, we propose a FOTD procedure to integrate~global and local analyses. Similar to \cite{Wright1990Solution}, FOTD is also built upon an SQP scheme. However, it utilizes an exact augmented Lagrangian merit function to adaptively select the stepsize via line search; and adopts an OTD strategy to approximately solve the Newton system. While there are many distributed methods for solving the Newton system exactly or approximately, e.g., \cite{Nielsen2015parallel, Nielsen2016O, Laine2019Parallelizing}, we specifically focus on OTD in order to build a relation to the Schwarz scheme. We prove that FOTD converges globally. A key technical step is to show that the approximate direction is a descent direction of the augmented Lagrangian, so that the iterates are improved towards the KKT point. This technical step justifies the choice of the augmented Lagrangian merit function. Furthermore,~we~show~that~the unit stepsize for the approximate direction is accepted locally, so that FOTD enjoys a uniform~local linear convergence. Such a linear convergence matches the one of the Schwarz scheme \cite{Na2022Convergence}; however, FOTD requires much fewer computations as it does not solve nonlinear subproblems to optimality. Our experiments validate the theorems and demonstrate the~superiority of FOTD.

\vskip 4pt

\noindent{\bf Structure of the paper:} In \cref{sec:2}, we present the preliminaries of OTD and review the~Schwarz scheme. In \cref{sec:3}, we introduce our FOTD procedure. In \cref{sec:4}, we study the approximation error of the Newton system. The global and local convergence results are established in \cref{sec:5} and \cref{sec:6}, respectively. Numerical experiments are presented in \cref{sec:7} followed by conclusions in \cref{sec:8}. All the proofs are collected in appendices to make the main paper compact.

\vskip 4pt

\noindent{\bf Notation:} For an integer $n$, $[n] \coloneqq  \{0, 1, \ldots, n\}$. For two integers $n, m$, we abuse the interval notation and let $[n, m]$, $(n, m)$, $[n, m)$, $(n, m]$ be the corresponding index sets. All vectors in the~paper~are column vectors. For two vectors $\ba$, $\bb$, $(\ba; \bb)$ denotes the column vector that stacks $\ba$ and $\bb$ sequentially. For a vector-valued function $f: \mR^n\rightarrow \mR^m$, $\nabla f\in\mR^{n\times m}$ is its Jacobian matrix. We let $\|\cdot\|$ denote the $\ell_2$ norm for vectors and the spectral norm for matrices. For a symmetric matrix $A$, $\lambda_{\min}(A)$ denotes its smallest eigenvalue. We let $I$ be the identity matrix and $\0$ be the zero matrix, whose dimensions can be inferred from the context. We also reserve~the following notation. $\bx = \bx_{0:N} = (\bx_0; \ldots; \bx_N)$~is the state vector; $\bu = \bu_{0:N-1} = (\bu_0; \ldots \bu_{N-1})$ is the control vector; $\bz_k = (\bx_k; \bu_k)$ ($\bz_N = \bx_N$) is the state-control pair at stage $k$, and $\bz = \bz_{0:N} = (\bz_0; \ldots; \bz_N)$. We let $n_z = (N+1)n_x + Nn_u$ and $\bz\in \mR^{n_z}$. We may also express $\bz = (\bx, \bu)$ when explicitly specifying the components $\bx, \bu$ of $\bz$.

\section{Preliminaries}\label{sec:2}

\hskip-6pt We start by setting up an overlapping temporal decomposition (OTD). Given the full horizon $[0, N]$, we decompose it into $M$ exclusive intervals with knots $0 = n_0<n_1<\cdots<n_M = N$. For example, we can choose evenly spaced knots $n_i = i\cdot N/M$ (suppose $M$ is a divisor of $N$). Each one of these intervals is then extended by $b\geq 1$ stages on two ends; that is, the boundaries of the extended intervals are
\begin{equation}\label{equ:m}
m_1^i = \max\{n_i - b,\; 0\}, \quad\quad m_2^i = \min\{n_{i+1} +b,\; N\}, \quad\quad i = 0,1,\ldots, M-1.
\end{equation}
To simplify the notation, we use $m_1, m_2$ to denote the boundaries of a general extended interval $i$. We note that two successive extended intervals overlap on $2b$ stages.

For the extended interval $i$, we define the corresponding subproblem as
\begin{subequations}\label{pro:2}
\begin{align}
\P^i_{\mu}(\bd_i):\quad  \min_{\tbx_i, \tbu_i}\;\;\; & \sum_{k=m_1}^{m_2-1} g_k(\bx_k, \bu_k) + \tg_{m_2}(\bx_{m_2}; \bd_{i, 2:4}), \label{pro:2a}\\
\text{s.t. }\; & \bx_{k+1}   = f_k(\bx_k, \bu_k), \quad k \in[m_1, m_2), \label{pro:2b}\\
& \bx_{m_1} = \bd_{i, 1}, \label{pro:2c}
\end{align}
\end{subequations}
where $\tbx_i = \bx_{m_1:m_2}$ and $\tbu_i = \bu_{m_1:m_2-1}$ are the state and control variables; $\bd_i = \bd_{i, 1:4}=(\bbx_{m_1}; \bbx_{m_2}; \\ \bbu_{m_2}; \bblambda_{m_2+1})$ are the boundary variables; $\tg_{m_2}(\bx_{m_2}; \bd_{i, 2:4}) = g_N(\bx_N)$ if $i=M-1$ (i.e., the last subproblem), otherwise for $\mu>0$,
\begin{equation}\label{equ:terminal:cost}
\tg_{m_2}(\bx_{m_2}; \bd_{i, 2:4}) = g_{m_2}(\bx_{m_2}, \bbu_{m_2}) - \bblambda_{m_2+1}^Tf_{m_2}(\bx_{m_2}, \bbu_{m_2}) + \frac{\mu}{2}\|\bx_{m_2} - \bbx_{m_2}\|^2.
\end{equation}
For the boundary variables $\bd_i= (\bd_{i,1}; \bd_{i,2:4})$, $\bd_{i,1} = \bbx_{m_1}$ is the initial state \mbox{variable};~$\bd_{i, 2:4} =~(\bbx_{m_2}; \bbu_{m_2}; \\ \bblambda_{m_2+1})$ are the terminal state, control, and dual variables. The boundary variables $\bd_i = (\bd_{i,1}; \bd_{i,2:4})$ are given to the subproblem. In \eqref{equ:terminal:cost}, $\mu$ is a \textit{uniform} quadratic penalty parameter independent~from $i$. In what follows, we denote by $\tbz_i = (\tbx_i, \tbu_i)$ the ordered state-control vector of the subproblem $i$.

The subproblem $\P^i = \P^i_{\mu}(\bd_i)$ in \eqref{pro:2} is essentially the truncation of the full problem \eqref{pro:1} onto the interval $[m_1, m_2]$, except that the initial state is fixed at $\bd_{i,1}= \bbx_{m_1}$ and the terminal cost on $\bx_{m_2}$ is adjusted by $\tg_{m_2}(\bx_{m_2}; \bd_{i, 2:4})$. The subproblem is parameterized by $\bd_i$; we always~specify~$\bd_i$~based on the previous iterate before solving $\P^i_{\mu}(\bd_i)$. The formula \eqref{equ:terminal:cost} was first proposed by \cite{Na2023Superconvergence} for analyzing the real-time model predictive control schemes. The middle term depending~on~$\bblambda_{m_2+1}$~and~$f_{m_2}$ reduces the KKT residual brought by the horizon truncation, and the quadratic penalty term~\textit{convexifies} the subproblem. The benefits of \eqref{equ:terminal:cost} will be clearer later. We clarify two corner cases:~for $i = 0$ and $m_1 = 0$, $\bbx_{m_1}$ is $\bbx_0$ from Problem \eqref{pro:1}; for $i=M-1$ and $m_2=N$, the adjustment on~the terminal cost is restored. We mention that there exist other subproblem formulations in some restrictive setups: if $g_{m_2}$, $f_{m_2}$ are separable, then $\bbu_{m_2}$ is not needed in \eqref{equ:terminal:cost}; if $g_{m_2}$ is quadratic and convex and $f_{m_2}$ is affine, then we can let $\mu=0$. All these formulations ensure that the subproblem $\P^i_{\mu}(\bd_i)$ is well-defined (e.g., is lower bounded) given a well-defined full problem.

Before introducing the Schwarz scheme, we need the following notation. We let $\blambda = \blambda_{0:N}$ be~the dual variables of \eqref{pro:1}, where $\blambda_{0}\in\mR^{n_x}$ is associated with the constraint in \eqref{pro:1c} and $\blambda_{k+1}\in\mR^{n_x}$ for $k\in[N-1]$ is associated with the $k$-th constraint in \eqref{pro:1b}. Similarly, we let $\tblambda_i = \blambda_{m_1:m_2}$ be the dual variables of \eqref{pro:2}. For the state variables $\tbx_i$ (similar for $\tbu_i, \tbz_i, \tblambda_i$) and $k\in[m_1, m_2]$,~$\tbx_{i, k}$~denotes the variable at the stage $k$ in the subproblem $i$. For example, $k = n_i$ belongs to both subproblem $i-1$ and subproblem $i$; thus $\tbx_{i-1, n_i}$ and $\tbx_{i, n_i}$ refer to different variables at the same stage. The primal-dual solution of $\P^i_{\mu}(\bd_i)$ is denoted by $(\tbzi(\bd_i), \tblambdai(\bd_i))$. We also define composition~and~decomposition operators in the next definition.

\begin{definition}[composition and decomposition]\label{def:1}
\hskip -8pt For the subproblem variables $\{(\tbz_i, \tblambda_i)\}_{i=0}^{M-1}$, we define a composition operator $\mC$ as $\mC(\{(\tbz_i, \tblambda_i)\}_i) = (\bz, \blambda)$, where $(\bz_k,\blambda_k) = (\tbz_{i, k}, \tblambda_{i, k})$ if $k\in [n_i, n_{i+1})$ for $i\in[M-1]$, and $\bz_N = \tbx_{M-1, N}$ and $\blambda_{N} = \tblambda_{M-1, N}$. Conversely, for the full-horizon~variable $(\bz, \blambda)$, we define a decomposition operator $\mD$ as $\mD(\bz, \blambda) =\{\mD_i(\bz, \blambda)\}_{i=0}^{M-1}$, where $\mD_i(\bz, \blambda) = (\tbz_i, \tblambda_i)$ with $\tbz_i = (\tbx_i, \tbu_i) = (\bx_{m_1:m_2}, \bu_{m_1:m_2-1})$ and $\tblambda_i = \blambda_{m_1:m_2}$.
\end{definition}

From Definition \ref{def:1}, we see that the variables on the overlapping stages are discarded during~the composition. That is, $\tbx_i = \tbx_{i, m_1:m_2}$ (similar for $\tbu_i, \tblambda_i$) only contributes $\tbx_{i, n_i:n_{i+1}-1}$ to the full vector.

\vskip 4pt
\noindent\textbf{Schwarz scheme.}
Given the $\tau$-th iterate $(\bz^\tau, \blambda^\tau)$, we specify the boundary variables by $\bd_i^\tau = (\bx_{m_1}^\tau; \\ \bx_{m_2}^\tau; \bu_{m_2}^\tau; \blambda_{m_2+1}^\tau)$; solve the subproblems $\{\P^i_{\mu}(\bd_i^\tau)\}_i$ in parallel to \textit{optimality}; and obtain the solutions $\{(\tbzi(\bd_i^\tau), \tblambdai(\bd_i^\tau))\}_i$. Then, $(\bz^{\tau+1}, \blambda^{\tau+1}) = \mC(\{(\tbzi(\bd_i^\tau), \tblambdai(\bd_i^\tau))\}_i)$. The Schwarz scheme is displayed in Algorithm \ref{alg:OTD}. The convergence of the scheme is achieved by iteratively updating $\bd_i$. One expects that, as $\tau$ increases, $\bd_i^\tau$ becomes more~\mbox{precise}~so~that~the~\mbox{solution}~of~$\P^i_{\mu}(\bd_i^\tau)$~approaches~to~the truncated full solution. Thus, the composed solution of subproblems can recover the full~solution.

The following result characterizes the relation between the optimality conditions of the subproblems \eqref{pro:2} and the full problem \eqref{pro:1}.

\begin{algorithm}[!tp]
\caption{Overlapping Schwarz Decomposition Procedure}
\label{alg:OTD}
\begin{algorithmic}[1]
\State \textbf{Input:} initial iterate $(\bz^0, \blambda^0)$ with $\bx_0^0 = \bbx_0$, a scalar $\mu>0$;
\For{$\tau = 0,1,2,\ldots$}
\For{$ i = 0, 1, \ldots, M-1$ \textbf{(in parallel)}}
\State Let $\bd_i^\tau = (\bx_{m_1}^\tau; \bx_{m_2}^\tau; \bu_{m_2}^\tau; \blambda_{m_2+1}^\tau)$ (note that $\bd_{M-1}^\tau = \bx_{m_1}^\tau$);
\State Solve $\P^i_\mu(\bd_i^\tau)$ to \textbf{optimality} and obtain the solution $(\tbzi(\bd_i^\tau), \tblambdai(\bd_i^\tau))$;
\EndFor
\State Let $(\bz^{\tau+1}, \blambda^{\tau+1}) = \mC(\{(\tbzi(\bd_i^\tau), \tblambdai(\bd_i^\tau))\}_i)$;
\EndFor
\end{algorithmic}
\end{algorithm}

\begin{theorem}[relation of KKT conditions]\label{thm:1}
For any scalar $\mu$, we have two cases:
	
\noindent\textbf{(i)} Suppose $(\bz^\star, \blambda^\star)$ is a KKT point of \eqref{pro:1}, then $\mD_i(\bz^\star, \blambda^\star)$, $\forall i\in[M-1]$, is~a~KKT~point~of~$\P^i_{\mu}(\bd_i^\star)$ where $\bd_i^\star = (\tx_{m_1}; \tx_{m_2}; \tu_{m_2}; \tlambda_{m_2+1})$;
	
\noindent\textbf{(ii)} Suppose $(\tbzi, \tblambdai) = (\tbzi(\bd_i), \tblambdai(\bd_i))$, $\forall i\in[M-1]$, is a KKT point of $\P^i_{\mu}(\bd_i)$ with any~boundary variables $\bd_i$ satisfying $\bd_{0, 1} = \bbx_0$, then $\mC(\{(\tbzi, \tblambdai)\}_i)$ is a KKT point of \eqref{pro:1} if and only~if~the~solutions of any two successive subproblems are compatible at the common boundaries. That is, for any $i \in[1, M-1]$ and knot $k = n_i$, we have\vskip-0.2cm
\begin{equation}\label{equ:con:cond}
\tbx_{i, k}^\star = \tbx_{i-1, k}^\star, \quad A_{k-1}^T(\tblambda_{i, k}^{\star} - \tblambda_{i-1, k}^{\star}) = \0, \quad B_{k-1}^T(\tblambda_{i, k}^{\star} - \tblambda_{i-1, k}^{\star}) = \0,
\end{equation}
where $A_{k-1} = \nabla_{\bx_{k-1}}^Tf_{k-1} \in \mR^{n_x\times n_x}$, $B_{k-1} = \nabla_{\bu_{k-1}}^Tf_{k-1} \in \mR^{n_x\times n_u}$ are the Jacobian matrices of $f_{k-1}$ evaluated at $\tbz_{i-1, k-1}^\star$.
\end{theorem}

\begin{proof}
See Appendix \ref{pf:thm:1}.
\end{proof} 

From Theorem \ref{thm:1}(i), we can see that the truncated KKT point $\mD_i(\bz^\star, \blambda^\star)$ is also a KKT point~of the subproblem, provided the boundary variables are correctly specified (i.e., $\bd_i = \bd_i^\star$). However, Theorem \ref{thm:1}(ii) provides a negative view: even if one obtains a KKT point for each subproblem with any boundary variables, composing the solutions together does not necessarily result in a~KKT point of the full problem. Only if the successive KKT points are compatible on knots $\{n_i\}_{i=1}^{M-1}$~can we guarantee to have a full-horizon KKT point. The subtlety lies in the fact that the solution at stage $n_i$ is from subproblem $i$, while the solution at stage $n_i-1$ is from subproblem $i-1$. Thus, \eqref{equ:con:cond} is needed to link solutions from two successive subproblems. By Theorem \ref{thm:1}(ii), we know the Schwarz scheme, which simply composes subproblem solutions, may not converge~globally~in~general.

However, if $(\bz^0, \blambda^0)$ is sufficiently close to $(\tz, \tlambda)$, \cite{Na2022Convergence} showed that for $i\in[M-1]$, $\P_{\mu}^i(\bd_i)$~has~a unique solution $\mD_i(\tz, \tlambda)$ in a neighborhood of $\bd_i^\star$. 
Thus, one expects $(\tbzi(\bd_i^\tau), \tblambdai(\bd_i^\tau)) \rightarrow \mD_i(\tz, \tlambda)$ as $\tau\rightarrow \infty$, and $\mC(\{(\tbzi(\bd_i^\tau), \tblambdai(\bd_i^\tau)) \}_i) \rightarrow \mC(\{\mD_i(\tz, \tlambda)\}_i) = (\tz, \tlambda)$. Then, the local convergence~of the Schwarz scheme is ensured. Specifically, \cite{Na2022Convergence} proved for some $C>0$ and~$\rho\in(0, 1)$ that
\begin{equation}\label{equ:local:result}
\max_k\|(\bz_k^\tau - \tz_k; \blambda_k^\tau - \tlambda_k)\|\leq
(C\rho^b)^\tau \cdot \max_k\|(\bz_k^0 - \tz_k; \blambda_k^0 - \tlambda_k)\|.
\end{equation}

In addition to lacking global convergence, the Schwarz scheme is computationally intensive, since Line 5 of Algorithm \ref{alg:OTD} requires finding optimal solutions to nonlinear subproblems. In the~next~section, we relax this computational requirement by using an SQP scheme, and finally show in \cref{sec:6} that \eqref{equ:local:result} holds even when Line 5 is substituted by one Newton step (cf. Theorem \ref{thm:8}).

\section{Embedding OTD into SQP}\label{sec:3}

We note that the lack of global convergence of the Schwarz scheme is due to the lack of a coordinator, which can monitor the convergence progress towards the full solution. This motivates us to solve \eqref{pro:1} under an SQP framework. For each SQP iteration, we apply OTD to approximately solve the Newton system, which corresponds to a linear-quadratic DP problem. Although other distributed methods are also applicable, we are particularly interested in OTD since it reveals a nice relation to the Schwarz scheme. By embedding OTD into SQP, we are able to establish global convergence, which resolves one of the limitations of the Schwarz scheme.

We write Problem \eqref{pro:1} in a compact form by
\vskip-5pt
\begin{equation*}
\min_{\bz} \;\; g(\bz),\;\;\quad \text{s.t.}\;\; f(\bz) = \0,
\end{equation*}
where
\begin{equation}\label{equ:def:gf}
\begin{aligned}
g(\bz) & = \sum_{k=0}^{N} g_k(\bz_k) =  \sum_{k=0}^{N-1}g_k(\bx_k,\bu_k) + g_N(\bx_N),\\
f(\bz) & = (\bx_0 - \bbx_0; \bx_1- f_0(\bz_0);\ldots; \bx_{N} - f_{N-1}(\bz_{N-1}))\\
&  =  (\bx_0 - \bbx_0; \bx_1- f_0(\bx_0, \bu_0);\ldots; \bx_{N} - f_{N-1}(\bx_{N-1},\bu_{N-1})).
\end{aligned}
\end{equation}
The Lagrangian function is $\mL(\bz, \blambda) = g(\bz) + \blambda^Tf(\bz)$ and the  KKT conditions are $\nabla_{\bz}\mL(\tz, \tlambda) = \0$, $\nabla_{\blambda}\mL(\tz, \tlambda) = \0$. The SQP scheme applies Newton's method on the KKT system. In~particular, given the $\tau$-th iterate $(\bz^\tau, \blambda^\tau)$, the Newton direction $(\Delta \bz^\tau, \Delta \blambda^\tau)$ is obtained by solving
\vskip-5pt\begin{equation}\label{equ:Newton}
\begin{pmatrix}
\hH^\tau & (G^{\tau})^T\\
G^\tau & \0
\end{pmatrix}\begin{pmatrix}
\Delta \bz^\tau\\
\Delta \blambda^\tau
\end{pmatrix} = - \begin{pmatrix}
\nabla_{\bz}\mL^\tau\\
\nabla_{\blambda}\mL^\tau
\end{pmatrix},
\end{equation}
where $\nabla_{\bz}\mL^\tau = \nabla_{\bz}\mL(\bz^\tau, \blambda^\tau)$, $\nabla_{\blambda}\mL^\tau = \nabla_{\blambda}\mL(\bz^\tau, \blambda^\tau)$, $G^\tau = \nabla_{\bz}^Tf(\bz^\tau)$, and $\hH^\tau$ is a modification of the Hessian $H^\tau = \nabla^2_{\bz}\mL(\bz^\tau, \blambda^\tau)$ that preserves the block diagonal structure of $H^\tau$. The goal of~the~modification is to let $\hH^\tau$ be positive definite in the null space $\{\bz: G^\tau\bz = \0\}$, if $H^\tau$ is not. For example, a simple structure-preserving modification is the Levenberg-style modification \cite{Dunn1989Efficient}: $\hH^\tau = H^\tau + \gamma I $ for a suitably large $\gamma>0$. Other Hessian modification methods are referred to in \cite{Nocedal2006Numerical}. The explicit formula of the system \eqref{equ:Newton} is displayed in Problem \eqref{pro:3}.

Given the Newton direction $(\Delta\bz^\tau, \Delta\blambda^\tau)$ from \eqref{equ:Newton}, the SQP iterate is updated as
\begin{equation}\label{equ:update}
\begin{pmatrix}
\bz^{\tau+1}\\
\blambda^{\tau+1}
\end{pmatrix} = \begin{pmatrix}
\bz^\tau\\
\blambda^\tau
\end{pmatrix} + \alpha_\tau\begin{pmatrix}
\Delta\bz^\tau\\
\Delta\blambda^\tau
\end{pmatrix},
\end{equation}
with the stepsize $\alpha_\tau$ being selected by passing a line search condition based on a merit function.~We employ the following differentiable exact augmented Lagrangian as the merit function
\begin{equation}\label{equ:aug:L}
\mL_{\eta}(\bz, \blambda) = \mL(\bz, \blambda) + \frac{\eta_1}{2}\|\nabla_{\blambda}\mL(\bz, \blambda)\|^2 + \frac{\eta_2}{2}\|\nabla_{\bz}\mL(\bz, \blambda)\|^2,
\end{equation}
where $\eta = (\eta_1, \eta_2)$ are the penalty parameters. The first penalty biases the feasibility error, while the second penalty biases the optimality error. The function \eqref{equ:aug:L} is called exact augmented Lagrangian, since one can show that the solution of the unconstrained problem $\min_{\bz, \blambda} \mL_{\eta}(\bz, \blambda)$ is also the~solution of \eqref{pro:1}, provided $\eta_1$ is large enough and $\eta_2$ is small enough \cite[Proposition 4.15]{Bertsekas1996Constrained}. We refer to \cite{Pillo1979New, Pillo1980method, Zavala2014Scalable, Na2022adaptive, Na2023Inequality} for more studies of \eqref{equ:aug:L} on constrained optimization problems. We should~mention that there are many penalty functions that can be used as a merit function; however, \eqref{equ:aug:L} is particularly suitable and important for our analysis. First, recalling the subproblem terminal~cost \eqref{equ:terminal:cost}, we need an approximation of the terminal dual variable, which means that the merit function has to endure a dual perturbation. Such a requirement rules out the penalty functions that depend only on the primal variables $\bz$. Second, for the SQP schemes, the differentiable merit functions~such as \eqref{equ:aug:L} can overcome the Maratos effect and locally accept a unit stepsize, which is critical to have a fast local convergence rate. In contrast, the non-smooth merit functions suffer from the Maratos effect and require non-trivial modifications (e.g., the second-order correction) of the SQP schemes to achieve a fast local rate \cite[Chapter 15.5]{Nocedal2006Numerical}. We will clearly see the benefits of the merit function \eqref{equ:aug:L} later. In particular, see the discussion below Theorem \ref{thm:4} and see Theorem \ref{thm:6} as well.

With \eqref{equ:aug:L}, the stepsize $\alpha_\tau$ is selected to make the following Armijo condition hold
\vskip-5pt\begin{equation}\label{equ:Armijo}
\mL_{\eta}^{\tau+1} \leq \mL_{\eta}^\tau + \beta\alpha_\tau\begin{pmatrix}
\nabla_{\bz}\mL_{\eta}^\tau\\
\nabla_{\blambda}\mL_{\eta}^\tau
\end{pmatrix}^T\begin{pmatrix}
\Delta \bz^\tau\\
\Delta \blambda^\tau
\end{pmatrix},
\end{equation}
where $\beta\in(0, 1/2)$ is a prespecified parameter, $\mL_{\eta}^{\tau} = \mL_{\eta}(\bz^{\tau}, \blambda^{\tau})$ (similar for $\mL_{\eta}^{\tau+1}$, $\nabla\mL_{\eta}^{\tau}$), and
\vskip-5pt\begin{equation}\label{equ:der:aug:L}
\begin{pmatrix}
\nabla_{\bz}\mL_{\eta}(\bz, \blambda)\\
\nabla_{\blambda}\mL_{\eta}(\bz, \blambda)
\end{pmatrix} = \begin{pmatrix}
I + \eta_2H(\bz, \blambda) & \eta_1 G^T(\bz) \\
\eta_2G(\bz) & I
\end{pmatrix}\begin{pmatrix}
\nabla_{\bz}\mL(\bz, \blambda)\\
\nabla_{\blambda}\mL(\bz, \blambda)
\end{pmatrix}.
\end{equation}

Among the steps in \eqref{equ:Newton}, \eqref{equ:update}, and \eqref{equ:Armijo}, solving the Newton system in \eqref{equ:Newton} is~often the most computationally expensive step. Thus, we apply the decomposition method, OTD, to solve \eqref{equ:Newton} approximately. We obtain an approximate direction $(\tDelta\bz^\tau, \tDelta\blambda^\tau)$, and then use it in the steps of \eqref{equ:update} and \eqref{equ:Armijo}. We introduce some extra notation. We let $H(\bz, \blambda) = \nabla_{\bz}^2\mL (\bz,\blambda) = \diag(H_0, \ldots, H_N)$ be the Hessian of $\mL(\bz, \blambda)$ with respect to $\bz$, where
\vskip-0.5cm\begin{equation}\label{equ:def:H}
H_k(\bz_k, \blambda_{k+1})  = \begin{pmatrix}
Q_k & S_k^T\\
S_k & R_k
\end{pmatrix} = \begin{pmatrix}
\nabla_{\bx_k}^2\mL & \nabla_{\bx_k\bu_k}\mL\\
\nabla_{\bu_k\bx_k}\mL & \nabla_{\bu_k}^2\mL
\end{pmatrix}, \quad \forall k\in[N-1], \quad H_N(\bz_N) = \nabla_{\bx_N}^2\mL.
\end{equation}
Note that $H$ does not depend on $\blambda_0$. We also let $A_k(\bz_k) = \nabla_{\bx_k}^Tf_k(\bz_k)$ and $B_k(\bz_k) = \nabla_{\bu_k}^Tf_k(\bz_k)$.~To simplify the notation, we suppress the evaluation points in $\{H_k, A_k, B_k\}$ and suppress the iteration index $\tau$ to refer to a general $\tau$-th iteration of \eqref{equ:Newton}. With the same block-diagonal structure as $H$, we have $\hH = \diag(\hH_0, \ldots, \hH_N)$, and use $\hQ_k, \hS_k, \hR_k$ to denote the components of $\hH_k$.

Solving \eqref{equ:Newton} is equivalent to solving the following linear-quadratic DP problem:
\begin{subequations}\label{pro:3}
\begin{align}
\min_{\bp, \bq}\;\;\; & \sum_{k=0}^{N-1} \cbr{\frac{1}{2}\begin{pmatrix}
\bp_k\\
\bq_k
\end{pmatrix}^T\begin{pmatrix}
\hQ_k & \hS_k^T\\
\hS_k & \hR_k
\end{pmatrix}\begin{pmatrix}
\bp_k\\
\bq_k
\end{pmatrix} + \begin{pmatrix}
\nabla_{\bx_k}\mL\\
\nabla_{\bu_k}\mL
\end{pmatrix}^T\begin{pmatrix}
\bp_k\\
\bq_k
\end{pmatrix} }+ \frac{1}{2}\bp_N^T\hQ_N\bp_N + \nabla_{\bx_N}^T\mL\bp_N, \label{pro:3a}\\
\text{s.t. }\; & \bp_{k+1}   = A_k\bp_k + B_k\bq_k - \nabla_{\blambda_{k+1}}\mL, \quad k \in[N-1], \label{pro:3b}\\
& \bp_{0} = -\nabla_{\blambda_{0}}\mL. \label{pro:3c}
\end{align}
\end{subequations}
We let $\bzeta = \bzeta_{0:N}$ be the dual variables and $\bw = \bw_{0:N} = (\bp, \bq)$ be the primal variables of Problem~\eqref{pro:3}. By the recursive constraints in \eqref{pro:3b}-\eqref{pro:3c}, we know that the Jacobian $G$ (in any iteration) has a ``staircase" structure: for all $k\in[N]$, the $(k, 2k-1)$-block matrix is an identity matrix. Thus, $G$ has full row rank. Suppose $\hH$ is positive definite in $\{\bz: G\bz = \0\}$, then the KKT~matrix~in~\eqref{equ:Newton}~is nonsingular; \eqref{pro:3} has a unique global solution $(\tw, \tzeta)$; and $(\Delta\bz, \Delta\blambda) = (\tw, \tzeta)$ \cite[Theorem~16.2]{Nocedal2006Numerical}. Note that the linear terms in \eqref{pro:3a} are given by the gradients of the Lagrangian. With this formulation, our dual solution $\tzeta$ of \eqref{pro:3} is the dual direction $\Delta\blambda$. When the linear terms in \eqref{pro:3a} are given by the gradients of the objective $g$, the dual solution $\tzeta$ would be the dual iterate $\blambda+\Delta\blambda$.

We now apply OTD on \eqref{pro:3}. By an analogy to \eqref{pro:2}, the subproblem $i$, $\LP_{\mu}^i(\bd_i)$, is defined as
\begin{subequations}\label{pro:4}
\begin{align}
\min_{\tbp_i, \tbq_i}\;\;\; & \sum_{k=m_1}^{m_2-1} \cbr{ \frac{1}{2}\begin{pmatrix}
\bp_k\\
\bq_k
\end{pmatrix}^T\begin{pmatrix}
\hQ_k & \hS_k^T\\
\hS_k & \hR_k
\end{pmatrix}\begin{pmatrix}
\bp_k\\
\bq_k
\end{pmatrix} + \begin{pmatrix}
\nabla_{\bx_k}\mL\\
\nabla_{\bu_k}\mL
\end{pmatrix}^T\begin{pmatrix}
\bp_k\\
\bq_k
\end{pmatrix} } \label{pro:4a}\\
&\quad  + \frac{1}{2}\begin{pmatrix}
\bp_{m_2}\\
\bbq_{m_2}
\end{pmatrix}^T\begin{pmatrix}
\hQ_{m_2} & \hS_{m_2}^T\\
\hS_{m_2} &  \hR_{m_2}
\end{pmatrix}\begin{pmatrix}
\bp_{m_2}\\
\bbq_{m_2}
\end{pmatrix} + \bp_{m_2}^T(\nabla_{\bx_{m_2}}\mL - A_{m_2}^T\bbzeta_{m_2+1})  + \frac{\mu}{2}\|\bp_{m_2} - \bbp_{m_2}\|^2, \nonumber\\
\text{s.t. }\; & \bp_{k+1}   = A_k\bp_k + B_k\bq_k - \nabla_{\blambda_{k+1}}\mL, \quad k \in[m_1, m_2), \label{pro:4b}\\
& \bp_{m_1} = \bbp_{m_1}. \label{pro:4c}
\end{align}
\end{subequations}
With slight abuse of notation, $\bd_i = \bd_{i, 1:4} = (\bbp_{m_1}; \bbp_{m_2}; \bbq_{m_2}; \bbzeta_{m_2+1})$ are the boundary variables; $\tbw_i = (\tbp_i, \tbq_i)$, $\tbzeta_i = \bzeta_{m_1:m_2}$ are the primal, dual variables of $\LP_{\mu}^i(\bd_i)$; and $(\tbw_i^{\star}(\bd_i), \tbzetai(\bd_i))$ is the solution.

\vskip 4pt
\noindent\textbf{FOTD scheme.} We set the stage to present the FOTD procedure. FOTD consists of three steps: given the $\tau$-th iterate $(\bz^\tau, \blambda^\tau)$, 

\noindent\hskip10pt \underline{Step 1}: Compute Hessian $\hH^\tau$, Jacobian $G^\tau$, and KKT residual vector $\nabla\mL^\tau$.

\noindent\hskip10pt\underline{Step 2}: Solve subproblems $\{\LP_{\mu}^i(\bd_i^\tau)\}_{i}$ with $\bd_i^\tau = (\0; \0; \0;\0)$ and obtain solutions~$\{(\tbw_i^{\star}(\bd_i^\tau), \tbzetai(\bd_i^\tau))\}_i$. Then, $(\tDelta\bz^\tau, \tDelta\blambda^\tau) = \mC(\{(\tbw_i^{\star}(\bd_i^\tau), \tbzetai(\bd_i^\tau))\}_i)$.

\noindent\hskip10pt\underline{Step 3}: Update the iterate by \eqref{equ:update} with $\alpha_\tau$ being selected by a line search based on \eqref{equ:Armijo}.

\vskip2pt
We summarize FOTD in Algorithm \ref{alg:FOTD} and present several remarks. 

\begin{remark}[necessity of the quadratic penalty]\label{rem:1}
Even if the full problem \eqref{pro:3} has a unique global solution, this does not necessarily imply that the subproblem \eqref{pro:4} has a unique~solution. Consider the following example. Let $N = 2$, $n_x =n_u = 1$, $\hQ_0 = \hR_0 = 1$, $\hQ_1 = -\hR_1 = -2$, $\hQ_2=3$, $\hS_0 = \hS_1 = 0$, $A_0 = A_1=B_0 = B_1=1$, and all linear terms in \eqref{pro:3a} are zero. Then, the full problem: $\min_{\bp, \bq}$ $p_0^2+q_0^2 - 2p_1^2+2q_1^2 + 3p_2^2$, s.t. $p_0 = 0$, $p_1 = p_0 + q_0$, $p_2 = p_1 + q_1$, has a unique global solution $(\tp, \tq) = (\0, \0)$. However, when we truncate onto $[0, 1]$, the subproblem has a quadratic~objective with the square matrix $\diag(1,1,-2+\mu)$ and constraints $p_0 = 0$,~$p_1 = p_0+q_0$. Thus, by plugging the constraints into the objective, we can easily obtain that, when $\mu<1$, the subproblem is unbounded below; when $\mu =1$, the subproblem has infinitely many solutions; when $\mu>1$, the subproblem has a unique global solution. Thus, to ensure that the subproblem has a unique global solution, a quadratic penalty with large enough $\mu$ on the terminal stage is necessary (cf. \eqref{pro:4a}).
\end{remark}

\begin{algorithm}[!tp]
\caption{A Fast Overlapping Temporal Decomposition Procedure}
\label{alg:FOTD}
\begin{algorithmic}[1]
\State \textbf{Input:} initial iterate $(\bz^0, \blambda^0)$ with $\bx_0^0 = \bbx_0$, scalars $\mu, \eta_1, \eta_2>0$, $\beta\in(0,1/2)$;
\For{$\tau = 0,1,2,\ldots$}
\State Compute $\hH^\tau$, $G^\tau$, and $\nabla\mL^\tau$;
\For{$ i = 0, 1, \ldots, M-1$ \textbf{(in parallel)}}
\State Let $\bd_i^\tau = (\0; \0; \0; \0)$ (note that $\bd_{M-1}^\tau = \0$);
\State Solve $\LP^i_\mu(\bd_i^\tau)$ to \textbf{optimality} and obtain the solution $(\tbwi(\bd_i^\tau), \tbzetai(\bd_i^\tau))$;
\EndFor
\State Let $(\tDelta\bz^{\tau}, \tDelta\blambda^{\tau}) = \mC(\{(\tbwi(\bd_i^\tau), \tbzetai(\bd_i^\tau))\}_i)$;
\State Select $\alpha_\tau$ by a line search based on \eqref{equ:Armijo} and update the iterate by \eqref{equ:update};
\EndFor
\end{algorithmic}
\end{algorithm}

\begin{remark}\label{rem:2}
We can express the optimality conditions of $\LP_{\mu}^i(\bd_i)$ using a Newton system like \eqref{equ:Newton}. We observe that the KKT matrix depends on $\{A_k, B_k, \hH_k\}_{k=m_1}^{m_2-1}\cup \{\hQ_{m_2}+\mu I\}$, but not on $\bd_i$. Thus, the uniqueness of the solution of $\LP_{\mu}^i(\bd_i)$ is independent from $\bd_i$. By Theorem \ref{thm:1}(i), if $\bd_i^\star = (\Delta\bx_{m_1}; \Delta\bz_{m_2}; \Delta\blambda_{m_2+1})$, then $(\tbw_i^{\star}(\bd_i^\star), \tbzetai(\bd_i^\star)) = (\Delta\bx_{m_1:m_2}, \Delta\bu_{m_1:m_2-1}, \Delta\blambda_{m_1:m_2})$.	However, since there is no good guess for a search direction, we always set $\bd_i = (\0;\0;\0;\0)$ in Algorithm~\ref{alg:FOTD}~(Line~5).
\end{remark}

\begin{remark}\label{rem:3}
Comparing Line 6 of Algorithm \ref{alg:FOTD} with Line 5 of Algorithm \ref{alg:OTD}, FOTD also solves subproblems $\{\LP_{\mu}^i\}_i$ to optimality. However, $\LP_{\mu}^i$ is a linear-quadratic DP, which can be solved~efficiently while solving $\P_{\mu}^i$ is more expensive. The problem $\LP_{\mu}^i$ can be regarded as a linearization of the problem $\P_{\mu}^i$; thus solving $\LP_{\mu}^i$ corresponds to computing one Newton step of $\P_{\mu}^i$.
\end{remark}

\vspace{-0.3cm}

\begin{remark}\label{rem:4}
Two SQP components, Hessian modification and line search, are less addressed in this paper. Performing them efficiently in a parallel environment is of course desirable from a practical aspect, while we put the implementation details of this regard aside in the paper~but~briefly discuss in this remark. Our paper is mainly concerned with solving the Newton system \eqref{equ:Newton}, which always requires more computations.

For line search, we note that the augmented Lagrangian $\mL_\eta$ and its gradient $\nabla \mL_\eta$ are separable in stages. Thus, we can easily evaluate them in parallel, where each processor computes a short~horizon e.g., $[n_i, n_{i+1})$, and a coordinator is needed to sum up the results of all processors to check \eqref{equ:Armijo}.

For Hessian modification, a desirable modification $\hH$ should be close to the Hessian $H$ whenever $H$ satisfies the second-order sufficient condition (SOSC, i.e., \eqref{equ:reduced} with $\hH$ being replaced by $H$). If $H$ does not satisfy the condition, we regularize $H$, e.g. $\hH = H+(\gamma_{RH} + \|H\|) I$, to let $\hH$ satisfy the condition, which ensures that the full Newton system \eqref{equ:Newton} has a unique solution. See Assumptions \ref{ass:1} and \ref{ass:5}. Thus, we have to check the positive definiteness of the reduced Hessian $Z^THZ$,~where the columns of $Z$ span the null space of the Jacobian $G = \nabla_{\bz}^Tf$. It is equivalent to testing for~positive definiteness of $H + c\cdot G^TG$ with a scalar $c$ sufficiently large \cite{Nocedal2006Numerical}. Note that $H + c\cdot G^TG$ is a block-tridiagonal matrix; thus, we can apply the parallel Cholesky factorization to check its definiteness~in a parallel environment \cite{Wright1990Solution, Wright1991Parallel, Cao2002Parallel}. As analyzed in \cite[Section 2.2]{Cao2002Parallel}, the total flops of the parallel Cholesky of $M$ processors are less than $7N(n_x+n_u)^3$, so the average flop of a single processor is~in order of $O(N(n_x+n_u)^3/M)$, which does not grow with $N$ if the ratio $N/M$ is a constant. 

\end{remark}

\vskip-0.3cm

\begin{remark}[extensions to inequality constraints]\label{rem:5}
The FOTD procedure can be~generalized to inequality-constrained problems. In particular, with inequality constraints $h_k(\bx_k, \bu_k)\leq \0$, $k\in[N-1]$, the full SQP problem \eqref{pro:3} will additionally have the linearized inequality constraints $C_k\bp_k + D_k\bq_k\leq \0$, $k\in[N-1]$, where $C_k = \nabla_{\bx_k}^Th_k$ and $D_k = \nabla_{\bu_k}^Th_k$ are the Jacobian matrices of $h_k$. The FOTD scheme still applies OTD on the full problem \eqref{pro:3} so that the subproblems \eqref{pro:4} are inequality-constrained quadratic programs (IQPs). Several methods with warm-start strategies can be applied on IQPs, such as the interior-point methods and active-set methods. Within the FOTD scheme, the exact augmented Lagrangian function \eqref{equ:aug:L} should also be adapted to the one~that~can handle inequality constraints. See \cite{Pillo2002Augmented} for a particular choice. We should mention that the alternative designs with the same flavor of FOTD are also available for dealing with inequality constraints. For example, we can exploit an active-set SQP scheme, where in each iteration we only consider the inequality constraints that are in the identified active set and regard them as equalities. Then, the full problem \eqref{pro:3} and the subproblem \eqref{pro:4} are still equality-constrained QPs (EQPs), which are easier to solve than IQPs. Since the design and analysis of inequality constraints are quite~involved, we~leave the above extensions of FOTD to the future, and mainly focus on connecting FOTD with the Schwarz scheme on equality-constrained problems in this paper.
\end{remark}

\vspace{-0.3cm}

\begin{remark}[\hskip-1pt relationships to other schemes]\label{rem:6} 

\hskip-0.35cm Besides the Schwarz scheme, FOTD~is~related to several other methods for optimal control problems. We introduce the connections in this remark. We emphasize that the two critical components of FOTD, overlapping temporal decomposition and augmented Lagrangian merit function, have not been investigated in the following methods. Also, the convergence of FOTD highly relies on the sensitivity analysis of NLDPs in \cite{Na2020Exponential, Na2022Convergence}, which~differs from the following methods.

\noindent\textbf{(a) Direct multiple shooting methods.} The direct multiple shooting methods decompose the full horizon into multiple \textit{exclusive} short intervals, construct the subproblems associated with short intervals, and compose the solution trajectories of subproblems under certain matching~conditions. See \cite{Bock1984Multiple, Bock2000Direct} for more details. The multiple shooting methods are often employed for real-time nonlinear

\noindent model predictive control \cite{Schaefer2007Fast, Diehl2009Efficient, Kirches2012Efficient}, where the subproblems are solved sequentially and an approximate but fast control feedback is desired for each subproblem. Many efficient solvers that exploit the problem sparsity structures can be used within the methods, such as qpDUNES \citep{Frasch2015parallel} and HPMPC \citep{Frison2014High}. Compared to the aforementioned literature, FOTD solves \eqref{pro:1} in an offline~fashion (i.e., the short intervals do not vary with the sampling time) and applies an \textit{overlapping} decomposition on the SQP problem instead of on \eqref{pro:1}. More importantly, as analyzed in \cref{sec:5},~FOTD~does~not~use any matching conditions to compose the subproblem solutions, but requires a large~overlap~size $b$. Such a design is inspired by the sensitivity analysis of NLDPs. The exclusive stages are far from~the boundaries where the system perturbations occur, and the sensitivity results state that~the effects of boundary perturbations decay exponentially fast away from the boundaries. Thus, the solutions at the exclusive middle stages are still accurate enough even without matching conditions. 

\vskip2pt
\noindent\textbf{(b) Alternating direction method of multipliers (ADMM).} ADMM is another popular method for optimal control problems \cite{ODonoghue2013splitting, Boyd2010Distributed}, where one introduces a set of consensus constraints~to split the full problem into multiple subproblems, and solves the subproblems in each iteration in parallel. We note that our terminal cost \eqref{equ:terminal:cost} is conceptually similar to the quadratic proximal~term in \cite[(3)]{ODonoghue2013splitting}. However, FOTD and ADMM have a few key differences. First, ADMM is designed~based on the augmented Lagrangian method, while FOTD is designed based on SQP. Thus, their primal-dual updating schemes are quite different. Second, even if we use an augmented Lagrangian merit function in FOTD, the function is different from the one in ADMM \cite{Boyd2010Distributed} in that it has a quadratic penalty on the optimality error (i.e., the last term in \eqref{equ:aug:L}). Also, we use the augmented Lagrangian for the stepsize selection, not for the direction computation, while ADMM combines the two steps together. Third, similar to the multiple shooting methods, ADMM decomposes the full problem~by introducing extra consensus constraints without overlaps.~Thus,~ADMM~is~sensitive to the~parameter $\mu$, while the overlaps in FOTD largely~suppress the boundary \mbox{perturbations}~brought~by~different choices of $\mu$ (and imprecise boundary variables $\bd_i$). See \cite[Figure 6]{Na2022Convergence} for empirical evidence.

\vskip2pt
\noindent\textbf{(c) Iterative linear-quadratic regulator (ILQR).} Another method for optimal control problems is ILQR \cite{Li2007Iterative}. In each step, the search direction of control variables is solved from a~QP~that~is defined by quadratic approximation of objective \eqref{pro:1a} with linear approximations of constraints \eqref{pro:1b}. A parallel implementation of ILQR was designed in \cite{Laine2019Parallelizing}. In addition to the difference that FOTD employs an overlapping decomposition for parallelism, the QP in \eqref{pro:3} also differs from the one in ILQR in that its objective is a quadratic approximation of the Lagrangian. Further, FOTD updates the state, control, and dual variables with the same stepsize selected by the line search on the augmented Lagrangian \eqref{equ:aug:L}, while ILQR updates the control variables with the line search on the objective \eqref{pro:1a}, and the state variables are computed by applying \eqref{pro:1b} in a forward pass.

\end{remark}

\vspace{-0.3cm}

\begin{remark}[QP solvers for the subproblems] The subproblems \eqref{pro:4} of FOTD are~QPs. Many efficient QP or linear system solvers with wart-start strategies can be adopted for FOTD.~For example, the conjugate gradient (CG) and minimal residual (MINRES) are popular methods for solving the Newton systems, and the popular solvers~qpDUNES~\cite{Frasch2015parallel},~HPMPC~\cite{Frison2014High},~and~FORCES~\cite{Domahidi2012Efficient} that exploit the sparsity structures of QPs can also be applied on \eqref{pro:4}. We note that the above QP solvers generally require a positive definite Hessian matrix which \eqref{pro:4} does not have. Fortunately, \cite{Verschueren2017Sparsity} proposed a convexification procedure to resolve this issue. 
\end{remark}

\vspace{-0.2cm}

\section{Error of the approximate search direction}\label{sec:4}

We study the difference between~$(\tDelta\bz^\tau, \tDelta\blambda^\tau)$ and $(\Delta\bz^\tau, \Delta\blambda^\tau)$, where the former is the OTD approximation (Step 2) and the latter is the exact Newton direction of \eqref{equ:Newton}. We state the assumptions that are required for the analysis.

\begin{assumption}[lower bound on the reduced Hessian]\label{ass:1}
For any iteration $\tau\geq0$, we let $Z^{\tau}$ be a matrix whose columns are orthonormal vectors and span the null space $\{\bw: G^\tau\bw = \0\}$.~We assume that there exists a uniform constant $\gamma_{RH}>0$ that is independent of $\tau$, such that 
\begin{equation}\label{equ:reduced}
(Z^\tau)^T\hH^\tau Z^\tau \succeq \gamma_{RH} \cdot I.
\end{equation}
The matrix on the left hand side is called the \textit{reduced Hessian}.
\end{assumption}

\begin{assumption}[controllability]\label{ass:2}
For any stage $k\in[N]$ and an integer $t\in[N-k]$, the~controllability matrix is defined as
\begin{equation*}
\Xi_{k, t}(\bz_{k:k+t-1}) = \begin{pmatrix}
B_{k+t-1} & A_{k+t-1}B_{k+t-2} & \cdots & (\prod_{l=1}^{t-1}A_{k+l})B_k
\end{pmatrix}\in \mR^{n_x\times tn_u}.
\end{equation*}
For any iteration $\tau\geq 0$, we assume that there exist a uniform constant $\gamma_{C}>0$ and an integer $t>0$ that are independent of $\tau$, such that for any $ k\in[N-t]$, there exists $t_k^\tau\in[1, t]$ so that $\Xi_{k, t_k^\tau}^\tau(\Xi_{k, t_k^\tau}^{\tau})^T\succeq \gamma_{C} I$, where $\Xi_{k, t_k^\tau}^\tau = \Xi_{k, t_k^\tau}(\bz_{k:k+t_k^\tau-1}^\tau)$.
\end{assumption}

\begin{assumption}[upper boundedness]\label{ass:3}
For any iteration $\tau\geq 0$, there exists a uniform constant $\Upsilon_{upper}$ that is independent of $\tau$, such that for any $k$, $\max\{\|\hH^\tau_k\|,\;  \|A_k^\tau\|,\; \|B_k^\tau\|\} \leq \Upsilon_{upper}$ .
\end{assumption}

Since $G^\tau = \nabla_{\bz}^Tf(\bz^\tau)$ has full row rank, Assumption \ref{ass:1} ensures that Problem \eqref{pro:3} has a unique global solution and the KKT matrix in \eqref{equ:Newton} is invertible \cite[Lemma 16.1]{Nocedal2006Numerical}. This assumption~is~standard in the SQP literature \cite{Bertsekas1996Constrained} and weaker than the linear-quadratic convex DP setup studied~in~\cite{Xu2018Exponentially}. Assumption \ref{ass:2} is specifically used for DP problems \cite{Keerthi1988Optimal, Xu2018Exponentially, Xu2019Exponentially, Na2020Exponential, Na2023Superconvergence, Na2022Convergence}. It ensures that the linearized dynamical system $\bp_{k+1} = A_k\bp_k+B_k\bq_k$ is controllable in at most $t$ steps (in each iteration). That is, given any initial state $\bp_k$ and target state $\bp_{k+t_k}$, we can always evolve from $\bp_k$ to $\bp_{k+t_k}$~by specifying a suitable control sequence $\{\bq_j\}_{j=k}^{k+t_k-1}$. Assumption \ref{ass:3} is standard in both SQP and DP literature \cite{Bertsekas1996Constrained, Xu2018Exponentially, Na2020Exponential}. In \cref{sec:5}, we will impose a compactness condition on the SQP iterates, which naturally implies the upper boundedness of the Hessian and Jacobian matrices. We show in Lemma \ref{lem:0} that Assumptions \ref{ass:2} and \ref{ass:3} imply a uniform lower bound on $G^\tau(G^\tau)^T$.

\begin{lemma}\label{lem:0}
Suppose $\max\{\|A_k^\tau\|,\; \|B_k^\tau\|\}\leq \Upsilon_{upper}$, then Assumption \ref{ass:2} implies that~$G^\tau(G^\tau)^T\succeq \gamma_{G} I$ where
\begin{equation}\label{equ:gamma:G}
\gamma_{G} = \gamma_{G}(\gamma_{C}, t, \Upsilon_{upper}) \coloneqq \rbr{\frac{\gamma_{C}}{\gamma_{C} + \frac{\Upsilon_{upper}^{t+1}}{\Upsilon_{upper}-1}} }^2\cdot \frac{\min\{1, \gamma_{C}\}}{(1+\Upsilon_{upper})^{2t}}.
\end{equation}
\end{lemma}

\begin{proof}
See Appendix \ref{pf:lem:0}.
\end{proof}

The next result shows that $\LP_{\mu}^i(\bd_i)$ has a unique solution if $\mu$ is large enough.

\begin{lemma}\label{lem:1}
Suppose Assumptions \ref{ass:1}-\ref{ass:3} hold for Problem \eqref{pro:3}. Let
\begin{equation}\label{equ:mu:cond}
\bar\mu = \bar\mu(\gamma_{C}, t, \Upsilon_{upper}) \coloneqq \frac{32\Upsilon_{upper}^{4t+1}}{\gamma_C}. 
\end{equation}
If $\mu\geq \bar\mu$, then the reduced Hessian of Problem $\LP_{\mu}^i(\bd_i)$ in \eqref{pro:4}, defined similarly to \eqref{equ:reduced}, is lower bounded by $\gamma_{RH}I$ for any iteration $\tau\geq 0$ and any boundary variables $\bd_i$. This implies that $\LP_{\mu}^i(\bd_i)$ has a unique global solution.
\end{lemma}

\begin{proof}
See Appendix \ref{pf:lem:1}.
\end{proof}

Lemma \ref{lem:1} shows that $\{\LP_{\mu}^i(\bd_i)\}_{i=0}^{M-1}$ are solvable for any iteration $\tau\geq 0$ if $\mu\geq \bar\mu$, where $\bar\mu$ is independent of the iteration index $\tau$ and the subproblem index $i$. An immediate consequence is that Assumptions \ref{ass:1}-\ref{ass:3} hold for the subproblems $\LP_{\mu}^i$ as well.

\begin{corollary}\label{cor:1}
Suppose Assumptions \ref{ass:1}-\ref{ass:3} hold for Problem \eqref{pro:3} and $\mu\geq \bar{\mu}$ with $\bar{\mu}$ given by \eqref{equ:mu:cond}. Then, the three conditions: lower bound on the reduced Hessian, controllability, and upper boundedness, hold for the subproblems $\{\LP_{\mu}^i(\bd_i)\}_{i\in[M-1]}$ with any $\bd_i$. Furthermore, the condition constants are independent of $i$ and $\tau$. Specifically, the subproblem $\LP_{\mu}^i(\bd_i)$ for $i\in[M-1]$ satisfies

\noindent(i) Assumption \ref{ass:1}: the reduced Hessian of $\LP_{\mu}^i(\bd_i)$ is lower bounded by $\gamma_{RH} I$.

\noindent(ii) Assumption \ref{ass:2}: the controllability condition of $\LP_{\mu}^i(\bd_i)$ is satisfied with the same constants $(\gamma_{C}, t)$.

\noindent(iii) Assumption \ref{ass:3}: the boundedness condition of $\LP_{\mu}^i(\bd_i)$ is satisfied with constant $\Upsilon_{upper} + \mu$.
	
\end{corollary}

\begin{proof}
See Appendix \ref{pf:cor:1}.
\end{proof}

We are now able to control the error $(\tDelta\bz^\tau - \Delta\bz^\tau, \tDelta\blambda^\tau - \Delta\blambda^\tau)$. To ease the notation, we suppress the iteration index $\tau$. The study of the error $(\tDelta\bz-\Delta\bz, \tDelta\blambda - \Delta\blambda)$ relies on the primal-dual sensitivity analysis of NLDPs in \cite{Na2020Exponential, Na2022Convergence}, which requires Assumptions \ref{ass:1}-\ref{ass:3}. We apply the sensitivity results on the subproblems; thus Corollary \ref{cor:1} is critical. It shows that Assumptions \ref{ass:1}-\ref{ass:3} hold for the subproblems as long as they hold for the full problem. In principle, the sensitivity results suggest that, if we perturb the objective and constraints on one stage, then the \mbox{perturbation} effects on the optimal solution decay exponentially fast away from that perturbed stage. In the OTD setup, the perturbations occur at the two boundaries of the extended intervals. Thus, the composition~that uses only the exclusive part $[n_i, n_{i+1})$ of the solution preserves all accurate variables.

Let us first bound $(\tbwi(\bd_i), \tbzetai(\bd_i)) - (\tbwi(\bd_i'), \tbzetai(\bd_i'))$ for any $\bd_i, \bd_i'$. Recall that $(\tbwi(\bd_i), \tbzetai(\bd_i))$ denotes the unique global solution of $\LP_{\mu}^i(\bd_i)$. In the following presentation, we use $C$ and $C'$ to denote universal constants that are independent of the iteration index $\tau$ and the subproblem~index~$i$.

\begin{theorem}\label{thm:2}
Let Assumptions \ref{ass:1}-\ref{ass:3} hold for Problem \eqref{pro:3} and $\mu\geq \bar{\mu}$ with $\bar{\mu}$ given~by~\eqref{equ:mu:cond}. There exist constants $C'>0$ and $\rho\in(0, 1)$ independent of $i$ and $\tau$, such that for~all~$k\in [m_1, m_2]$,
\begin{equation*}
\max\{\|\tbw_{i, k}^\star(\bd_i) - \tbw_{i, k}^\star(\bd_i')\| , \; \|\tbzeta_{i, k}^\star(\bd_i) - \tbzeta_{i, k}^\star(\bd_i')\|\} 
\leq C'(\rho^{k-m_1}\|\bd_{i, 1} - \bd_{i, 1}'\| + \rho^{m_2-k}\|\bd_{i,2:4} - \bd_{i,2:4}'\|)
\end{equation*}
for any two boundary variables $\bd_i$ and $\bd_i'$.
\end{theorem}

\begin{proof}
See Appendix \ref{pf:thm:2}.
\end{proof}

The next theorem characterizes the approximation error of the Newton direction.

\begin{theorem}[error of approximate direction]\label{thm:3}
Let Assumptions \ref{ass:1}-\ref{ass:3} hold for Problem \eqref{pro:3} and $\mu\geq \bar{\mu}$ with $\bar{\mu}$ given by \eqref{equ:mu:cond}. There exist constants $C>0$ and $\rho\in(0, 1)$ independent~of $\tau$, such that
\begin{equation}\label{equ:error}
\|(\tDelta\bz^\tau - \Delta\bz^\tau, \tDelta\blambda^\tau - \Delta\blambda)\| \leq C\rho^b\|(\Delta\bz^\tau, \Delta\blambda^\tau)\|.
\end{equation}	
	
\end{theorem}

\begin{proof}
See Appendix \ref{pf:thm:3}.
\end{proof}

Theorem \ref{thm:3} suggests that the error of the approximate direction decays exponentially fast in terms of the overlap size $b$. We can naturally expect that $(\tDelta\blambda^\tau, \tDelta\blambda^\tau)$ contains enough~information to decrease the merit function in each iteration, provided $b$ is large. We study how the~inexactness of the direction affects SQP in the next section.

\section{Global convergence of FOTD}\label{sec:5}

To enable general distributed methods for solving \eqref{equ:Newton}, we suppose that a decomposition method outputs a direction $(\tDelta\bz^\tau, \tDelta\blambda^\tau)$ satisfying
\begin{equation}\label{equ:error:cond}
\|(\tDelta\bz^\tau - \Delta\bz^\tau, \tDelta\blambda^\tau - \Delta\blambda^\tau)\| \leq \delta\cdot \|(\Delta\bz^\tau, \Delta\blambda^\tau)\|, \quad \text{ for } \delta \in (0, 1).
\end{equation}
Theorem \ref{thm:3} shows that FOTD specializes \eqref{equ:error:cond} with $\delta = C\rho^b$. In this section, we study the convergence of SQP with the merit function \eqref{equ:aug:L} and the direction $(\tDelta\bz^\tau, \tDelta\blambda^\tau)$. We require a compactness assumption to strengthen Assumption \ref{ass:3} to hold in a compact set. We recall that the SQP iterates are generated by \eqref{equ:update} and \eqref{equ:Armijo}.

\begin{assumption}[compactness]\label{ass:4}
There exists a compact set $\mZ\times \Lambda$, where $\mZ = \mZ_0\times\cdots\times\mZ_N$ and $\Lambda = \Lambda_0\times\cdots\times\Lambda_N$, such that $(\bz^\tau + \alpha\tDelta\bz^\tau, \blambda^\tau+ \alpha\tDelta\blambda^\tau) \in \mZ\times \Lambda$, $\forall \tau\geq 0, \alpha\in[0, 1]$; and we assume that $\{g_k, f_k\}$ are thrice continuously differentiable and
\begin{equation}\label{equ:upper:third:new}
\max\{\sup_{\mZ_k\times\Lambda_{k+1}} \|H_k(\bz_k, \blambda_{k+1})\|,\; \sup_{\mZ_k}\|A_k(\bz_k)\|,\;   \sup_{\mZ_k}\|B_k(\bz_k) \|\}  \leq \Upsilon_{upper},
\end{equation}
for a constant $\Upsilon_{upper}>0$ independent of $k$. 
\end{assumption}

We use the same constant $\Upsilon_{upper}$ as in Assumption \ref{ass:3} for notational simplicity. The compactness assumption is standard in the SQP literature \cite[Proposition 4.15]{Bertsekas1996Constrained}. One equivalent assumption is~to assume that the sublevel set $\{(\bz, \blambda): \mL_{\eta}(\bz, \blambda) \leq \mL_{\eta}^0\}$ is contained in a compact set. Furthermore, assuming the existence of the third derivatives on $\{g_k,f_k\}$ is common for the augmented Lagrangian merit function \eqref{equ:aug:L}, since $\nabla^2\mL_{\eta}$ requires $\nabla^3g_k$ and $\nabla^3f_k$ \cite{Bertsekas1996Constrained, Zavala2014Scalable, Na2022adaptive, Na2023Inequality}. Note that~the~third~derivatives are only required in the analysis, and not computed or used in the algorithm.

The following lemma is an immediate consequence of Assumption \ref{ass:4}.

\begin{lemma}\label{lem:2}
	
Under Assumption \ref{ass:4}, there exists a constant $\Upsilon_{HG}>0$ such that
\begin{equation}\label{equ:upper}
\max\{\sup_{\mZ\times \Lambda}\|H(\bz, \blambda)\|, \sup_{\mZ}\|G(\bz)\|\}  \leq \Upsilon_{HG}.
\end{equation}
Further, if Assumptions \ref{ass:1}-\ref{ass:3} hold as well, there exists a constant $\Upsilon_{KKT}>0$ that is independent of $\tau$, such that
\begin{equation*}
\nbr{\begin{pmatrix}
\hH^\tau & (G^\tau)^T\\
G^\tau & \0
\end{pmatrix}^{-1} } \leq \Upsilon_{KKT}, \quad \forall\tau\geq 0.
\end{equation*}
\end{lemma}

\begin{proof}
See Appendix \ref{pf:lem:2}
\end{proof}

The bound \eqref{equ:upper} suggests that $\|G^\tau\|\leq\Upsilon_{HG}$, $\forall\tau\geq 0$ (since $\bz^\tau\in \mZ$). Since \eqref{equ:upper:third:new} implies $\max\{\|A_k^\tau\|,\;\\ \|B_k^\tau\|\}\leq \Upsilon_{upper}$ for any $k\in[N-1]$, under Assumption \ref{ass:4}, we only need $\|\hH^\tau\|\leq~\Upsilon_{upper}$~from~Assumption \ref{ass:3}. To simplify the presentation, we define $\Upsilon = \max\{\Upsilon_{upper}, \Upsilon_{KKT}, \Upsilon_{HG}\}$.

We now show that $(\tDelta\bz^\tau, \tDelta\blambda^\tau)$ is a descent direction of $\mL_{\eta}^\tau$ provided $\eta_1$ is sufficiently large~and $\eta_2, \delta$ are sufficiently small.

\begin{theorem}\label{thm:4}
Suppose Assumptions \ref{ass:1}, \ref{ass:2}, \ref{ass:3}, \ref{ass:4} hold for the SQP iterates with a search direction $(\tDelta\bz^\tau, \tDelta\blambda^\tau)$ satisfying \eqref{equ:error:cond}. If
\begin{equation}\label{cond:1}
\eta_1\geq \frac{17}{\eta_2\gamma_G}, \quad \eta_2\leq \frac{\gamma_{RH}}{12\Upsilon^2}, \quad \delta \leq \frac{\eta_2\gamma_{G}}{9\eta_1\Upsilon^2},
\end{equation}
where $\gamma_{G}$ is defined in \eqref{equ:gamma:G}, then
\begin{equation}\label{equ:while:loop}
\begin{pmatrix}
\nabla_{\bz}\mL_{\eta}^\tau\\
\nabla_{\blambda}\mL_{\eta}^\tau
\end{pmatrix}^T\begin{pmatrix}
\tDelta\bz^\tau\\
\tDelta\blambda^\tau
\end{pmatrix} \leq -\frac{\eta_2}{2}\nbr{\begin{pmatrix}
\nabla_{\bz}\mL^\tau\\
\nabla_{\blambda}\mL^\tau
\end{pmatrix}}^2.
\end{equation}	
\end{theorem}

\begin{proof}
See Appendix \ref{pf:thm:4}.
\end{proof}

By the results of Theorem \ref{thm:4} and the mean value theorem, we immediately know that, under~the presented conditions, the stepsize $\alpha_\tau$ to satisfy the Armijo~condition \eqref{equ:Armijo} can be found by the~backtracking line search, and the SQP iterates with the direction $(\tDelta\bz^\tau, \tDelta\blambda^\tau)$ make successful progress towards a stationary point.

Theorem \ref{thm:4} also shows the importance of using the exact augmented Lagrangian merit function \eqref{equ:aug:L}. In particular, as proved in \eqref{pequ:15}, we rely on a critical property: for suitably chosen $\eta = (\eta_1,\eta_2)$, there exists a small constant $\kappa(\eta_2)>0$ (depending on $\eta_2$) such that
\begin{align}\label{equ:margin}
\begin{pmatrix}
\nabla_{\bz}\mL_{\eta}^\tau\\
\nabla_{\blambda}\mL_{\eta}^\tau
\end{pmatrix}^T\begin{pmatrix}
\Delta\bz^\tau\\
\Delta\blambda^\tau
\end{pmatrix}\leq - \kappa(\eta_2)\bigg(\underbrace{\nbr{\begin{pmatrix}
\nabla_{\bz}\mL^\tau\\
\nabla_{\blambda}\mL^\tau
\end{pmatrix}}^2}_{\text{ensure convergence}} +  \underbrace{\nbr{\begin{pmatrix}
\Delta\bz^\tau\\
\Delta\blambda^\tau
\end{pmatrix}}^2}_{\text{allow approximation}} \bigg).
\end{align}
The first term is needed for global convergence when we accumulate the descent over iterations.~The second term allows us to replace $(\Delta\bz^\tau, \Delta\blambda^\tau)$ by $(\tDelta\bz^\tau, \tDelta\blambda^\tau)$ because, as proved in \eqref{pequ:28},
\begin{equation*}
\nbr{\begin{pmatrix}
\nabla_{\bz}\mL_{\eta}^\tau\\
\nabla_{\blambda}\mL_{\eta}^\tau
\end{pmatrix}}\nbr{\begin{pmatrix}
\tDelta\bz^\tau - \Delta\bz^\tau\\
\tDelta\blambda^\tau - \Delta\blambda^\tau
\end{pmatrix}} \lesssim \delta \nbr{\begin{pmatrix}
\Delta\bz^\tau\\
\Delta\blambda^\tau
\end{pmatrix}}^2,
\end{equation*}
where $\lesssim$ means that the inequality holds up to a constant multiplier. Therefore, for small enough~$\delta$, the margin $\|(\Delta\bz^\tau, \Delta\blambda^\tau)\|$ in \eqref{equ:margin} allows for an approximation error. Our analysis rules out the~exact penalty merit functions that depend only on the primal variables $\bz$. Such a function $\mM(\bz)$ satisfies $\nabla_{\bz}^T\mM^\tau \Delta\bz^\tau \leq -\kappa \|\Delta\bz^\tau\|^2$ and $\|\nabla_{\bz}\mM^\tau\| \lesssim\|\Delta\bz^\tau\|$ (e.g., \cite{Fletcher1973exact, Glad1979multiplier, Schittkowski1982nonlinear, Pillo1994Exact, Bertsekas1996Constrained}). If we approximate $\Delta\bz^\tau$~by $\tDelta\bz^\tau$, Theorems \ref{thm:2}, \ref{thm:3} suggest that the approximation error $\|\tDelta\bz^\tau - \Delta\bz^\tau\|$ also depends on $\|\Delta\blambda^\tau\|$. Thus, the margin $\|\Delta\bz^\tau\|^2$ may not be enough to endure the approximation error. This reveals the benefits of using the primal-dual exact merit functions based on \eqref{equ:aug:L}.

We summarize the global convergence results in the next theorem.

\begin{theorem}[global convergence]\label{thm:5}
\hskip-0.35cm Suppose Assumptions \ref{ass:1}, \ref{ass:2}, \ref{ass:3}, \ref{ass:4} hold for the SQP iterates with $\{(\tDelta\bz^\tau, \tDelta\blambda^\tau)\}_\tau$ satisfying \eqref{equ:error:cond} and parameters $(\eta, \delta)$ satisfying \eqref{cond:1}, then $\|\nabla\mL^\tau\|\rightarrow 0$ as $\tau\rightarrow \infty$. Furthermore, for the FOTD iterates, if $\mu\geq \bar{\mu}$ with $\bar{\mu}$ given by \eqref{equ:mu:cond} as well as $b$ satisfies
\begin{equation}\label{equ:cond:b}
b\geq \frac{\log\rbr{9C\eta_1\Upsilon^2/(\eta_2\gamma_{G})}}{\log(1/\rho)}
\end{equation}
with $C>0$ and $\rho\in(0, 1)$ from Theorem \ref{thm:3}, then the FOTD iterates $\{(\bz^\tau, \blambda^\tau)\}_\tau$ generated by~Algorithm \ref{alg:FOTD} satisfy $\|\nabla\mL^\tau\|\rightarrow 0$ as $\tau\rightarrow \infty$.
\end{theorem}

\begin{proof}
See Appendix \ref{pf:thm:5}.
\end{proof}

We have shown the global convergence of FOTD. Theorem \ref{thm:5} complements the local convergence result of the Schwarz scheme \cite{Na2022Convergence} and answers Q1 raised in \cref{sec:1}. The next section studies the local convergence of FOTD. We show that FOTD and the Schwarz have the same local behavior.

To end this section, we discuss an adaptivity extension of our scheme.

\begin{remark}[adaptivity on penalty parameters]
Given the conditions \eqref{cond:1} on $(\eta, \delta = C\rho^b)$, it is possible to design a scheme that adaptively selects the suitable penalty parameters $\eta$ and overlap size $b$. Since the upper/lower bound constants in \eqref{cond:1} do not depend on $\tau$, we can~design a while loop to achieve this goal. In particular, we let $\nu>1$ be fixed. While \eqref{equ:while:loop} does not~hold,~we~let
\begin{equation*}
\eta_2 \leftarrow \eta_2/\nu, \quad \eta_1\leftarrow \eta_1\nu^2, \quad \delta\leftarrow \delta/\nu^4.
\end{equation*}
Then, we know $\eta_1\eta_2$ increases by a factor of $\nu$, and $\eta_2$ and $\delta\eta_1/\eta_2$ decrease by a factor of $1/\nu$. Thus, for sufficiently large $\tau$, all parameters are stabilized. We also note that $\delta\leftarrow \delta/\nu^4$ is equivalent to letting $b\leftarrow b + 4\log\nu/\log(1/\rho)$, where only $\rho$ has to be tuned manually. 
\end{remark}

\section{Local convergence of FOTD}\label{sec:6}

\hskip-0.1cm From now on, we suppose the FOTD iterates satisfy~$(\bz^\tau, \blambda^\tau)$ $\rightarrow (\tz, \tlambda)$ as $\tau\rightarrow \infty$. For two positive sequences $\{a_\tau\}$ and $\{b_\tau\}$, $a_\tau = O(b_\tau)$ if $a_\tau/b_\tau$ is uniformly bounded over $\tau$; $a_\tau = o(b_\tau)$ if $a_\tau/b_\tau\rightarrow 0$ as $\tau\rightarrow \infty$. Our local analysis is divided into three steps:
\begin{enumerate}[label=(\alph*),topsep=3pt]
\setlength\itemsep{0.0em}
\item we show that $\alpha_\tau = 1$ is selected for the Armijo condition \eqref{equ:Armijo} when $\tau$ is large.

\item we show a relationship between FOTD and the Schwarz scheme: Line 6 in Algorithm \ref{alg:FOTD} is~equivalent to performing a single Newton step for subproblems in Line 5 in Algorithm \ref{alg:OTD}.

\item we prove that FOTD converges linearly, with a linear rate that decays exponentially in $b$.
\end{enumerate}
We present additional assumptions for local analysis.

\begin{assumption}[Hessian approximation vanishes]\label{ass:5}
We assume that $\|H^\tau - \hH^\tau\| = o(1)$, where $H^\tau = \nabla_{\bz}^2\mL^\tau$ is the Lagrangian Hessian~and $\hH^\tau$ is its approximation.
\end{assumption}

A vanishing Hessian modification is typical for the local analysis of SQP to have a superlinear (or quadratic) convergence \cite{Boggs1995Sequential}. As discussed in Remark \ref{rem:4}, we can check \eqref{equ:reduced} for Hessian $H^\tau$ to decide if a modification is needed, since \eqref{equ:reduced} holds locally if SOSC is satisfied at $(\bz^\star, \blambda^\star)$. Equivalently, we check the positive definiteness of $H^\tau + c (G^\tau)^TG^\tau$ for a constant $c$, which is a block-tridiagonal matrix. A parallel Cholesky decomposition is applicable in this regard.

\begin{assumption}[local Lipschitz continuity]\label{ass:6}

We assume there exists a constant $\Upsilon_L$ independent of $k$ such that, for any two points $(\bz, \blambda)$ and $(\bz', \blambda')$ sufficiently close to $(\tz, \tlambda)$,
\begin{align*}
\max\{\|A_k(\bz_k) - A_k(\bz_k')\|,\; \|B_k(\bz_k) - B_k(\bz_k')\|\} \leq &\Upsilon_L\|\bz_k - \bz_k'\|, \\
\|H_k(\bz_k, \blambda_{k+1}) - H_k(\bz_k', \blambda_{k+1}')\|  \leq & \Upsilon_L\nbr{(
\bz_k - \bz_k';\blambda_{k+1} - \blambda_{k+1}')}.
\end{align*}
\end{assumption}

Assumption \ref{ass:6} strengthens the boundedness condition in Assumption \ref{ass:4} to the (local) Lipschitz continuity. We start the local analysis with Step (a).

\vskip4pt
\noindent{\bf Step (a): A unit stepsize is accepted}. We have the following theorem.

\begin{theorem}\label{thm:6}
Suppose Assumptions \ref{ass:1}, \ref{ass:2}, \ref{ass:3}, \ref{ass:4}, \ref{ass:5}, \ref{ass:6} hold for the SQP iterates with search directions $\{(\tDelta\bz^\tau, \tDelta\blambda^\tau)\}_\tau$ satisfying \eqref{equ:error:cond} and parameters $(\eta, \delta)$ satisfying
\begin{equation}\label{cond:2}
\eta_1\geq \frac{17}{\eta_2\gamma_G}, \quad \eta_2\leq \frac{\gamma_{RH}}{12\Upsilon^2}, \quad  \delta \leq \frac{1/2-\beta}{3/2-\beta}\cdot\frac{\eta_2\gamma_{G}}{9\eta_1\Upsilon^2},
\end{equation}
then $\alpha_\tau = 1$ for all sufficiently large $\tau$.
\end{theorem}

\begin{proof}
See Appendix \ref{pf:thm:6}.
\end{proof}

The condition \eqref{cond:2} on $\delta$ is stronger than \eqref{cond:1} up to a multiplier depending on $\beta$. For~FOTD,~we suppose $\mu\geq \bar{\mu}$ with $\bar{\mu}$ given by \eqref{equ:mu:cond}, and let $\delta = C\rho^b$ in \eqref{cond:2} to get a condition on $b$ as in \eqref{equ:b:cond}.

\vskip4pt
\noindent\textbf{Step (b): A relation between FOTD and the Schwarz scheme}. We establish a relationship between FOTD and the Schwarz scheme. We will show that, by specifying $\bd^\tau = (\0; \0; \0; \0)$~in~Problem \eqref{pro:4}, FOTD is equivalent to performing \textit{one Newton step} for subproblems \eqref{pro:2} with a \textit{warm-start initialization} in the Schwarz scheme. The result is summarized in the next theorem.

\begin{theorem}\label{thm:7}
	
Given the current iterate $(\bz^\tau, \blambda^\tau)$, we consider two procedures:

\noindent(a) Schwarz with a warm-start initialization: for $i\in[M-1]$, we specify boundary variables $\bd_i^\tau = (\bx_{m_1}^\tau; \bx_{m_2}^\tau; \bu_{m_2}^\tau; \blambda_{m_2+1}^\tau)$; perform \textbf{one full Newton step} for $\P_{\mu}^i(\bd_i^\tau)$ \eqref{pro:2} at $\mD_i(\bz^\tau, \blambda^\tau)$ and get $(\tbz_i^{\tau+1}, \tblambda_i^{\tau+1})$; then let $(\bz^{\tau+1}, \blambda^{\tau+1}) = \mC(\{(\tbz_i^{\tau+1}, \tblambda_i^{\tau+1})\}_i)$.

\noindent(b) FOTD scheme without Hessian approximation: for $i \in[M-1]$, we specify~boundary variables $\bd_i^\tau = (\0; \0; \0; \0)$; solve $\LP_{\mu}^i(\bd_i^\tau)$ in \eqref{pro:4} with $\hH_k^\tau = H_k^\tau$, $\forall k\in[m_1, m_2]$; obtain $(\tbwi(\bd_i^\tau), \tbzetai(\bd_i^\tau))$; then let $(\tDelta\bz^\tau, \tDelta\blambda^\tau) = \mC(\{(\tbwi(\bd_i^\tau), \tbzetai(\bd_i^\tau))\}_i)$ and update as $(\bz^{\tau+1}, \blambda^{\tau+1}) = (\bz^{\tau}, \blambda^{\tau}) + (\tDelta\bz^\tau, \tDelta\blambda^\tau)$.

Then, starting from $(\bz^\tau, \blambda^\tau)$, both procedures generate the same next iterate $(\bz^{\tau+1}, \blambda^{\tau+1})$.
\end{theorem}

\begin{proof}
See Appendix \ref{pf:thm:7}.
\end{proof}

Theorem \ref{thm:7} reveals a strong relation between the Schwarz scheme and FOTD. Locally, FOTD can be seen as an improvement of the warm-start Schwarz scheme, where a single Newton step is performed for the subproblems, instead of solving the subproblems to optimality as the original Schwarz did. The warm initialization is recommended by \cite{Na2022Convergence} for practical purpose, which avoids~the case where the solutions of the same subproblem in different iterations are very distinct.

\vskip4pt

\noindent\textbf{Step (c): local linear convergence of FOTD}. We now establish the local convergence rate for FOTD. For $\tau\geq 0$ we let
\vskip-10pt\begin{equation}\label{equ:def:psi}
\Psi^\tau \coloneqq \max_{k\in[N]}\Psi_k^\tau \coloneqq \max_{k\in[N]}\|(\bz_k^\tau, \blambda_k^\tau) - (\tz_k, \tlambda_k)\|.
\end{equation}
We require the following lemma that shows the one-step error recursion. We use $C_1, C_2$ to~denote generic constants that are ensured to exist, but may differ from the constant $C$ in Theorem \ref{thm:3}.~The constant $\rho$ is from Theorem \ref{thm:3}.

\begin{lemma}\label{lem:3}
Consider the FOTD iterates $\{(\bz^\tau, \blambda^\tau)\}_\tau$ under Assumptions \ref{ass:1}, \ref{ass:2}, \ref{ass:3}, \ref{ass:4},~\ref{ass:5}, \ref{ass:6} and suppose $\mu\geq \bar{\mu}$ with $\bar{\mu}$ given by $\eqref{equ:mu:cond}$ and $\eta$ satisfies \eqref{cond:2}. There exist constants $C_1>0$ and $\rho\in(0, 1)$ independent of $\tau, \eta_1, \eta_2, \beta$, such that, if $b$ satisfies
\vskip-3pt\begin{equation}\label{equ:b:cond}
b\geq \frac{\log\cbr{(3/2-\beta)C_1\eta_1/((1/2-\beta)\eta_2)}}{\log(1/\rho)},
\end{equation}
then, for all sufficiently large $\tau$ (in each case below, $m_1, m_2$ depend on $i$ correspondingly, see \eqref{equ:m}),
\vskip-10pt\begin{subequations}\label{equ:result}
\begin{align}
\Psi_k^{\tau+1} \leq & o(\Psi^\tau) + C_1\cbr{\rho^{k-m_1}\|\bx_{m_1}^\tau - \tx_{m_1}\| + \rho^{m_2-k}\nbr{\begin{pmatrix}
\bz_{m_2}^\tau - \tz_{m_2}\\
\blambda_{m_2+1}^\tau - \tlambda_{m_2+1}
\end{pmatrix}} },\;\;  \substack{\forall k\in[n_i, n_{i+1}) \\ i\in[M-2]}\;\;,  \label{equ:result:b}\\[3pt]
\Psi_k^{\tau+1} \leq & o(\Psi^\tau) + C_1\rho^{k-m_1}\|\bx_{m_1}^\tau - \tx_{m_1}\|, \quad \forall k\in[n_{M-1}, n_M]. \label{equ:result:c}
\end{align}	
\end{subequations}	
\end{lemma}

\begin{proof}
See Appendix \ref{pf:lem:3}.
\end{proof}

Lemma \ref{lem:3} relies on the decay structure of the KKT matrix inverse, established in \cite[Lemma 2]{Na2023Superconvergence}. In particular, the authors showed that, if Assumptions \ref{ass:1}-\ref{ass:3} hold for subproblems \eqref{pro:4} (verified in Corollary \ref{cor:1}), the block matrices of KKT inverse corresponding to each stage have an~exponentially decay structure (see \cite[Figure 2]{Na2023Superconvergence}). The core of such a result is the primal-dual sensitivity analysis of NLDPs \cite{Na2020Exponential, Na2022Convergence}, which we also made use of in the proof of Theorem \ref{thm:2}.

Seeing from \eqref{equ:result}, the first $o(\Psi^\tau)$ term is the algorithmic convergence rate, which is superlinear and achieved by the SQP framework; the second term $C_1\{\rho^{k-m_1}\|\bx_{m_1}^\tau - \tx_{m_1}\| + \rho^{m_2-k}\|(\bz_{m_2}^\tau, \blambda_{m_2+1}^\tau) \\- (\tz_{m_2}, \tlambda_{m_2+1})\|\}$ has a linear convergence rate, which is brought by the horizon truncation in~OTD. Specifically, the perturbations come from the misspecification of boundary variables: we use $\bd_i^\tau = (\0;\0;\0;\0)$ while $\bd_i^\star = (\Delta\bx_{m_1}; \Delta\bx_{m_2}; \Delta\bu_{m_2};\Delta\blambda_{m_2+1})$.

Our result in Lemma \ref{lem:3} is different from \cite[Theorem 2]{Na2023Superconvergence}, where the algorithmic convergence~rate is quadratic as they used the unperturbed Hessian $H^\tau$ in each iteration. The Hessian modification is necessary in our case as it ensures the global convergence. Our result is also different from \cite[Theorem 7]{Na2022Convergence}, which has no algorithmic rate as they solved the subproblems to optimality. Our result reveals that, even if we do not solve the subproblems to optimality, as long as the algorithmic rate is faster than linear (superlinear in our case), the perturbation rate will always dominate for large $\tau$, which is linear. Therefore, a local linear convergence is still achieved.

We summarize the local convergence guarantee in the next theorem.

\begin{theorem}[uniform linear convergence]\label{thm:8}
Under the setup of Lemma \ref{lem:3}, we have for all sufficiently large $\tau$ that $\Psi^{\tau+1} \leq 3C_1\rho^b\cdot \Psi^\tau$ for constants $C_1>0$ and $\rho\in(0, 1)$ from Lemma \ref{lem:3}.
\end{theorem}

\begin{proof}
We use the facts that $\bx_0^\tau = \tx_0 = \bbx_0$ (see the proof of Theorem \ref{thm:3}) and $\min\{k-m_1,m_2-k\} \geq b$ for all $k\in[n_i, n_{i+1}]$, and apply \eqref{equ:result}. Then, we obtain $\Psi_k^{\tau+1}\leq 3C_1\rho^b\Psi^\tau$ for all $k$. Thus, by \eqref{equ:def:psi},~we complete the proof.
\end{proof}

We finally summarize our convergence analysis of FOTD in the next theorem.

\begin{theorem}[convergence of FOTD in Algorithm \ref{alg:FOTD}]\label{thm:9}
\hskip-10pt Consider Algorithm \ref{alg:FOTD} under~Assumptions \ref{ass:1}-\ref{ass:3}, \ref{ass:4}, and suppose $\mu$ satisfies \eqref{equ:mu:cond}. There exist constants $C_2>0, \rho\in(0, 1)$ independent of $\tau$ and algorithmic parameters $\eta_1, \eta_2, \beta$, such that if
\begin{equation*}
\eta_1\geq \frac{17}{\eta_2\gamma_G}, \quad \eta_2\leq \frac{\gamma_{RH}}{12\Upsilon^2}, \quad b \geq \frac{\log\cbr{(3/2-\beta)C_2\eta_1/((1/2-\beta)\eta_2)} }{\log(1/\rho)},
\end{equation*}
then $\|\nabla\mL^\tau\|\rightarrow 0$ as $\tau\rightarrow \infty$. Moreover, if Assumptions \ref{ass:5}, \ref{ass:6} hold locally as well, then for all sufficiently large $\tau$, $\alpha_\tau = 1$ and
\begin{equation}\label{equ:linear}
\Psi^{\tau+1} \leq (C_2\rho^b)\Psi^\tau.
\end{equation}
\end{theorem}

\begin{proof}
By Theorems \ref{thm:5}, \ref{thm:6}, \ref{thm:8}, Lemma \ref{lem:3} and rescaling $C_2$ properly, we complete the~proof.
\end{proof}

The result in \eqref{equ:linear} matches the local result \eqref{equ:local:result} proved in \cite{Na2022Convergence}; thus, we answer the Q2 raised in \cref{sec:1}---it is not necessary to solve the subproblems to optimality for achieving the local linear convergence. As mentioned earlier, as long as the algorithmic rate is faster than linear (e.g., one Newton step for each subproblem), the local linear rate, induced by the decay of the~sensitivity of perturbations, is always achieved. We should also mention that both the Schwarz scheme and FOTD have linear rates of the form $C\rho^b$. Since the analyses of both algorithms only~claim the~existence~of some constants $C>0$ and $\rho\in(0,1)$ (see Theorem \ref{thm:9} and \cite[Theorem 8]{Na2022Convergence}), we can always use~larger constants between the two algorithms to make their linear rates identical. In fact, comparing their constant $C$ can be difficult since it depends on the sharpness of the derivation and various quantities of the problem (e.g., $\Upsilon,\gamma_{G}$ etc.). However, the constant $\rho$ is the same, and is from \cite[Theorem~5.7]{Na2020Exponential}. More importantly, both their linear rates decay exponentially in the overlap size $b$.

\section{Numerical experiments}\label{sec:7}

We first conduct a numerical experiment on a toy NLDP in \cite{Na2023Superconvergence}:\vskip-0.5cm
\begin{subequations}\label{pro:5}
\begin{align}
\min_{\bx, \bu}\;\; & \sum_{k=0}^{N-1}\cbr{2\cos(x_k-d_k)^2 + C_1(x_k-d_k)^2 - C_2(u_k-d_k)^2} + C_1x_N^2, \label{pro:5a}\\
\text{s.t.}\;\; & x_{k+1} = x_k + u_k + d_k, \quad \forall k\in[N-1], \label{pro:5b}\\
& x_0 = 0. \label{pro:5c}
\end{align}
\end{subequations}\vskip-0.1cm
\noindent Here, $n_x=n_u = 1$, and references $\{d_k\}$ are specified later. As checked in \cite{Na2023Superconvergence}, if $C_1-2>4|C_2|$, then $Z^T(\bz)H(\bz, \blambda)Z(\bz)\succeq (C_1-2 - 4|C_2|)/4\cdot I \;\; \forall \bz, \blambda$. Thus, we simply let $\hH^\tau = H^\tau$ in implementation.

We implement seven methods: one centralized method IPOPT \cite{Waechter2005implementation} (which is our baseline), and~six parallel methods including the proposed FOTD, the Schwarz method \cite{Na2022Convergence}, the direct \mbox{multiple}~shooting method \cite{Bock1984Multiple}, the iterative differential dynamic \mbox{programming}~(IDDP) method \cite{Tassa2012Synthesis}, ILQR \cite{Li2007Iterative},~and ADMM \cite{ODonoghue2013splitting}. The scheme of IDDP is almost the same as ILQR, except that the control variables of QPs are computed by rolling out the policies along the original nonlinear dynamics rather than~the linearized dynamics \cite{Roulet2022Iterative}. Since the dynamics \eqref{pro:5b} are linear, IDDP and ILQR are identical in~this case. We will study a temperature control problem of thin plates later,~where~we~then~have nonlinear dynamics. We aim to demonstrate three points in this experiment: (i) FOTD converges~globally while the Schwarz may not converge \textit{within a reasonable computational budget} for some initializations. (ii) FOTD exhibits at least linear convergence locally, and the larger overlap size $b$ leads~to the faster convergence. (iii) FOTD is a superior parallel method. It is as competitive as the centralized solver IPOPT, and robust to the penalty parameter $\mu$ (cf. \eqref{pro:4a}). To illustrate the first~point,~we generate random initial iterates that are shared by all methods, and see if each method converges~for all initializations. To illustrate the second point, we plot $\|\nabla\mL^\tau\|$ v.s. $\tau$ for FOTD with different $b$, and see how $\|\nabla\mL^\tau\|$ behaves on the tail. To illustrate the third point, we compare the KKT~residual $\|\nabla\mL^\tau\|$ and running  time for all methods, and vary $\mu$ for FOTD to test its robustness.

\vskip 4pt
\noindent{\bf Simulation setting:} We consider three cases in Table \ref{tab:1}. For all parallel methods, we use the same horizon decomposition; that is, we decompose $[0,N]$ evenly with length $M$ for each (exclusive)~short interval. Different from the methods of multiple shooting, ILQR, IDDP, and ADMM, the Schwarz and FOTD have overlaps between two successive short intervals. We vary the overlap size $b$ from $\{1,5,25\}$. For FOTD, we let $\eta_1=10$, $\eta_2=\beta=0.1$, and vary $\mu$~in a wide range $\{1,25,125\}$. The~setup of $\mu$ is shared by the Schwarz and ADMM, both of which also require the penalty parameter. When doing the backtracking line search, we decrease the stepsize by a factor of $0.9$ each time until the line search condition is satisfied. For each case in Table \ref{tab:1} and each setup of $b$ and $\mu$, we generate 5 initial iterates for FOTD that are shared by other methods: one is $(\bz^0, \blambda^0) = (\0,\0)$ and the~other four are from $x_k^0,u_k^0,\lambda_k^0\stackrel{iid}{\sim}\text{Uniform}(-10^5,10^5)$ (with $x_0^0 = 0$).~For~those~methods~that~have~to~solve 
nonlinear subproblems (i.e., the Schwarz, multiple shooting, ADMM), we apply Julia/JuMP package \cite{Dunning2017JuMP} with IPOPT solver \cite{Waechter2005implementation}. For all methods except ADMM, we stop the iteration if \vskip-0.3cm
\begin{equation}\label{equ:stop:cond}
\|\nabla\mL^\tau\|\leq 10^{-6}\quad\quad \text{ OR }\quad\quad \|(\bz^{\tau+1}-\bz^\tau; \blambda^{\tau+1}-\blambda^\tau)\|\leq 10^{-6}.
\end{equation}\vskip-0.1cm
\noindent Since FOTD, ILQR, and IDDP have comparable computations in each iteration (i.e., they all~solve QPs), we regard them~as \textit{converged} if they trigger the condition \eqref{equ:stop:cond} within 40 iterations budget (we see from Figure \ref{fig:2} that FOTD actually needs much less iterations). For the multiple shooting~and Schwarz that solve NLDPs, we reduce the iteration budget a bit to 30 for sake of~a~\mbox{comparable}~total computation cost. For ADMM, we observe in our experiment that its KKT sequence has a long~flat tail (see \cite[Figure 6]{Na2022Convergence}). Thus, we prefer to stop the ADMM iteration early by relaxing the condition \eqref{equ:stop:cond} to $\|\nabla\mL^\tau\|\leq 10^{-6}$ OR $\|(\bz^{\tau+1}-\bz^\tau; \blambda^{\tau+1}-\blambda^\tau)\|\leq 10^{-3}$, and increase its iteration budget to~100. For the methods that do not compute $\blambda$ (i.e., multiple shooting, ADMM,~ILQR, IDDP),~we~let~$\blambda^\tau = - (G^\tau(G^\tau)^T)^{-1}G^\tau\nabla g(\bz^\tau)$ when evaluating $\nabla\mL^\tau$. Note that $\blambda^\tau\rightarrow\tlambda$ as $\bz^\tau\rightarrow\tz$.
In addition to the above settings, we try three linear system solvers for FOTD: one is sparse LU (which is a default choice and adopted by other methods), and the other~two~are~\mbox{generalized}~\mbox{minimal}~\mbox{residual}~(GMRES) and induced dimension reduction (IDR) methods implemented in Julia/IterativeSolvers package.

\vskip-0.6cm
\begin{table}[!htp]
\centering
\captionsetup{justification=centering}
\caption{Simulation setups}\label{tab:1}
\begin{tabular}{ |c|c|c|c|c| }
\hline
Cases  & $N$ &  $M$  & $(C_1, C_2)$ & $d_k$  \\
\hline
Case 1 & $5000$ &  $50$  &  $(8, 1)$ & $1$ \\
\hline
Case 2 & $5000$ & $100$  & $(15, 3)$ & $100\sin(k)^2$ \\
\hline
Case 3  & $10000$ & $100$  & $(12, 2)$ & $5\sin(k)$ \\
\hline
\end{tabular}
\end{table}

\vskip-0.3cm
\noindent\textbf{Result summary:} First, we investigate whether different methods converge within the computational budget for all five initializations. We observe that the multiple shooting, ILQR, and FOTD converge by triggering \eqref{equ:stop:cond} for all three cases in Table \ref{tab:1} and for all initializations and setups.~The Schwarz converges for most of cases, but does not converge \textit{within~the~budget} for Case 2 with $b=1$. In particular, for this case, there are $2,4,5$ out of $5$ initializations that~the Schwarz does not trigger \eqref{equ:stop:cond} when setting $\mu = 1,25,125$, respectively. For ADMM, it converges for all three~cases~with~$\mu=1$, but does not converge for all three cases with $\mu=25$ and $125$. We do not claim that ADMM diverges with large $\mu$ (its KKT is actually below $10^{-3}$ with $>2000$ iterations), but we clearly see that ADMM is not as robust as the Schwarz and FOTD to the parameter $\mu$.

Second, we draw the KKT convergence plots for FOTD in Figure \ref{fig:2}. We apply the sparse LU~solver for solving QPs and, for Cases 1, 2, and 3, we take $\mu = 1$, $25$, and $125$ as examples respectively. We emphasize that the other setups of $\mu$ of each case have similar convergence behavior (as revealed by Tables \ref{tab:2}-\ref{tab:4}). From Figure \ref{fig:2}, we observe that FOTD converges for all initializations~with~\mbox{different} $\mu$ for different cases and exhibits between linear and superlinear convergence locally. Its performance is robust to $\mu$, and a larger $b$ generally leads to a faster convergence (cf. Figures \ref{Case31}-\ref{Case33}).~Our~observation is consistent with Theorem \ref{thm:9}.

\begin{figure}[tp]
\centering
\subfloat[Case 1, $b=1$]{\label{Case11}\includegraphics[width=5.6cm]{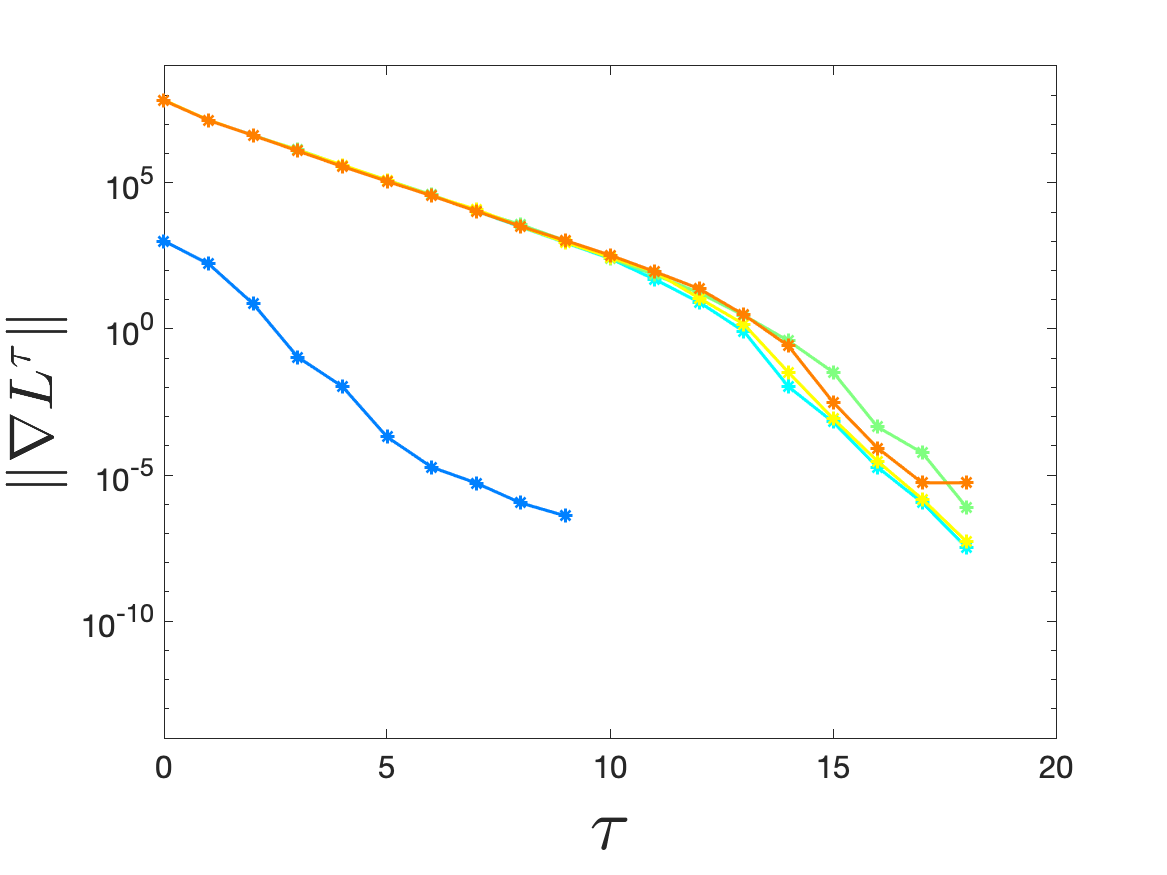}}
\subfloat[Case 1, $b=5$]{\label{Case12}\includegraphics[width=5.6cm]{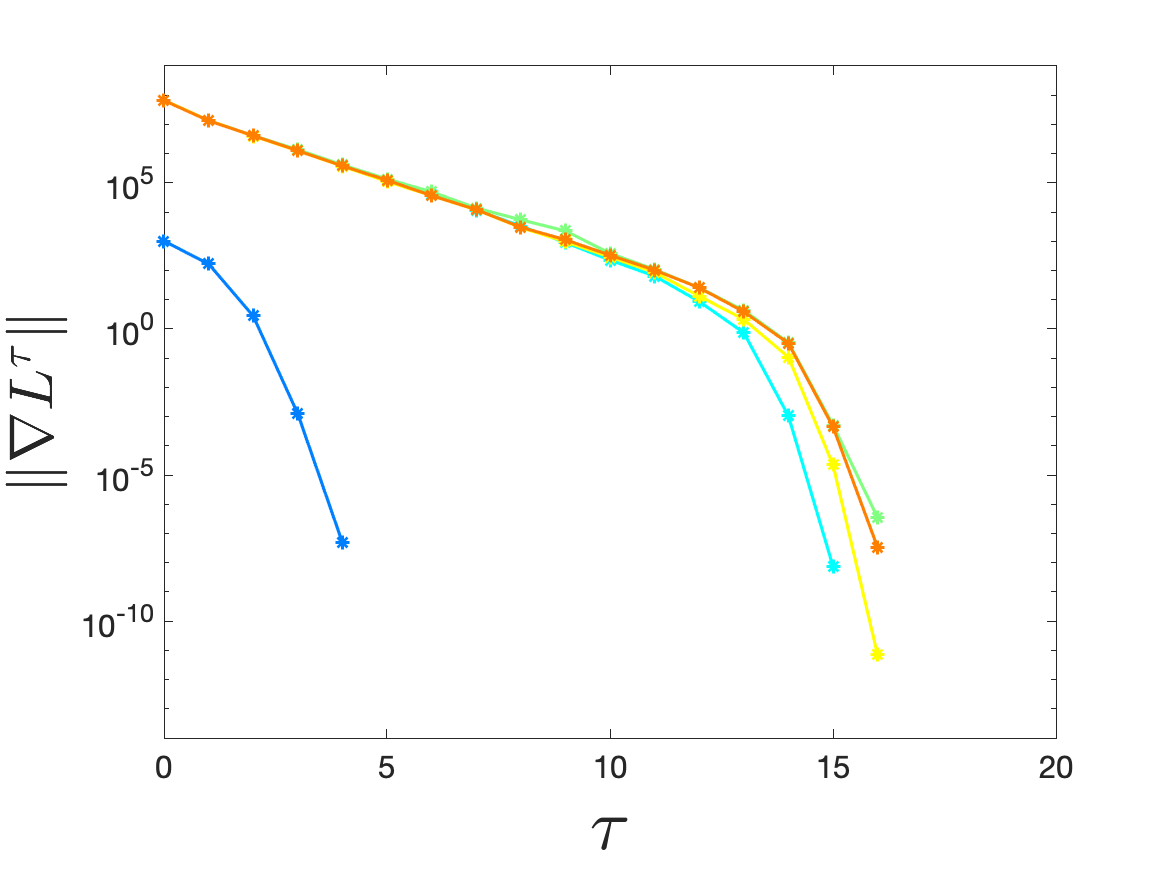}}
\subfloat[Case 1, $b=25$]{\label{Case13}\includegraphics[width=5.6cm]{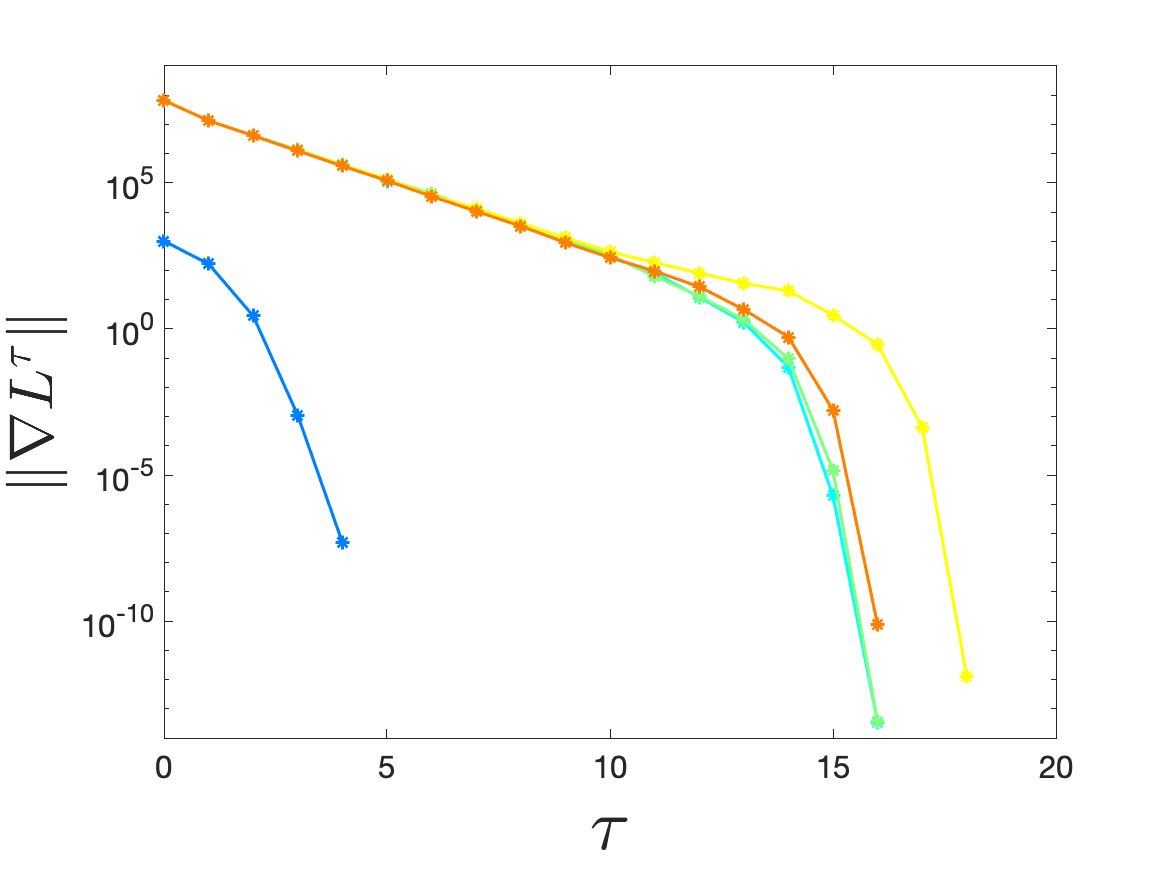}}
\vspace{-0.3cm}
\subfloat[Case 2, $b=1$]{\label{Case21}\includegraphics[width=5.6cm]{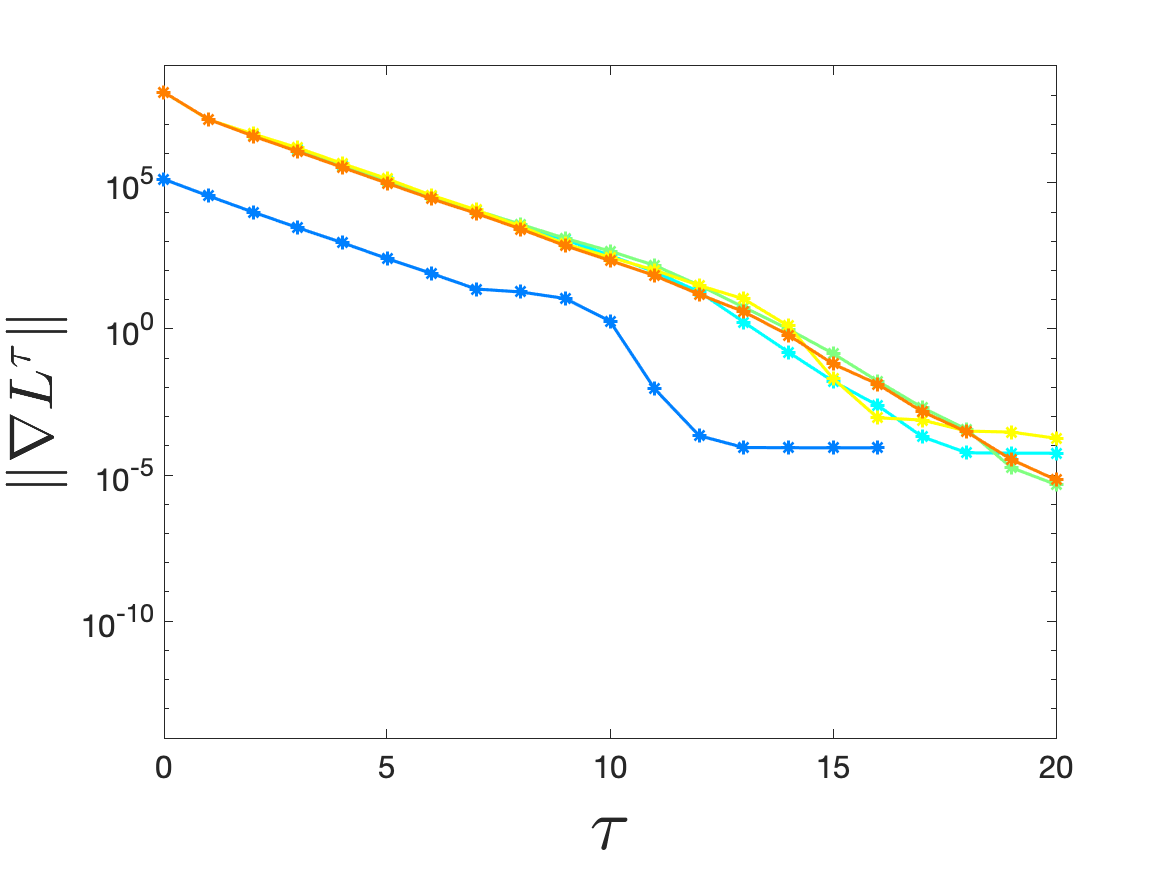}}
\subfloat[Case 2, $b=5$]{\label{Case22}\includegraphics[width=5.6cm]{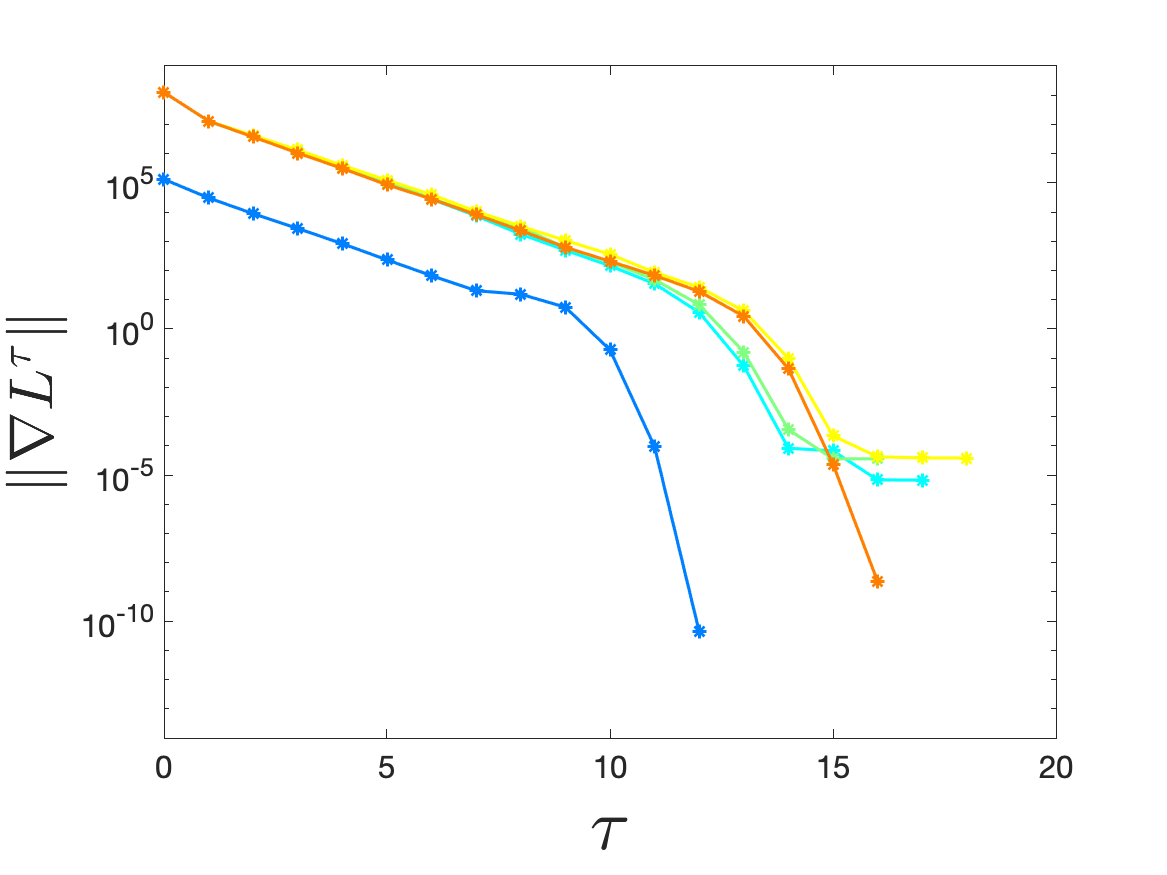}}
\subfloat[Case 2, $b=25$]{\label{Case23}\includegraphics[width=5.6cm]{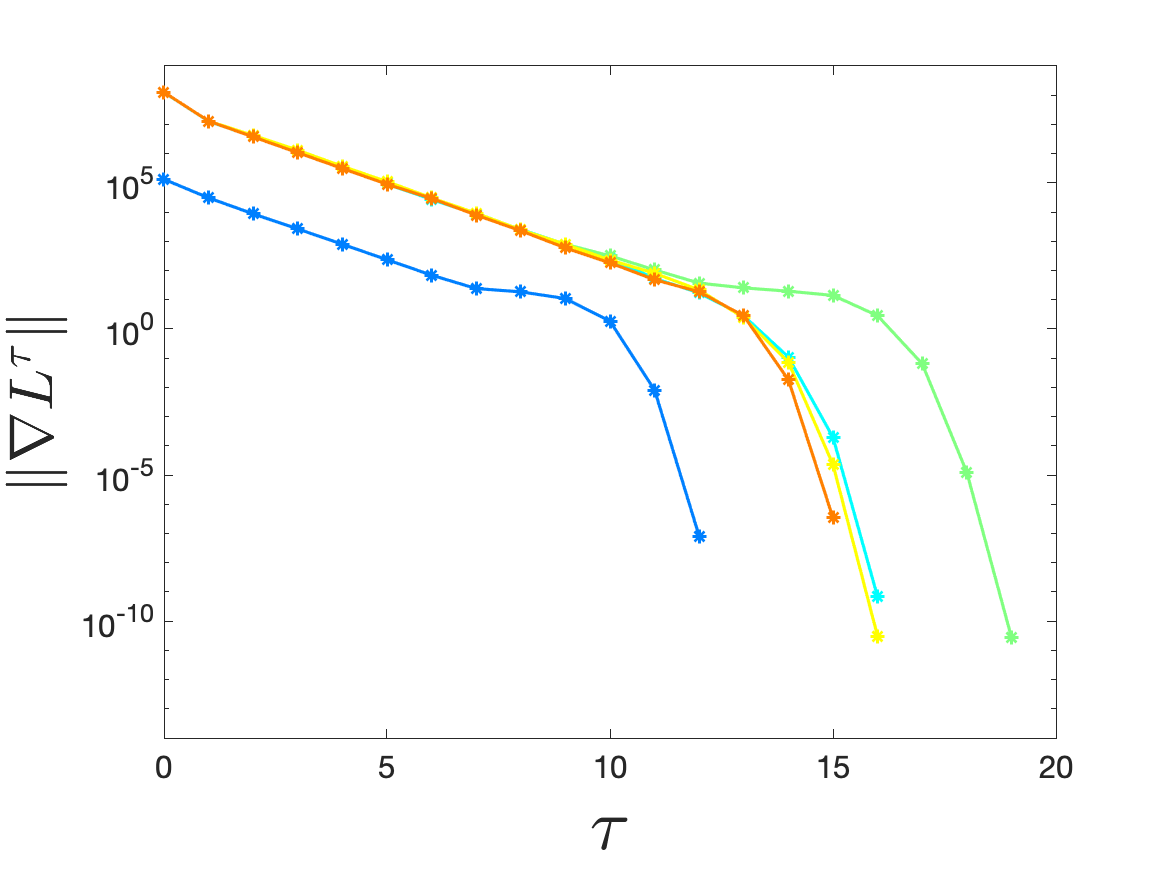}}
\vspace{-0.3cm}
\subfloat[Case 3, $b=1$]{\label{Case31}\includegraphics[width=5.6cm]{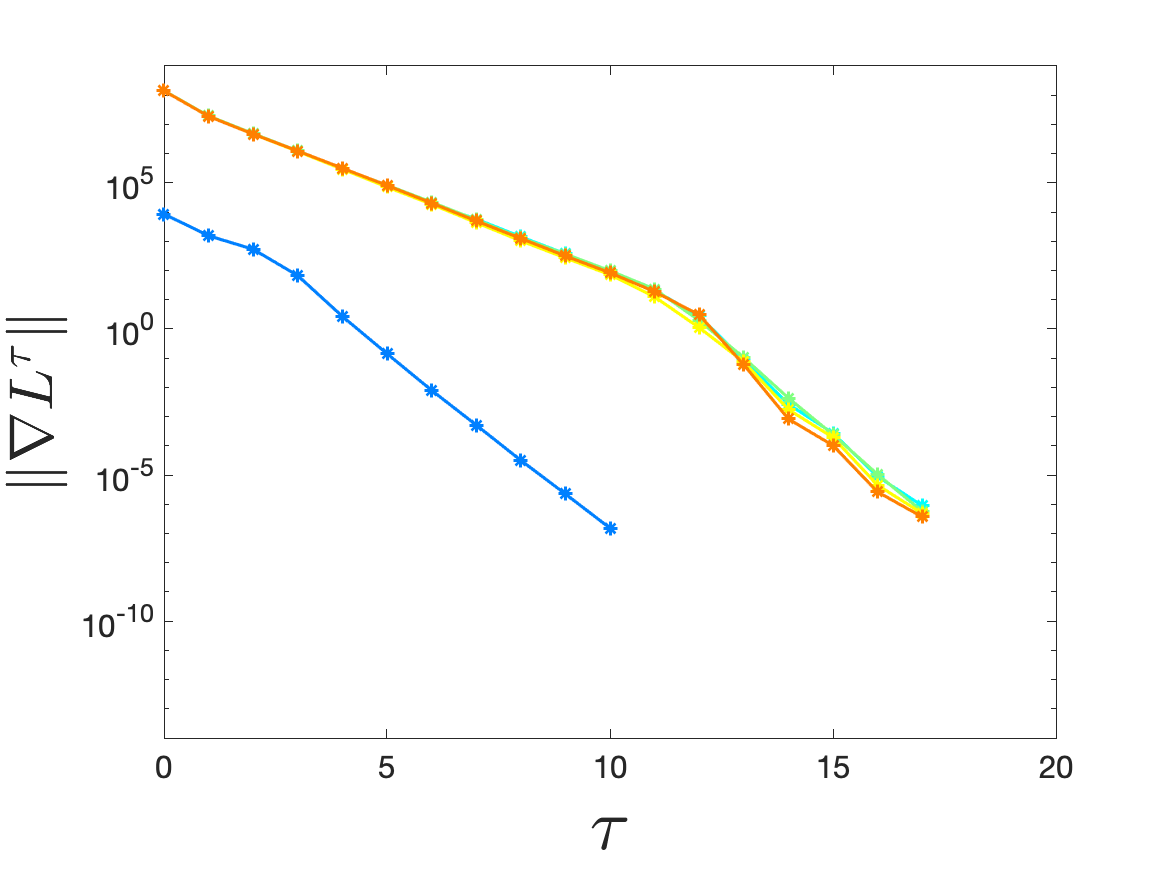}}
\subfloat[Case 3, $b=5$]{\label{Case32}\includegraphics[width=5.6cm]{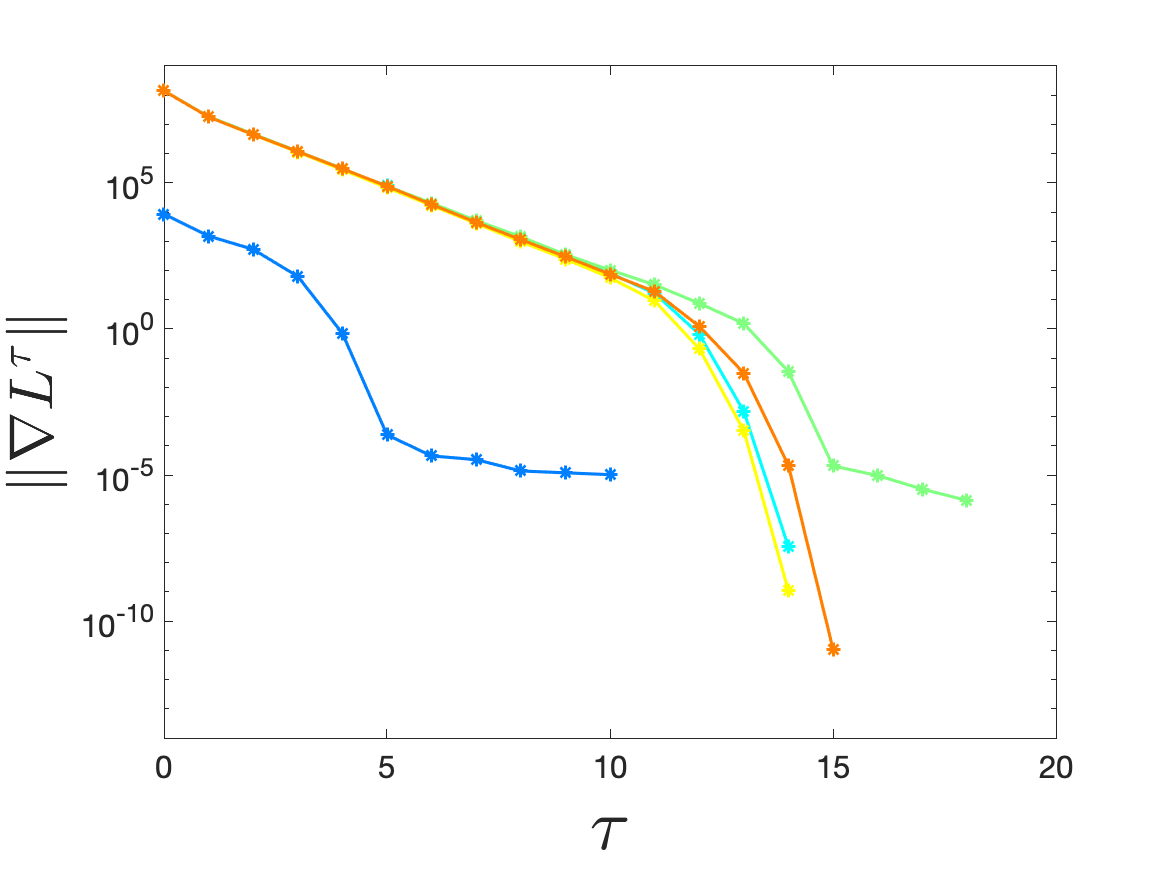}}
\subfloat[Case 3, $b=25$]{\label{Case33}\includegraphics[width=5.6cm]{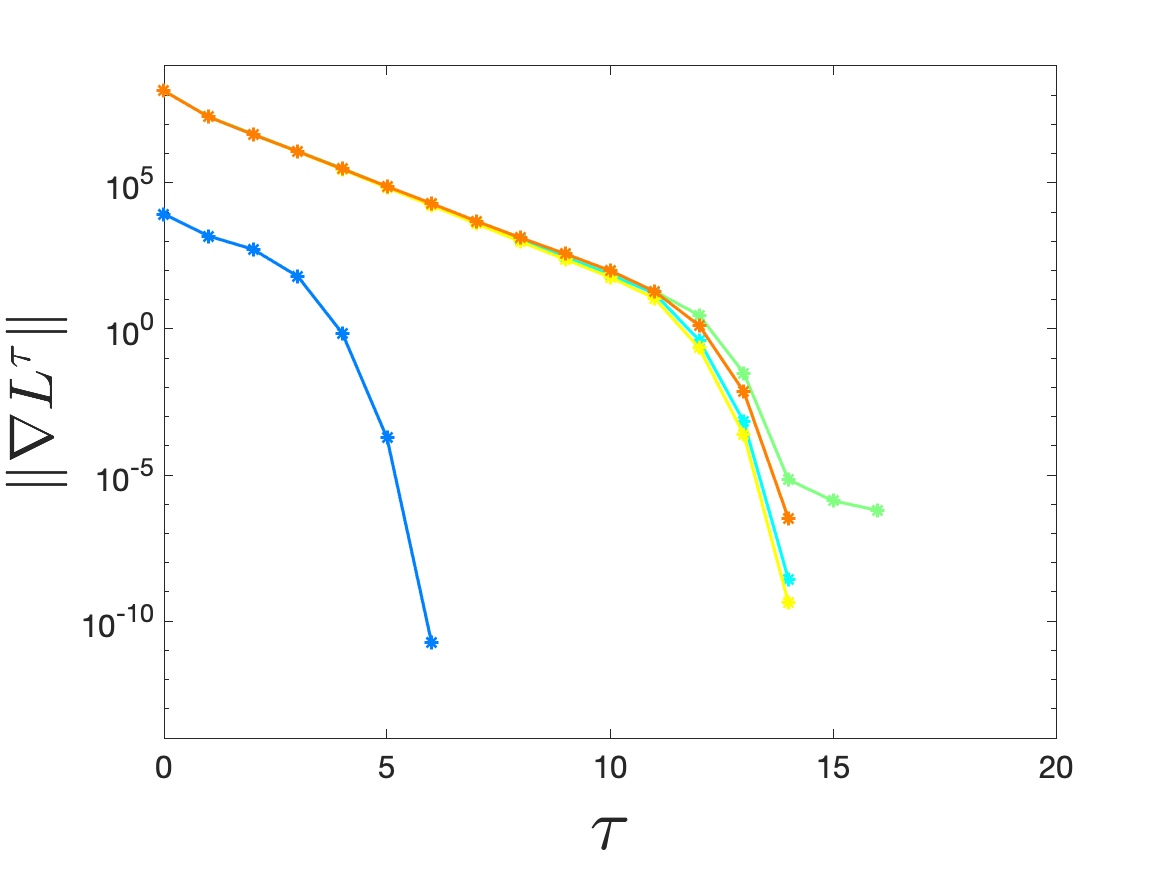}}
\caption{Convergence figures. The three plots on each row correspond to the same case in Table \ref{tab:1}, while the three plots on each column correspond to the same setup of $b$. Each plot has 5 lines corresponding to 5 initial iterates. The blue line that is separated from the other four lines corresponds to the initial point $(\bz^0, \blambda^0) = (\0,\0)$. We observe that FOTD converges for all initializations. It exhibits between linear and superlinear convergence locally, and a larger $b$ generally leads to a faster convergence.}\label{fig:2}
\end{figure}

\begin{table}[t]
\centering\caption{The KKT residual and running time for all methods implemented on Case 1. For the Schwarz and FOTD, the smallest KKT residual and least running time among different setups of $b$ are highlighted for each setup of $\mu$.}\label{tab:2}
\begin{tabular}{ |c|cc|ccc|ccc| }
\hline
Type &  \multicolumn{2}{c|}{Method} & \multicolumn{3}{c|}{KKT residual $(10^{-7})$} & \multicolumn{3}{c|}{Time (sec.)} \\
\hline
Centralized & \multicolumn{2}{c|}{IPOPT} & \multicolumn{3}{c|}{172.686} & \multicolumn{3}{c|}{0.851}\\
\hline
\multirow{16}{*}{Decomposed} & \multicolumn{2}{c|}{MultiShoot} &  \multicolumn{3}{c|}{23.375} & \multicolumn{3}{c|}{7.863}\\
\cline{2-9}
& \multicolumn{2}{c|}{ILQR} &  \multicolumn{3}{c|}{29.268} & \multicolumn{3}{c|}{7.047}\\
\cline{2-9}
& \multicolumn{2}{c|}{ADMM} &  \multicolumn{3}{c|}{2603.063} & \multicolumn{3}{c|}{36.125}\\
\cline{2-9}
& & & $b=1$ & $b=5$ & $b=25$ & $b=1$ & $b=5$ & $b=25$\\
\cline{4-9}
& \multirow{3}{*}{FOTD (sparse LU)} & $\mu=1$ & 13.324 & 0.878 & \textbf{0.0979} & 0.474 & \textbf{0.455} & 0.881\\
&  & $\mu=25$ & 4.828 & 1.618 & \textbf{0.0980} & 0.460 & \textbf{0.449} & 0.885\\
&  & $\mu=125$ & 7.210 & 0.328 & \textbf{0.0977} & 0.459 & \textbf{0.453} & 0.881\\
\cline{2-9}
& \multirow{3}{*}{FOTD (GMRES)} & $\mu=1$ & 4.686 & 0.672 & \textbf{0.549} & 0.566 & \textbf{0.537} & 1.019\\
&  & $\mu=25$ & 12.841 & 0.897 & \textbf{0.132} & 0.597 & \textbf{0.559} & 1.026\\
&  & $\mu=125$ & 4.997 & \textbf{0.251} & 0.451 & 0.621 & \textbf{0.604} & 1.041\\
\cline{2-9}
& \multirow{3}{*}{FOTD (IDR)} & $\mu=1$ & 4.051 & 0.275 & \textbf{0.106} & 0.619 & \textbf{0.468} & 0.875\\
&  & $\mu=25$ & 2.688 & 0.832 & \textbf{0.0982} & 0.500 & \textbf{0.467} & 0.896\\
&  & $\mu=125$ & 5.206 & 0.380 & \textbf{0.0978} & 0.536 & \textbf{0.491} & 0.878\\
\cline{2-9}
& \multirow{3}{*}{Schwarz} & $\mu=1$ & 0.448 & 0.298 & \textbf{0.0418} & \textbf{1.969} & 2.197 & 2.009\\
&  & $\mu=25$ & 1.451 & 0.454 & \textbf{0.0418} & 2.116 & 2.199 & \textbf{2.027}\\
&  & $\mu=125$ & 3.467 & 0.521 & \textbf{0.0422} & 2.173 & 2.185 & \textbf{2.160}\\
\hline
\end{tabular}
\end{table}

\begin{table}[h]
\centering\caption{The KKT residual and running time for all methods implemented on Case 2. For the Schwarz and FOTD, the smallest KKT residual and least running time among different setups of $b$ are highlighted for each setup of $\mu$. The dash ``-" means the scheme does not trigger \eqref{equ:stop:cond} within the budget.}\label{tab:3}
\begin{tabular}{ |c|cc|ccc|ccc| }
\hline
Type &  \multicolumn{2}{c|}{Method} & \multicolumn{3}{c|}{KKT residual $(10^{-7})$} & \multicolumn{3}{c|}{Time (sec.)} \\
\hline
Centralized & \multicolumn{2}{c|}{IPOPT} & \multicolumn{3}{c|}{31.227} & \multicolumn{3}{c|}{1.075}\\
\hline
\multirow{16}{*}{Decomposed} & \multicolumn{2}{c|}{MultiShoot} &  \multicolumn{3}{c|}{15.998} & \multicolumn{3}{c|}{10.126}\\
\cline{2-9}
& \multicolumn{2}{c|}{ILQR} &  \multicolumn{3}{c|}{26.685} & \multicolumn{3}{c|}{9.334}\\
\cline{2-9}
& \multicolumn{2}{c|}{ADMM} &  \multicolumn{3}{c|}{11841.583} & \multicolumn{3}{c|}{35.133}\\
\cline{2-9}
& & & $b=1$ & $b=5$ & $b=25$ & $b=1$ & $b=5$ & $b=25$\\
\cline{4-9}
& \multirow{3}{*}{FOTD (sparse LU)} & $\mu=1$ & 607.625 & 117.596 & \textbf{1.108} & 0.708 & \textbf{0.599} & 0.863\\
&  & $\mu=25$ & 364.783 & 106.214 & \textbf{0.848} & 0.699 & \textbf{0.601} & 0.906\\
&  & $\mu=125$ & 240.339 & \textbf{0.268} & 4.753 & 0.735 & \textbf{0.557} & 1.026\\
\cline{2-9}
& \multirow{3}{*}{FOTD (GMRES)} & $\mu=1$ & 798.660 & 115.055 & \textbf{56.258} & 1.224 & \textbf{1.084} & 1.503\\  
&  & $\mu=25$ & 193.723 & \textbf{5.581} & 12.745 & 1.265 & \textbf{1.077} & 1.652\\
&  & $\mu=125$ & 738.168 & 62.946 & \textbf{33.445} & 1.310 & \textbf{1.029} & 1.520\\
\cline{2-9}
& \multirow{3}{*}{FOTD (IDR)} & $\mu=1$ & 662.446 & 37.652 & \textbf{8.290} & 0.828 & \textbf{0.695} & 1.098\\
&  & $\mu=25$ & 310.175 & 48.570 & \textbf{7.662} & 0.848 & \textbf{0.724} & 1.203\\
&  & $\mu=125$ & 160.98 & \textbf{9.470} &  15.706 & 0.900 & \textbf{0.754} & 1.245\\
\cline{2-9}
& \multirow{3}{*}{Schwarz} & $\mu=1$ & 11.675 & 14.911 & \textbf{1.916} & 2.427 & 2.316 & \textbf{2.055}\\
&  & $\mu=25$ & 14.402 & 14.770 & \textbf{1.203} & 2.852 & 2.335 & \textbf{2.057}\\
&  & $\mu=125$ & - & 22.714 & \textbf{9.227} & - & 2.229 & \textbf{2.155}\\
\hline
\end{tabular}
\end{table}

\begin{table}[h]
\centering\caption{The KKT residual and running time for all methods implemented on Case 3. For the Schwarz and FOTD, the smallest KKT residual and least running time among different setups of $b$ are highlighted for each setup of $\mu$. }\label{tab:4}
\begin{tabular}{ |c|cc|ccc|ccc| }
\hline
Type &  \multicolumn{2}{c|}{Method} & \multicolumn{3}{c|}{KKT residual $(10^{-7})$} & \multicolumn{3}{c|}{Time (sec.)} \\
\hline
Centralized & \multicolumn{2}{c|}{IPOPT} & \multicolumn{3}{c|}{285.911} & \multicolumn{3}{c|}{1.405}\\
\hline
\multirow{16}{*}{Decomposed} & \multicolumn{2}{c|}{MultiShoot} &  \multicolumn{3}{c|}{5.652} & \multicolumn{3}{c|}{18.571}\\
\cline{2-9}
& \multicolumn{2}{c|}{ILQR} &  \multicolumn{3}{c|}{19.983} & \multicolumn{3}{c|}{15.732}\\
\cline{2-9}
& \multicolumn{2}{c|}{ADMM} &  \multicolumn{3}{c|}{28467.487} & \multicolumn{3}{c|}{38.571}\\
\cline{2-9}
& & & $b=1$ & $b=5$ & $b=25$ & $b=1$ & $b=5$ & $b=25$\\
\cline{4-9}
& \multirow{3}{*}{FOTD (sparse LU)} & $\mu=1$ & 69.097 & 159.907 & \textbf{1.471} & 1.067 & \textbf{1.048} & 1.480\\
&  & $\mu=25$ & 23.867 & 6.988 & \textbf{0.714} & 1.161 & \textbf{0.942} & 1.480\\
&  & $\mu=125$ & 4.944 & 23.393 & \textbf{1.906} & 1.066 & \textbf{1.038} & 1.513\\
\cline{2-9}
& \multirow{3}{*}{FOTD (GMRES)} & $\mu=1$ & 4.293 & 28.472 & \textbf{0.0544} & 1.382 & \textbf{1.373} & 1.903\\
&  & $\mu=25$ & 47.298 & 12.676 & \textbf{0.624} & 1.490 & \textbf{1.284} & 1.877\\
&  & $\mu=125$ & \textbf{5.838} & 11.084 & 17.412 & 1.559 & \textbf{1.355} & 1.897\\
\cline{2-9}
& \multirow{3}{*}{FOTD (IDR)} & $\mu=1$ & 144.449 & 5.971 & \textbf{3.612} & 1.110 & \textbf{0.989} & 1.571\\
&  & $\mu=25$ & 40.627 & 0.735 & \textbf{0.532} & 1.147 & \textbf{1.020} & 1.557\\
&  & $\mu=125$ & 31.889 & 29.349 &  \textbf{0.0760} & 1.140 & \textbf{1.062} & 1.574\\
\cline{2-9}
& \multirow{3}{*}{Schwarz} & $\mu=1$ & 2.363 & 1.691 & \textbf{0.0154} & 4.914 & 3.994 & \textbf{2.686}\\
&  & $\mu=25$ & 1.383 & 0.134 & \textbf{0.0166} & 4.441 & 3.485 & \textbf{2.711}\\
&  & $\mu=125$ & 3.127 & 0.182 & \textbf{0.0155} & 4.937 & 3.536 & \textbf{2.715}\\
\hline
\end{tabular}
\end{table}

\begin{table}[h]
\small
\centering\caption{The KKT residual and running time for all methods implemented on the temperature control problem. For the Schwarz and FOTD, the smallest KKT residual and least running time among different setups of $b$ are highlighted for each setup of $\mu$. For ADMM, the best results among different $\mu$ are highlighted.}\label{tab:5}
\begin{tabular}{ |c|cc|ccc|ccc| }
\hline
Type &  \multicolumn{2}{c|}{Method} & \multicolumn{3}{c|}{KKT residual $(10^{-7})$} & \multicolumn{3}{c|}{Time (sec.)} \\
\hline
Centralized & \multicolumn{2}{c|}{IPOPT} & \multicolumn{3}{c|}{0.889} & \multicolumn{3}{c|}{16.934}\\
\hline
\multirow{17}{*}{Decomposed} & \multicolumn{2}{c|}{MultiShoot} &  \multicolumn{3}{c|}{3937.071} & \multicolumn{3}{c|}{104.021}\\
\cline{2-9}
& \multicolumn{2}{c|}{ILQR} &  \multicolumn{3}{c|}{2941.207} & \multicolumn{3}{c|}{89.655}\\
\cline{2-9}
& \multicolumn{2}{c|}{IDDP} &  \multicolumn{3}{c|}{2619.498} & \multicolumn{3}{c|}{93.832}\\
\cline{2-9}
& \multicolumn{2}{c|}{\multirow{2}{*}{ADMM}} & $\mu=1$ & $\mu=25$ & $\mu=125$ &  $\mu=1$ & $\mu=25$ & $\mu=125$\\
\cline{4-9}
& & & \textbf{1680.350 }& 11684.271 & 61455.896 & \textbf{76.732} & 224.930 & 479.194\\
\cline{2-9}
& & & $b=1$ & $b=5$ & $b=25$ & $b=1$ & $b=5$ & $b=25$\\
\cline{4-9}
& \multirow{3}{*}{FOTD (sparse LU)} & $\mu=1$ & 13.621 & 13.488 & \textbf{8.050} & 16.914 & \textbf{15.134} & 25.531\\
&  & $\mu=25$ & 28.266 & 8.024 & \textbf{5.316} & 19.973 & \textbf{18.023} & 28.752\\
&  & $\mu=125$ & 11.448 & 3.112 & \textbf{2.712} & 19.416 & \textbf{18.356} & 28.841\\
\cline{2-9}
& \multirow{3}{*}{FOTD (GMRES)} & $\mu=1$ & 17.837 & 15.228 & \textbf{2.368} & \textbf{16.653} & 16.729 & 21.293\\
&  & $\mu=25$ & 61.917 & 6.623 & \textbf{0.573} & 21.023 & \textbf{17.350} & 22.830\\
&  & $\mu=125$ & 15.052 & 2.937 & \textbf{1.704} & 23.336 & \textbf{17.871} & 23.571\\
\cline{2-9}

& \multirow{3}{*}{FOTD (IDR)} & $\mu=1$ & 14.934 & 13.515 & \textbf{3.998} & \textbf{17.106} & 17.425 & 21.393\\
&  & $\mu=25$ & 33.351 & \textbf{6.954} & 8.766 & 18.459 & \textbf{16.515} & 21.888\\
&  & $\mu=125$ & 28.830 & 8.553 & \textbf{7.119} & 20.244 & \textbf{17.318} & 22.750\\
\cline{2-9}
& \multirow{3}{*}{Schwarz} & $\mu=1$ & 10.995 & 8.750 & \textbf{1.303} & 23.873 & \textbf{21.155} & 21.724\\
&  & $\mu=25$ & 12.976 & 4.260 & \textbf{2.991} & 24.503 & \textbf{22.410} & 23.306\\
&  & $\mu=125$ & 5.124 & 3.501 & \textbf{1.996} & 23.506 & \textbf{20.431} & 22.359\\
\hline
\end{tabular}
\end{table}

Third, we report the KKT residual and running time for all methods. We average the results~over the convergent runs among five runs (corresponding to five initializations). Since ADMM converges within the budget for $\mu=1$ only, we report its results under this setup. The results for Cases 1, 2, and 3 are summarized in Tables \ref{tab:2}, \ref{tab:3}, and \ref{tab:4}, respectively. From the tables, we have the following observations. (i) The proposed FOTD and Schwarz outperform other parallel methods such as the multiple shooting, ILQR, and ADMM, among which ADMM has the worst performance. We~believe the reasons are two folds. First, the Schwarz \cite{Na2022Convergence} and FOTD employ an overlapping decomposition. The overlaps facilitate the information exchange between subproblems, and effectively suppress the system perturbations brought by the horizon truncation and different choices of $\mu$. Second, as also observed in \cite[Figure 6]{Na2022Convergence}, the ADMM iterates often generate a small stepsize, which is less~effective than the stepsize that is selected by the line search. As the augmented Lagrangian method, ADMM also suffers when it is initialized with a poor penalty parameter and/or poor~Lagrange~multipliers, especially for nonconvex problems \citep{Curtis2014adaptive}. Note that our objective coefficient for the control variables in \eqref{pro:5a} is $-C_2$ with $C_2>0$; thus, \eqref{pro:5}~is~nonconvex even if we have a linear system in the constraints. 
(ii) With three different (but efficient) QP solvers, FOTD performs equally well, although GMRES has slightly longer running time (within 0.3 sec.) than the other two solvers. Thus, different~linear system solvers can be employed~in FOTD to accelerate its computation. (iii) Between Schwarz~and FOTD, the Schwarz tends to attain a smaller KKT residual than FOTD, which is more significant when the overlap size $b$ is as small as $1$. For the majority of cases, both methods attain the smallest KKT residual when $b$ is as large as 25,~while attain the largest KKT residual when $b$ is as~small as 1. As for the running time, FOTD consistently converges faster than the Schwarz for different~choices of $\mu$ and $b$. The running time of FOTD is comparable to that of the centralized solver IPOPT. For the three choices~of $b$, the Schwarz tends to converge faster for a large $b$ than for moderate or~small $b$. This is because that the Schwarz performs less iterations when $b$ is large even if solving each~subproblem is also~more~\mbox{expensive. On the contrary, FOTD converges faster for a moderate $b$, which} reveals the trade-off between the total number of iterations and the computation cost of a single iteration. Overall, our experiments demonstrate that both FOTD and Schwarz are superior~\mbox{parallel} methods and robust to the parameter $\mu$. FOTD is more efficient than the Schwarz, and is as efficient as the popular solver IPOPT. However, as the parallel method, FOTD (and Schwarz) offers more flexibility to the computing environments and can be applied when a single high-speed processor is not accessible.

\vskip4pt
\noindent\textbf{Thin plate temperature control:} We apply FOTD on a thin plate temperature control problem studied in \cite{Na2022Convergence}. We refer to \cite{MathWorks2022Nonlinear, Buikis1999mathematical, Buikis1999mathematicala} for the physical background of the problem. In particular, we let $k\in[0, 1]$ be the continuous time index and $\bomega \in \Omega = [0,1]\times [0,1]$ be the position in the domain~$\Omega$. Then, we consider the following problem
\begin{subequations}\label{pro:6}
\begin{align}
\min_{x, u}\;\; & \int_0^1\int_{\bomega\in \Omega} \cbr{(\bx(\bomega,k) - d(\bomega,k))^2 + u(\bomega,k)^2} d\bomega\; dk, \label{pro:6a}\\[3pt]
\text{s.t.}\;\; & \frac{\partial x(\bomega, k)}{\partial k} =  \nabla^2x(\bomega,k) + u(\bomega,k) + \frac{2h_c}{\kappa_c t_c}(T_c - x(\bomega,k)) + \frac{2\epsilon_c\sigma_c}{\kappa_c t_c}(T_c^4 - x(\bomega,k)^4), \label{pro:6b}\\[3pt]
& \hskip10cm  \forall k \in [0, 1], \forall\bomega\in\Omega,  \nonumber \\
& x(\bomega, k) = 0, \quad \forall (\bomega,k)\in \Omega\times\{0\} \text{ or } \partial\Omega\times [0,1], \label{pro:6c}
\end{align}
\end{subequations}\vskip-0.1cm
\noindent where $\nabla^2$ is the Laplace operator, that is, $\nabla^2x = \frac{\partial^2 x}{\partial \bomega_1^2} + \frac{\partial^2 x}{\partial \bomega_2^2}$, and $d(\bomega,k)$ in the objective \eqref{pro:6a}~is~the prespecified desired temperature. Here, the PDE constraints in \eqref{pro:6b} are governed by a controlled heat equation (i.e., the first two terms) with extra convection and radiation terms (i.e., the third~and fourth terms). All the symbols with a subscript ``c" are the prespecified constants. In particular,~$h_c$~is the convection coefficient; $\kappa_c$ is the thermal conductivity; $\epsilon_c$ is the emissivity coefficient; $\sigma_c$ is the Stefan-Boltzmann constant; $T_c$ is the ambient temperature; and $t_c$ is the plate thickness. We unify the coefficients of heat equation for simplicity. The initial and boundary conditions are in \eqref{pro:6c}.

In our implementation, we discretize $\Omega$ by a $4\times 4$ mesh grid, i.e. $2\times 2$ mesh grid in the interior. The temporal horizon is decomposed by 5000 evenly spaced knots with $50$ knots for each subproblem. The setups of all methods are as before, and we follow \cite{MathWorks2022Nonlinear} to set up the problem parameters. In particular, we let $h_c=1$, $\kappa_c=400$, $\epsilon_c = 0.5$, $\sigma_c=5.67\times 10^{-8}$, $T_c=300$, $t_c=0.01$, and let~$d(\bomega,k) = \sin(k)$. The KKT residual and running time of the methods are summarized in Table \ref{tab:5}. From~the table, we again observe that the Schwarz and FOTD outperform~other four parallel methods.~The Schwarz attains smaller KKT residual than FOTD, while FOTD~converges faster than the Schwarz. Overall, our experiment shows the superiority of the overlapping decomposition-based methods.

\section{Conclusion}\label{sec:8}

This paper proposes a fast overlapping temporal decomposition (FOTD) procedure for solving long-horizon NLDPs in \eqref{pro:1}. FOTD relies on the sequential quadratic programming (SQP) and incorporates SQP with OTD technique. We establish global convergence and~uniform, local linear convergence for FOTD. The local result matches \cite{Na2022Convergence}, while FOTD~requires fewer computations in each iteration (cf. Theorem \ref{thm:7}).

Considering the improvement of the performance of the Schwarz scheme over ADMM \cite{Na2022Convergence}, and the relation between FOTD and the Schwarz, we believe the extension of FOTD is worth studying. One of the drawbacks of FOTD is the separation of modifying the Hessian matrix and solving the linear-quadratic subproblems (steps 1 and 2 in \cref{sec:3}), which leads to two separate factorizations for the subproblem matrices. Such a drawback does not enlarge the flops order of a single processor, but indeed results in a suboptimal flop multiplier. A more desirable algorithm should perform the Hessian modification and solve the subproblem in a single machine jointly, with one factorization, such as parallel quasi-Newton scheme \cite{Byrd1988Parallel}. Further, we can embed OTD into more advanced SQP frameworks, such as the trust region-SQP. We can also replace the line search step by the filter~step, and study the behavior of the approximate direction obtained by OTD on the filter step.

In addition, FOTD can be applied on graph-structured problems, seeing that a similar exponential decay of sensitivity for graph-structured problems was established in \cite{Shin2022Exponential}. Finally, a reasonable conjecture for the improved performance of the Schwarz scheme and FOTD over ADMM (and~other parallel methods) is the~lack of information exchange among subproblems in ADMM. No overlaps are adopted in ADMM. Thus, whether we can embed OTD into ADMM to improve the performance 
of ADMM, and whether the OTD-based ADMM exhibits a similar convergence rate as~FOTD are interesting future research directions.

\section*{Acknowledgments.} 
This material was based upon work supported by the U.S. Department of Energy, Office of Science, Office of Advanced Scientific Computing Research (ASCR) under Contract DE-AC02-06CH11347 and by NSF through award CNS-1545046.


\begin{APPENDICES}

\section{Proofs of results in \cref{sec:2}}\label{appendix:1}

\subsection{Proof of Theorem \ref{thm:1}}\label{pf:thm:1}

We start by writing the KKT conditions of Problem \eqref{pro:1} and Problem \eqref{pro:2}. For $k\in[N-1]$, we let $A_k(\bz_k) = \nabla_{\bx_k}^Tf_k(\bz_k)$, $B_k(\bz_k) = \nabla_{\bu_k}^Tf_k(\bz_k)$ be the Jacobian matrices. Then the Lagrange function of \eqref{pro:1} is
\begin{equation}\label{equ:full:Lagrange}
\mL(\bz,\blambda) = \sum_{k=0}^{N-1} \{g_k(\bz_k) + \blambda_{k}^T\bx_k - \blambda_{k+1}^Tf_k(\bz_k)\}+ \{g_N(\bx_N) + \blambda_{N}^T\bx_N\} - \blambda_{0}^T\bbx_0.
\end{equation}
Thus, the KKT conditions of \eqref{pro:1} are
\begin{subequations}\label{equ:full:KKT}
\begin{align}
\nabla_{\bx_k}g_k(\bz_k) + \blambda_{k}  - A_k^T(\bz_k)\blambda_{k+1} =\; &\0, \quad  \forall k \in [N-1], \label{pequ:full:KKTa}\\
\nabla_{\bu_k}g_k(\bz_k) - B_k^T(\bz_k)\blambda_{k+1}  =\; &\0, \quad  \forall k \in [N-1], \label{pequ:full:KKTb}\\
\nabla_{\bx_N}g_N(\bz_N) + \blambda_{N}  =\; & \0, \label{pequ:full:KKTc}\\
\bx_{k+1} - f_k(\bz_k)  = \; & \0, \quad  \forall k\in[N-1], \label{pequ:full:KKTd}\\
\bx_0 - \bbx_0  = \; & \0. \label{pequ:full:KKTe}
\end{align}
\end{subequations}
Similarly, the KKT conditions of $\P^i_\mu(\bd_i)$ in \eqref{pro:2} are
\begin{subequations}\label{equ:subprob:KKT}
\begin{align}
\nabla_{\bx_k}g_k(\tbz_{i, k}) + \tblambda_{i, k} - A_k^T(\tbz_{i, k})\tblambda_{i, k+1} = \; & \0, \quad  \forall k \in [m_1, m_2), \label{equ:subprob:KKTa}\\
\nabla_{\bu_k}g_k(\tbz_{i, k}) - B_k^T(\tbz_{i, k})\tblambda_{i, k+1} = \; & \0, \quad  \forall k \in [m_1, m_2), \label{equ:subprob:KKTb}\\
\nabla_{\bx_{m_2}}g_{m_2}(\tbx_{i, m_2}, \bbu_{m_2}) + \tblambda_{i, m_2} - A_{m_2}^T(\tbx_{i, m_2}, \bbu_{m_2})\bblambda_{m_2+1} \nonumber\\
 + \mu(\tbx_{i, m_2} - \bbx_{m_2})=\; &\0, \label{equ:subprob:KKTc}\\
\tbx_{i, k+1} - f_k(\tbz_{i, k}) =\; & \0, \quad  \forall k\in[m_1, m_2), \label{equ:subprob:KKTd}\\
\tbx_{i, m_1} - \bbx_{m_1} =\; &\0. \label{equ:subprob:KKTe}
\end{align}
\end{subequations}
\noindent\textbf{(i)}. By Definition \ref{def:1}, $\mD_i(\tz, \tlambda) = (\tx_{m_1:m_2}, \tu_{m_1:m_2-1}, \tlambda_{m_1:m_2})$. Letting $\bd_i = \bd_i^\star$ in \eqref{equ:subprob:KKTc}, \eqref{equ:subprob:KKTe}, we see \eqref{equ:subprob:KKT} is a subsystem of \eqref{equ:full:KKT} so that $\mD_i(\tz, \tlambda)$ satisfies \eqref{equ:subprob:KKT}. Thus, the statement holds.

\noindent\textbf{(ii)}. Suppose $(\tbz_i, \tblambda_i) = (\tbzi(\bd_i), \tblambdai(\bd_i))$ satisfies conditions \eqref{equ:subprob:KKT}. Since $[n_i, n_{i+1})\subseteq[m_1, m_2)$, we know from \eqref{equ:subprob:KKTa}, \eqref{equ:subprob:KKTb}, \eqref{equ:subprob:KKTd} that $(\tbz_{i, n_i:n_{i+1}-1}, \tblambda_{i, n_i:n_{i+1}-1})$ satisfies
\begin{subequations}\label{pequ:1}
\begin{align}
\nabla_{\bx_k}g_k(\tbz_{i, k}) + \tblambda_{i, k} - A_k^T(\tbz_{i, k})\tblambda_{i, k+1} = \; & \0, \quad \forall k\in[n_i, n_{i+1}), \label{pequ:1a}\\
\nabla_{\bu_k}g_k(\tbz_{i, k}) - B_k^T(\tbz_{i, k})\tblambda_{i, k+1} = \; & \0, \quad \forall k\in[n_i, n_{i+1}), \label{pequ:1b}\\
\tbx_{i, k+1} - f_k(\tbz_{i, k}) = \; & \0, \quad \forall k\in[n_i, n_{i+1}), \label{pequ:1c}
\end{align}
\end{subequations}
which is a subset of \eqref{pequ:full:KKTa}, \eqref{pequ:full:KKTb}, \eqref{pequ:full:KKTd}. We consider the composed point $(\bz, \blambda) = \mC(\{(\tbz_i, \tblambda_i)\}_i)$. By Definition \ref{def:1}, we know
\begin{equation}\label{pequ:con}
\begin{aligned}
&\bz_k = \tbz_{i, k}, \quad \text{if } k \in[n_i, n_{i+1}) \text{ for some } i \in [M-1], \quad \quad \bz_{N} = \tbx_{M-1, N},\\
&\blambda_k = \tblambda_{i, k}, \quad \text{if } k \in[n_i, n_{i+1}) \text{ for some } i \in [M-1], \quad \quad \blambda_{N} = \tblambda_{M-1, N}.
\end{aligned}
\end{equation}
Thus, \eqref{pequ:full:KKTc} is implied by \eqref{equ:subprob:KKTc} for $i = M-1$, and \eqref{pequ:full:KKTe} is implied by \eqref{equ:subprob:KKTe} for $i = 0$. For~\eqref{pequ:full:KKTb}, we use the decomposition $[N-1] = \cup_{i=0}^{M-1}[n_i, n_{i+1})$. For any $k\in[N-1]$, we have two cases.

\noindent\textbf{(a)}. $k \neq n_{i+1}-1$, $\forall i\in[M-2]$. Then, $k\in[n_i, n_{i+1}-1)$ for some $i\in[M-2]$, or $k\in [n_{M-1}, n_M)$. Thus, $k+1\in[n_i+1, n_{i+1})$ for $i\in[M-2]$ or $k+1 \in [n_{M-1}+1, N]$. By \eqref{pequ:con}, we know for both cases that $\bz_k$ and $\blambda_{k+1}$ are from the same subproblem, i.e. $\bz_k = \tbz_{i, k}$ and $\blambda_{k+1} = \tblambda_{i, k+1}$. Thus, \eqref{pequ:full:KKTb} is implied by \eqref{pequ:1b}.

\noindent\textbf{(b)}. $k = n_{i+1}-1$ for $i\in[M-2]$. Then, $k + 1 = n_{i+1}$ and thus $\blambda_{k+1} = \tblambda_{i+1, k+1}$. Comparing \eqref{pequ:full:KKTb} with \eqref{pequ:1b}, we know that
\begin{align}\label{pequ:e:1}
\nabla_{\bu_k}g_k(\bz_k) - B_k^T(\bz_k)\blambda_{k+1} = \0 &\Longleftrightarrow \nabla_{\bu_k}g_k(\tbz_{i, k}) - B_k^T(\tbz_{i, k})\tblambda_{i+1, k+1} =\0 \nonumber\\
&\stackrel{\eqref{pequ:1b}}{\Longleftrightarrow} B_k^T(\tbz_{i, k})\tblambda_{i+1, k+1} = B_k^T(\tbz_{i, k})\tblambda_{i, k+1}.
\end{align}
Following the same derivation, we consider \eqref{pequ:full:KKTa} and can easily get
\begin{equation}\label{pequ:e:2}
\nabla_{\bx_k}g_k(\bz_k)+\blambda_k - A_k^T(\bz_k)\blambda_{k+1} = \0 \stackrel{\eqref{pequ:1a}}{\Longleftrightarrow} A_k^T(\tbz_{i, k})\tblambda_{i+1, k+1} = A_k^T(\tbz_{i, k})\tblambda_{i, k+1},
\end{equation}
if $k = n_{i+1}-1$ for $i\in[M-2]$. Finally, we consider \eqref{pequ:full:KKTd}, where we have two cases.

\noindent\textbf{(a)}. $k \neq n_{i+1}-1$, $\forall i\in[M-2]$. As before, $\bz_k$ and $\bx_{k+1}$ are from the same subproblem, so that \eqref{pequ:full:KKTd} is implied by \eqref{pequ:1c}.

\noindent\textbf{(b)}. $k = n_{i+1}-1$ for $i\in[M-2]$. Then, $\bx_{k+1} = \tbx_{i+1, k+1}$. Comparing \eqref{pequ:full:KKTd} with \eqref{pequ:1c}, we know
\begin{equation}\label{pequ:e:3}
\bx_{k+1} = f_k(\bz_k) \Longleftrightarrow \tbx_{i+1, k+1} = f_k(\tbz_{i, k}) \stackrel{\eqref{pequ:1c}}{\Longleftrightarrow} \tbx_{i+1, k+1} = \tbx_{i, k+1}.
\end{equation}
Combining \eqref{pequ:e:1}, \eqref{pequ:e:2}, and \eqref{pequ:e:3}, we complete the proof.

\section{Proofs of results in \cref{sec:4}}

\subsection{Proof of Lemma \ref{lem:0}}\label{pf:lem:0}

We suppress the iteration index $\tau$ since the result holds uniformly over~$\tau$. By Assumption \ref{ass:2}, we know that for any integer $k$, there exists an integer $t_k\in[1,t]$ such that $\Xi_{k, t_k}\Xi_{k, t_k}^T\succeq \gamma_{C}I$. Let us define the knots recursively by $k_{j+1} = k_j + t_{k_{j}}$, $j = 0,1,\ldots,J-1$ with $k_0 = 0$. We suppose $J$ is large enough such that $k_J>N-t$. Thus, we derive a horizon decomposition $[N-1] = \cup_{j= 0}^{J-1}[k_j, k_{j+1}-1]\cup [k_{J}, N-1]$. Since $G = \nabla^T_{\bz} f$, we can apply the definition of $f$ in \eqref{equ:def:gf}~and write $GG^T$ explicitly. We have
\begin{equation}\label{equ:GG}
GG^T = \left(\begin{smallmatrix}
I & -A_0^T\\
-A_0 & \substack{I + A_0A_0^T + B_0B_0^T} & \ddots \\
& \ddots & \ddots & - A_{N-1}^T\\
& & -A_{N-1} & \substack{I + A_{N-1}A_{N-1}^T + B_{N-1}B_{N-1}^T}
\end{smallmatrix}\right).
\end{equation}
We let $k_j' = k_{j+1}-1 = k_j+t_{k_j}-1$, and define matrices $\{\T_j\}_{j\in[J]}$ corresponding to each interval~$[k_j, k_j']$ as ($\T_J$ is defined similarly, except that the bottom corner block is $I + A_{N-1}A_{N-1}^T + B_{N-1}B_{N-1}^T$)
\begin{equation*}
\T_j = \left(\begin{smallmatrix}
I & -A_{k_j}^T\\
-A_{k_j} & \substack{I + A_{k_j}A_{k_j}^T + B_{k_j}B_{k_j}^T} & \ddots \\
& \ddots & \ddots & - A_{k_j'}^T\\
& & -A_{k_j'} & \substack{A_{k_j'}A_{k_j'}^T + B_{k_j'}B_{k_j'}^T}
\end{smallmatrix}\right), \quad \forall j\in[J-1].
\end{equation*}
By the expression in \eqref{equ:GG}, it suffices to show that each $\T_j$ is lower bounded away from zero. Then, we have $\lambda_{\min}(GG^T)\geq \min_{j\in[J]}\lambda_{\min}(\T_j)$. Let us consider $\T_J$ first. We let $A_{m}^n = A_{m}A_{m-1}\cdots A_n$ and define a matrix $\T_J^{1/2}$ as
\begin{align*}
\T_J^{1/2} \coloneqq & \left(\begin{smallmatrix}
I\\
-A_{k_J} & I & & & & -B_{k_J}\\
& \ddots & \ddots & & & &  \ddots\\
& & & I & & & & -B_{N-2}\\
& & & -A_{N-1} & I & & & & -B_{N-1}
\end{smallmatrix}\right) \\
=&  \left(\begin{smallmatrix}
I\\
-A_{k_J} & I \\
&  & \ddots \\
& & & I \\
& & & & I 
\end{smallmatrix}\right)\times\cdots \times\left(\begin{smallmatrix}
I\\
& I \\
&  & \ddots \\
& & & I \\
& & & - A_{N-1}& I 
\end{smallmatrix}\right) \times\left(\begin{smallmatrix}
I\\
& I & & & & -B_{k_J}\\
& & \ddots & & & \vdots&  \ddots\\
& & & I & & - A_{N-2}^{k_J+1}B_{k_J}& & -B_{N-2}\\
& & & & I & -A_{N-1}^{k_J+1}B_{k_J}& & -A_{N-1}B_{N-2}& -B_{N-1}
\end{smallmatrix}\right)\\
\eqqcolon& P_{k_J}\cdots P_{N-1}Q_J.
\end{align*}
Then, we use the fact that $Q_JQ_J^T\succeq I$ and have
\begin{align*}
\T_J = \T_{J}^{1/2}(\T_J^{1/2})^T = P_{k_J}\cdots P_{N-1}Q_JQ_J^TP_{N-1}^T\cdots P_{k_J}^T \succeq P_{k_J}\cdots P_{N-1}P_{N-1}^T\cdots P_{k_J}^T. 
\end{align*}
Since $\|A_k\|\leq \Upsilon_{upper}$ for $k\in[k_J, N-1]$, we know $ \|P_k^{-1}\| \leq 1+\Upsilon_{upper}$. Thus, $P_{k}P_k^T\succeq 1/(1+\Upsilon_{upper})^2I$. By the above display and using $N-k_J \leq t$, we have $\lambda_{\min}(\T_J)\succeq 1/(1+\Upsilon_{upper})^{2t}$. We then consider $\T_j$ for $j\in[J-1]$ similarly. By the same matrix multiplication, we have
\begin{equation*}
\T_j^{1/2} \coloneqq  \left(\begin{smallmatrix}
I\\
-A_{k_j} & I  & & & -B_{k_j}\\
& \ddots & \ddots &  & &  \ddots\\
& & & I & & &  -B_{k_j'-1}\\
& & & -A_{k_j'}  & & & & -B_{k_j'}
\end{smallmatrix} \right) =  P_{k_j}\cdots P_{k_j'}Q_j,
\end{equation*}
where (slightly different from $Q_J$)
\begin{equation*}
Q_j = \left(\begin{smallmatrix}
I\\
& I & & &  -B_{k_j}\\
&  & \ddots  & & \vdots&  \ddots\\
& & & I  & - A_{k_j'-1}^{k_j+1}B_{k_j}& & -B_{k_j'-1}\\
& & & & -A_{k_j'}^{k_j+1}B_{k_j}& & -A_{k_j'}B_{k_j'-1}& -B_{k_j'}
\end{smallmatrix} \right)\eqqcolon\begin{pmatrix}
I & Q_{j,1}\\
\0 & Q_{j,2}
\end{pmatrix}.
\end{equation*}
Since $Q_{j,2}$ is just the controllability matrix $\Xi_{k_j, t_{k_j}}$ (except that the components are in a reverse order), Assumption \ref{ass:2} implies that $Q_{j,2}Q_{j,2}^T\succeq \gamma_{C} I$. Furthermore, since $\max\{\|A_k\|,\; \|B_k\|\}\leq \Upsilon_{upper}$, $\forall k\in[k_j, k_j'-1]$,~we get
\begin{equation*}
\max\{\|Q_{j, 2}\|,\;\|Q_{j,1}\|\}\leq \Upsilon_{upper} + \Upsilon_{upper}^2 + \cdots\Upsilon_{upper}^{k_j'-k_j+1} \leq \sum_{i=1}^{t}\Upsilon_{upper}^i \leq \frac{\Upsilon_{upper}^{t+1}}{\Upsilon_{upper}-1},
\end{equation*}
where the second inequality is due to $t_{k_j}\leq t$. Finally, noting that
\begin{equation*}
Q_jQ_j^T = \begin{pmatrix}
I + Q_{j,1}Q_{j,1}^T & Q_{j, 1}Q_{j, 2}^T\\
Q_{j,2}Q_{j,1}^T & Q_{j,2}Q_{j,2}^T
\end{pmatrix},
\end{equation*}
we apply the algebra result in \cite[Lemma 4.8(ii)]{Na2020Exponential} and obtain
\begin{align*}
Q_jQ_j^T\succeq \rbr{\frac{\gamma_{C}}{\gamma_{C} + \frac{\Upsilon_{upper}^{t+1}}{\Upsilon_{upper}-1}} }^2\cdot \min\{1,\; \gamma_{C}\}\cdot I \eqqcolon \gamma_{Q}I.
\end{align*}
Thus, using $P_kP_k^T\succeq 1/(1+\Upsilon_{upper})^2 I$, we obtain
\begin{align*}
\T_j = \T_j^{1/2}(\T_j^{1/2})^T = P_{k_j}\cdots P_{k_j'}Q_jQ_j^TP_{k_j'}^T\cdots P_{k_j}^T \succeq \frac{\gamma_{Q}}{(1+\Upsilon_{upper})^{2t}}I.
\end{align*}
Since $\gamma_{Q}<1$, we let $\gamma_{G} = \gamma_{Q}/(1+\Upsilon_{upper})^{2t}$ and have $\lambda_{\min}(\T_j)\succeq \gamma_{G}$, $\forall j\in[J]$. This finishes the proof.

\subsection{Proof of Lemma \ref{lem:1}}\label{pf:lem:1}

We adapt the proof of \cite[Lemma 1 and Theorem 1]{Na2023Superconvergence}. We suppress the iteration index $\tau$. The reduced Hessian of $\LP_{\mu}^i$ is independent from $\bd_i$, which only affects linear terms. From \eqref{pro:4b}-\eqref{pro:4c}, we know that the Jacobian matrix of the constraints in $\LP_{\mu}^i$ has full row rank. Thus, it suffices to show that the reduced Hessian of $\LP_{\mu}^i$ is lower bounded by $\gamma_{RH} I$. Let
\begin{equation*}
\hH^{(i)} = \diag(\hH_{m_1}, \hH_{m_1+1}, \ldots, \hH_{m_2-1}, \hQ_{m_2} + \mu I)
\end{equation*}
be the Hessian of $\LP_{\mu}^i$. We only need to show that $\tbw_i^T\hH^{(i)}\tbw_i \geq \gamma_{RH} \|\tbw_i\|^2$ for any $\tbw_i = (\tbp_i, \tbq_i) \neq \0$ satisfying
\begin{subequations}\label{pequ:2}
\begin{align}
\tbp_{i, k+1} = & A_k\tbp_{i, k} + B_k\tbq_{i, k}, \quad k\in[m_1, m_2), \label{pequ:2a}\\
\tbp_{i, m_1} = & \0. \label{pequ:2b}
\end{align}
\end{subequations}
We take the last subproblem (i.e., $i=M-1$) as an example to illustrate the proof idea. For~$i=M-1$, we have $\mu = 0$. For any $\tbw_{M-1} = (\tbp_{M-1}, \tbq_{M-1}) \neq\0$ satisfying \eqref{pequ:2}, we extend $\tbw_{M-1}$ forward by filling with $\0$ to get a full-horizon vector $\bw$. That is~$\bw = (\0; \tbw_{M-1})$. We can verify that $\bw = (\bp, \bq) \in \{\bw: G\bw = \0, \bw \neq\0\}$. Therefore,
\begin{equation*}
\tbw_{M-1}^T\hH^{(M-1)}\tbw_{M-1} = \bw^T\hH\bw \geq \gamma_{RH}\|\bw\|^2 = \gamma_{RH}\|\tbw_{M-1}\|^2,
\end{equation*}
where the first and last equalities are due to the construction of $\bw$, and the middle inequality is~due to Assumption \ref{ass:1}. For $i\in [M-2]$, we consider the following two cases.

\noindent\underline{\bf Case 1}: $m_2 \geq N-t$. Given $\tbw_i = (\tbp_i, \tbq_i) \neq \0$ satisfying \eqref{pequ:2}, we can still extend it forward by filling with $\0$. However, extending backward with $\0$ will make the full vector $\bw$ outside the space $\{\bw: G\bw = \0\}$. Instead, we construct the following extension. We let $\bq_{m_2:N-1} = \0$, $\bp_{k+1} = A_k\bp_k$ for $k = [m_2, N)$. Thus,
\begin{equation}\label{pequ:cons:w}
\bw = (\0; \tbw_i; \0; \bp_{m_2+1}; \0; \bp_{m_2+2}\; \ldots; \0; \bp_N)\in \{\bw: G\bw = \0, \bw \neq\0\}.
\end{equation}
Moreover, by Assumption \ref{ass:3},
\begin{multline}\label{pequ:3}
\|\bp_{m_2+1:N}\|^2 = \sum_{k=m_2+1}^{N}\|\bp_k\|^2
\leq \sum_{j=1}^{N-m_2}\Upsilon_{upper}^{2j}\|\bp_{m_2}\|^2\\
= \frac{\Upsilon_{upper}^2(\Upsilon_{upper}^{2(N-m_2)} -  1)}{\Upsilon_{upper}^2-1}\|\bp_{m_2}\|^2 \leq \frac{\Upsilon_{upper}^{2t+2} - \Upsilon_{upper}^2}{\Upsilon_{upper}^2-1}\|\bp_{m_2}\|^2.
\end{multline}
The last inequality uses $N-m_2 \leq t$. Further,
\begin{align}\label{pequ:4}
\bw^T\hH\bw & \stackrel{\eqref{pequ:cons:w}}{=} \tbw_i^T\hH^{(i)}\tbw_i - \mu\|\bp_{m_2}\|^2 + \sum_{k=m_2+1}^{N}\bp_k^T\hQ_k\bp_k  \leq   \tbw_i^T\hH^{(i)}\tbw_i - \mu\|\bp_{m_2}\|^2 + \Upsilon_{upper}\|\bp_{m_2+1:N}\|^2 \nonumber\\
& \stackrel{\eqref{pequ:3}}{\leq}  \tbw_i^T\hH^{(i)}\tbw_i - \rbr{\mu - \frac{\Upsilon_{upper}(\Upsilon_{upper}^{2t+2} - \Upsilon_{upper}^2)}{\Upsilon_{upper}^2-1}}\|\bp_{m_2}\|^2,
\end{align}
where the second inequality is due to Assumption \ref{ass:3}. By Assumption \ref{ass:1}, $\bw^T\hH\bw \geq \gamma_{RH}\|\bw\|^2 \geq \gamma_{RH}\|\tbw_i\|^2$. Combining with \eqref{pequ:4}, we see that $\tbw_i^T\hH^i\tbw_i \geq \gamma_{RH} \|\tbw_i\|^2$ provided
\begin{equation}\label{pequ:5}
\mu \geq \frac{\Upsilon_{upper}(\Upsilon_{upper}^{2t+2} - \Upsilon_{upper}^2)}{\Upsilon_{upper}^2-1}.
\end{equation}
\noindent\underline{\bf Case 2}: $m_2<N-t$. In this case, using the construction of $\bw$ in \eqref{pequ:cons:w} will result in a bound on~$\mu$ that grows exponentially in $N$. Instead, we make use of controllability in Assumption \ref{ass:2}.~In~particular, we still let $(\bp_{0:m_1-1}, \bq_{0:m_1-1}) = (\0, \0)$ and $\bq_{m_2} = \0$. Then $\bp_{m_2+1} = A_{m_2}\bp_{m_2}$. Let $l = m_2+1$ and let $(\bp_{k}, \bq_k) = (\0, \0)$, $\forall k\geq l+t_l$. We now show how to evolve from $\bp_{l}$ to $\bp_{l + t_{l}}$. Applying \eqref{pequ:2a} recursively, we have
\begin{equation}\label{pequ:6}
\bp_{l+j} = \rbr{\prod_{h=1}^{j}A_{l+h-1}}\bp_{l} + \Xi_{l, j}\begin{pmatrix}
\bq_{l+j-1}\\
\vdots\\
\bq_{l}
\end{pmatrix}, \quad\quad \forall j\geq 1.
\end{equation}
Letting $j = t_l$ in \eqref{pequ:6}, we see that if we specify the control sequence $\bq_{l+t_l-1:l}$ by
\begin{equation}\label{pequ:7}
\bq_{l+t_l-1:l} = -\Xi_{l, t_l}^T(\Xi_{l, t_l}\Xi_{l, t_l}^T)^{-1}(\prod_{h=1}^{t_l}A_{l+h-1})\bp_l,
\end{equation}
and generate $\bp_{l+j}$ as \eqref{pequ:2a}, then we have $\bp_{l+t_l} = \0$. Thus, $\bw = (\bp, \bq)\in \{\bw: G\bw = \0, \bw\neq \0\}$.~Moreover, by Assumptions \ref{ass:2} and \ref{ass:3}, we obtain from \eqref{pequ:7} that
\begin{equation}\label{pequ:8}
\|\bq_{l+t_l-1:l}\| \leq \|\Xi_{l, t_l}^T(\Xi_{l, t_l}\Xi_{l, t_l}^T)^{-1}\|\cdot \Upsilon_{upper}^{t_l}\|\bp_l\| 
\leq \frac{\Upsilon_{upper}^{t_l+1}}{\sqrt{\gamma_C}}\|\bp_{m_2}\| \leq \frac{\Upsilon_{upper}^{t+1}}{\sqrt{\gamma_C}}\|\bp_{m_2}\|.
\end{equation}
The second inequality uses the fact that $\|A^T(AA^T)^{-1}\| = \|\Sigma^{-1}\|$ for any full row rank matrix $A$, and $U\Sigma V^T$ is its singular value decomposition. The last inequality is due to $t_l\leq t$ by Assumption \ref{ass:2}. Furthermore, for any $1\leq j\leq t_l-1$,
\begin{align}\label{pequ:9}
\|\bp_{l+j}\| \stackrel{\eqref{pequ:6}}{\leq} &  \Upsilon_{upper}^j\|\bp_l\| + \|\Xi_{l, j}\|\|\bq_{l+j-1:l}\| 
\leq  \Upsilon_{upper}^j\|\bp_l\| + \rbr{\sum_{h=1}^{j}\Upsilon_{upper}^h}\|\bq_{l+t_l-1:l}\| \nonumber\\
\stackrel{\eqref{pequ:8}}{\leq} & \Upsilon_{upper}^{j+1}\|\bp_{m_2}\| + \frac{\Upsilon_{upper}^{j+1}- \Upsilon_{upper}}{\Upsilon_{upper}-1}\cdot \frac{\Upsilon_{upper}^{t+1}}{\sqrt{\gamma_C}}\|\bp_{m_2}\|,
\end{align}
where the second inequality is due to $\|\Xi_{l, j}\|\leq \sum_{h=1}^j\Upsilon_{upper}^h$ by the definition of $\Xi_{l, j}$ in Assumption \ref{ass:2}. Without loss of generality, we suppose $\Upsilon_{upper}/2 \geq 1\geq \gamma_C$, then 
\begin{equation}\label{pequ:U}
\Upsilon_{upper} \leq \frac{\Upsilon_{upper}^{t+2}}{(\Upsilon_{upper}-1)\sqrt{\gamma_C}} \leq \frac{2\Upsilon_{upper}^{t+1}}{\sqrt{\gamma_C}}.
\end{equation}
Thus, \eqref{pequ:9} can be further simplified~as
\begin{multline*}
\|\bp_{l+j}\| \stackrel{\eqref{pequ:9}}{\leq} \Upsilon_{upper}^{j+1}\|\bp_{m_2}\| + \frac{\Upsilon_{upper}^{j+1}}{\Upsilon_{upper}-1}\cdot \frac{\Upsilon_{upper}^{t+1}}{\sqrt{\gamma_C}}\|\bp_{m_2}\| \\
\stackrel{\eqref{pequ:U}}{\leq} \frac{2\Upsilon_{upper}^{j+1}}{\Upsilon_{upper}-1}\cdot \frac{\Upsilon_{upper}^{t+1}}{\sqrt{\gamma_C}}\|\bp_{m_2}\| \stackrel{\eqref{pequ:U}}{\leq} 4\Upsilon_{upper}^j\cdot \frac{\Upsilon_{upper}^{t+1}}{\sqrt{\gamma_C}}\|\bp_{m_2}\|.
\end{multline*}
The above inequality also holds for $j=0$. Thus,
\begin{multline}\label{pequ:10}
\|\bp_{l:l+t_l-1}\|^2 = \sum_{j=0}^{t_l-1}\|\bp_{l+j}\|^2 \leq 16\sum_{j=0}^{t_l-1}\Upsilon_{upper}^{2j}\cdot  \frac{\Upsilon_{upper}^{2t+2}}{\gamma_C}\|\bp_{m_2}\|^2 \\
= \frac{16(\Upsilon_{upper}^{2t}-1)\Upsilon_{upper}^{2t+2}}{\gamma_C(\Upsilon_{upper}^2-1)}\|\bp_{m_2}\|^2 \leq \frac{31\Upsilon_{upper}^{4t}}{\gamma_C}\|\bp_{m_2}\|^2,
\end{multline}
where the last inequality uses $\Upsilon_{upper}^2/(\Upsilon_{upper}^2-1)\leq 31/16$ (as $\Upsilon_{upper}\geq 2$). Combining \eqref{pequ:10} with \eqref{pequ:8} and noting that $2t+2\leq 4t$, we get
\begin{equation*}
\|(\bp_{l:l+t_l-1}; \bq_{l:l+t_l-1})\|^2 \leq \frac{32\Upsilon_{upper}^{4t}}{\gamma_C} \|\bp_{m_2}\|^2.
\end{equation*}
Finally, by a similar derivation as \eqref{pequ:4}, we know $\tbw_i^T\hH^{(i)}\tbw_i \geq \gamma_{RH}\|\tbw_i\|^2$ provided $\mu \geq 32\Upsilon_{upper}^{4t+1}/\gamma_C$. This condition implies \eqref{pequ:5}, so we complete the proof.

\subsection{Proof of Corollary \ref{cor:1}}\label{pf:cor:1}

We note that all three conditions are independent of $\bd_i$, so the statement holds for any $\bd_i$. By Lemma~\ref{lem:1}, the reduced Hessian of $\LP_{\mu}^i(\bd_i)$ is	lower bounded by $\gamma_{RH} I$ for any $i\in[M-1]$. Thus, (i) holds. The controllability in Assumption \ref{ass:2} holds naturally for the subproblems with the same constants $(\gamma_{C}, t)$, as the dynamics of the subproblems are a~subset of the full problem. Thus, (ii) holds. For the upper boundedness condition, we note that the last square matrix of the objective of \eqref{pro:4a} is bounded by $\Upsilon_{upper}+\mu$, and other square matrices are~the same as the full problem \eqref{pro:3}. Thus, (iii) holds. This completes the proof.

\subsection{Proof of Theorem \ref{thm:2}}\label{pf:thm:2}

Our proof relies on the primal-dual sensitivity results of NLDPs \cite{Na2020Exponential, Na2022Convergence}. We omit the subproblem index $i$ in the proof. We define
\begin{align*}
\tg_k(\bp_k, \bq_k) = & \frac{1}{2}\begin{pmatrix}
\bp_k\\
\bq_k
\end{pmatrix}^T\begin{pmatrix}
\hQ_k & \hS_k^T\\
\hS_k & \hR_k
\end{pmatrix}\begin{pmatrix}
\bp_k\\
\bq_k
\end{pmatrix} + \begin{pmatrix}
\nabla_{\bx_k}\mL\\
\nabla_{\bu_k}\mL
\end{pmatrix}^T\begin{pmatrix}
\bp_k\\
\bq_k
\end{pmatrix},\quad \forall k\in[m_1, m_2),\\
\tg_{m_2}(\bp_{m_2}; \bd_{2:4}) = &\frac{1}{2}\begin{pmatrix}
\bp_{m_2}\\
\bd_3
\end{pmatrix}^T\begin{pmatrix}
\hQ_{m_2} & \hS_{m_2}^T\\
\hS_{m_2} &  \hR_{m_2}
\end{pmatrix}\begin{pmatrix}
\bp_{m_2}\\
\bd_3
\end{pmatrix} + \nabla_{\bx_{m_2}}^T\mL\bp_{m_2}  - \bd_4^TA_{m_2}\bp_{m_2} + \frac{\mu}{2}\|\bp_{m_2} - \bd_2\|^2,\\
\tf_k(\bp_k, \bq_k) = & A_k\bp_k + B_k\bq_k - \nabla_{\blambda_{k+1}}\mL, \quad k\in[m_1, m_2).
\end{align*}
Then, \eqref{pro:4} is rewritten as $\min \sum_{k=m_1}^{m_2-1} g_k(\bp_k, \bq_k) + \tg_{m_2}(\bp_{m_2}; \bd_{2:4})$, s.t. $\bp_{m_1} = \bd_{1}$, $\bp_{k+1} = \tf_k(\bp_k, \bq_k)$, $\forall k \in[m_1, m_2)$. By Lemma \ref{lem:1}, this problem has a unique solution $(\tbw^\star(\bd), \tbzeta^\star(\bd))$ for any $\bd$. Let us define a parameterized perturbation path from $\bd$ to $\bd'$:
\begin{align*}
\bd^{(1)}(\omega) = & (\bd_1; \bd_{2:4}) + \omega\rbr{\frac{\bd_1'-\bd_1}{\|\bd_1'-\bd_1\|}; \0}, \quad\quad \forall 0\leq \omega\leq \|\bd_1'-\bd_1\|,\\
\bd^{(2)}(\omega) = & (\bd'_1; \bd_{2:4}) + \omega\rbr{\0; \frac{\bd_{2:4}' - \bd_{2:4}}{\|\bd_{2:4}' - \bd_{2:4}\|}}, \quad \forall 0\leq \omega \leq \|\bd_{2:4}'-\bd_{2:4}\|,
\end{align*}
where we essentially first perturb $\bd_1$ and then perturb $\bd_{2:4}$. At $\bd^{(j)}(\omega)$ for $j=1,2$, we define the directional	derivatives of the solution trajectories as (similar for $D\tq_k, D\tzeta_k$)
\begin{equation*}
D\tp_k(\bd^{(j)}(\omega)) = \lim\limits_{\epsilon\searrow 0}\frac{\bp_k^\star(\bd^{(j)}(\omega+\epsilon)) - \bp_k^\star(\bd^{(j)}(\omega))}{\epsilon}, \quad\quad \forall k\in[m_1, m_2].
\end{equation*}	
The existence of the directional derivatives is ensured by Lemma \ref{lem:1} and \cite[Theorem 2.3]{Na2020Exponential}. By Corollary \ref{cor:1} and the fact that (by Assumption \ref{ass:3}) 
\begin{equation*}
\|\nabla_{\bp_{m_2}\bd_{2:4}}\tg_{m_2}\| = \|(\mu I\;\; S_{m_2}^T\;\; -A_{m_2}^T)\|\leq \mu + 2\Upsilon_{upper},
\end{equation*}
we know that \cite[Assumption 4.2]{Na2020Exponential} is satisfied. Thus, by \cite[Theorem 5.7]{Na2020Exponential},
\begin{equation}\label{pequ:11}
\|D\tp_k(\bd^{(1)}(\omega))\| \leq C'\rho^{k-m_1}, \quad \|D\tp_k(\bd^{(2)}(\omega))\|\leq C'\rho^{m_2-k}, \quad \forall k\in[m_1, m_2],
\end{equation}	
for constants $C'>0$ and $\rho\in(0, 1)$ depending on $\mu, \Upsilon_{upper}, \gamma_{RH}, \gamma_C$ only. By \cite[Theorem 5.7]{Na2020Exponential} and \cite[Theorem 5]{Na2022Convergence}, we know \eqref{pequ:11} holds for $D\tq_k(\bd^{(j)}(\omega))$ and $D\tzeta_{k}(\bd^{(j)}(\omega))$ as well. Furthermore,
\begin{align*}
\|\bp_{k}^\star(\bd) - \bp_{k}^\star(\bd')\| & = \|\bp_{k}^\star(\bd^{(1)}(0)) - \bp_{k}^\star(\bd^{(2)}(\|\bd'_{2:4} - \bd_{2:4}\|))\|\\
& \leq  \|\bp_{k}^\star(\bd^{(1)}(0)) - \bp_{k}^\star(\bd^{(1)}(\|\bd'_1 - \bd_1\|))\| + \|\bp_{k}^\star(\bd^{(2)}(0)) - \bp_{k}^\star(\bd^{(2)}(\|\bd'_{2:4} - \bd_{2:4}\|))\|\\
& = \nbr{\int_{0}^{\|\bd'_1 - \bd_1\|}D\tp_{k}(\bd^{(1)}(\omega))d\omega  } + \nbr{\int_{0}^{\|\bd'_{2:4} - \bd_{2:4}\|}D\tp_{k}(\bd^{(2)}(\omega))d\omega  }\\
& \stackrel{\mathclap{\eqref{pequ:11}}}{\leq}C'\rbr{\rho^{k-m_1}\|\bd'_1 - \bd_1\| + \rho^{m_2-k}\|\bd'_{2:4} - \bd_{2:4}\|}.
\end{align*}
The above inequality also holds for $\|\bq_k^\star(\bd) - \bq_k^\star(\bd')\|$, $\|\bzeta_{k}^\star(\bd) - \bzeta_{k}^\star(\bd')\|$. Since $\|\bw_k^\star(\bd) - \bw_k^\star(\bd')\| = \sqrt{\|\bp_{k}^\star(\bd) - \bp_{k}^\star(\bd')\|^2 + \|\bq_k^\star(\bd) - \bq_k^\star(\bd')\|^2}$, the proof is complete by redefining $C'\leftarrow \sqrt{2}C'$.

\subsection{Proof of Theorem \ref{thm:3}}\label{pf:thm:3}

Since $(\tDelta\bz, \tDelta\blambda) = \mC(\{(\tbwi(\bd_i), \tbzetai(\bd_i))\}_i)$ with $\bd_i = (\0;\0;\0;\0)$ ($\bd_{M-1} = \0$), by Definition \ref{def:1} we know that 
\begin{equation}\label{pequ:C1}
\begin{aligned}
&\tDelta\bz_k = \tbw_{i, k}^\star(\bd_i), \quad \text{if } k \in[n_i, n_{i+1}) \text{ for some } i \in [M-1], \quad \quad \tDelta\bz_N = \tbw_{M-1, N}^\star(\bd_{M-1}),\\
&\tDelta\blambda_k = \tbzeta_{i, k}^\star(\bd_i), \quad\; \text{if } k \in[n_i, n_{i+1}) \text{ for some } i \in [M-1], \quad \quad \tDelta\blambda_N = \tbzeta_{M-1, N}^\star(\bd_{M-1}).
\end{aligned}
\end{equation}
Applying Theorem \ref{thm:1}(i) on Problem \eqref{pro:3}, we know that if 
\begin{equation}\label{pequ:dstar}
\bd_i^\star = (\Delta\bx_{m_1}; \Delta\bx_{m_2};\Delta\bu_{m_2}; \Delta\blambda_{m_2+1}),
\end{equation}
then 
\begin{equation}\label{pequ:C2}
\begin{aligned}
&\Delta\bz_k = \tbw_{i, k}^\star(\bd_i^\star), \quad \text{if } k \in[n_i, n_{i+1}) \text{ for some } i \in [M-1], \quad \quad \Delta\bz_N = \tbw_{M-1, N}^\star(\bd_{M-1}^\star),\\
&\Delta\blambda_k = \tbzeta_{i, k}^\star(\bd_i^\star), \quad\; \text{if } k \in[n_i, n_{i+1}) \text{ for some } i \in [M-1], \quad \quad \Delta\blambda_N = \tbzeta_{M-1, N}^\star(\bd_{M-1}^\star).
\end{aligned}
\end{equation}
Moreover, we claim $\bd_{0,1}^\star = \Delta\bx_0 = \0$ for any	iteration $\tau$. In fact, from the input to Algorithm~\ref{alg:FOTD},~$\bx_0^0 = \bbx_0$. By \eqref{pro:3c}, $\Delta\bx_0^0 = -(\bx_0^0 - \bbx_0)= \0$. Furthermore, if $\bx_0^{\tau-1} = \bbx_0$ for $\tau\geq 1$, we use the fact that $\tDelta\bx_0^{\tau-1} = \bd_{0,1}^{\tau-1} = \0$ (the first equality is due to \eqref{pro:4c}; the second equality is due to the specification of the boundary variable), and obtain $\bx_0^\tau = \bx_0^{\tau-1}+\alpha_{\tau-1}\tDelta\bx^{\tau-1} = \bx_0^{\tau-1} = \bbx_0$. Thus, we have $\Delta\bx_0^{\tau} \stackrel{\eqref{pro:3c}}{=} - (\bx_0^\tau - \bbx_0) = \0$. Comparing \eqref{pequ:C1} and \eqref{pequ:C2} for each stage and applying Theorem \ref{thm:2}, we have
\begin{subequations}\label{pequ:12}
\begin{align}
\max\{ & \|\tDelta\bz_k - \Delta\bz_k\|^2,\;  \|\tDelta\blambda_k  - \Delta\blambda_k\|^2\} \nonumber\\[3pt]
& \leq  2(C')^2(\rho^{2(k-m_1)}\|\bd_{i, 1}^\star\|^2 +  \rho^{2(m_2-k)}\|\bd_{i, 2:4}^\star\|^2), \hskip 1cm \forall k\in[n_i, n_{i+1}), \; i\in[M-2], \label{pequ:12a}\\[3pt]
\max\{& \|\tDelta\bz_k  - \Delta\bz_k\|^2,\; \|\tDelta\blambda_k  - \Delta\blambda_k\|^2\} \leq (C')^2\rho^{2(k-m_1)}\|\bd_{i, 1}^\star\|^2,\hskip 1cm \forall k\in[n_{M-1}, n_M], \label{pequ:12b}
\end{align}
\end{subequations}
where \eqref{pequ:12a} holds for $i=0$ since $\bd_{0, 1}^\star = \0$ as we just claimed. Thus, we get
\begin{align*}
& \|\tDelta\bz - \Delta\bz\|^2 =  \sum_{i=0}^{M-2}\sum_{k = n_i}^{n_{i+1}-1}\|\tDelta\bz_k  - \Delta\bz_k\|^2 + \sum_{k=n_{M-1}}^{n_M}\|\tDelta\bz_k - \Delta\bz_k\|^2\\
& \stackrel{\mathclap{\eqref{pequ:12}}}{\leq}\;  2(C')^2\sum_{i=0}^{M-2}\rbr{\sum_{j=0}^{n_{i+1}-n_i-1}\rho^{2(b+j)}\|\bd_{i, 1}^\star\|^2 +  \sum_{j=1}^{n_{i+1}-n_i}\rho^{2(b+j)}\|\bd_{i,2:4}^\star\|^2} + (C')^2\sum_{j=0}^{n_M - n_{M-1}} \rho^{2(b+j)}\|\bd_{M-1}^\star\|^2\\
&\leq 2(C')^2\rho^{2b}\sum_{i=0}^{M-2}\sum_{j=0}^{\infty}\rho^{2j}\|\bd_i^\star\|^2 + (C')^2\rho^{2b}\sum_{j=0}^{\infty}\rho^{2j}\|\bd_{M-1}^\star\|^2\\
& \leq  \frac{2(C')^2\rho^{2b}}{1-\rho^2}\sum_{i=0}^{M-1}\|\bd_i^\star\|^2 \stackrel{\eqref{pequ:dstar}}{\leq} \frac{2(C')^2\rho^{2b}}{1-\rho^2}\|(\Delta\bz, \Delta\blambda)\|^2.
\end{align*}
The above derivation also holds for $\|\tDelta\blambda - \Delta\blambda\|^2$. Thus,
\begin{equation*}
\|(\tDelta\bz - \Delta\bz, \tDelta\blambda - \Delta\blambda)\|^2 \leq \frac{4(C')^2\rho^{2b}}{1-\rho^2}\|(\Delta\bz, \Delta\blambda)\|^2.
\end{equation*}
Letting $C=2C'/\sqrt{1-\rho^2}$, we complete the proof.

\section{Proofs of results in \cref{sec:5}}

\subsection{Proof of Lemma \ref{lem:2}}\label{pf:lem:2}

Recall from \eqref{equ:def:H} that $H(\bz, \blambda) = \diag(H_0,\ldots, H_N)$. Thus,
\begin{align*}
\|H(\bz, \blambda)\| \leq &  \max_{k\in[N]}\|H_k(\bz_k,\blambda_{k+1})\| \leq \Upsilon_{upper}
\end{align*}
for any $(\bz, \blambda)\in\mZ\times \Lambda$, where the last inequality is due to Assumption \ref{ass:4}. Furthermore, using the expression of $GG^T$ in \eqref{equ:GG}, we immediately have
\begin{equation*}
\|G(\bz)\| = \sqrt{\|GG^T\|} \leq \sqrt{1+\|A_k\|^2+\|B_k^2\| + 2\|A_k\|}\stackrel{\eqref{equ:upper:third:new}}{\leq} 1+2\Upsilon_{upper}.
\end{equation*}
Thus, we can let $\Upsilon_{HG} = 1+2\Upsilon_{upper}$ and the first part of the statement holds. Moreover, if $\|\hH_k\|\leq \Upsilon_{upper}$ by Assumption \ref{ass:3}, then $\|\hH\| = \|\diag(\hH_0, \ldots, \hH_N)\| \leq \Upsilon_{upper}$. Let $Z$ be defined in Assumption \ref{ass:1}. Then, noting that $G^T(GG^T)^{-1}G + ZZ^T = I$, we can verify that
\begin{equation*}
\mB \coloneqq \begin{pmatrix}
\hH & G^T\\
G & \0
\end{pmatrix}^{-1}
=
\begin{pmatrix}
\mB_1 & \mB_2^T \\
\mB_2 & \mB_3
\end{pmatrix},
\end{equation*}
where 
\begin{align*}
\mB_1 =& Z(Z^T\hH Z)^{-1}Z^T,\quad\quad \mB_2 = (GG^T)^{-1}G(I - \hH Z(Z^T\hH Z)^{-1}Z^T),\\
\mB_3 = & (GG^T)^{-1}G(\hH Z(Z^T\hH Z)^{-1}Z^T\hH - \hH)G^T(GG^T)^{-1}.
\end{align*}
Then, by Assumptions \ref{ass:1}-\ref{ass:3} and Lemma \ref{lem:0}, we have
\begin{align*}
\|\mB_1\| \leq & \frac{1}{\gamma_{RH}},\quad 
\|\mB_2\| \leq  \|(GG^T)^{-1}G\|(1 + \frac{\Upsilon_{upper}}{\gamma_{RH}}) \leq \frac{1}{\sqrt{\gamma_G}}(1 + \frac{\Upsilon_{upper}}{\gamma_{RH}}), \\
\|\mB_3\| \leq & \frac{1}{\gamma_G}(\Upsilon_{upper} + \frac{\Upsilon_{upper}^2}{\gamma_{RH}}).
\end{align*}
Noting that $\|\mB\| \leq \|\mB_1\| + 2\|\mB_2\| + \|\mB_3\|$, we complete the proof.

\subsection{Proof of Theorem \ref{thm:4}}\label{pf:thm:4}

We suppress the iteration index $\tau$. We know
\begin{equation}\label{pequ:13}
\begin{pmatrix}
\nabla_{\bz}\mL_{\eta}\\
\nabla_{\blambda}\mL_{\eta}
\end{pmatrix}^T\begin{pmatrix}
\tDelta\bz\\
\tDelta\blambda
\end{pmatrix} = \begin{pmatrix}
\nabla_{\bz}\mL_{\eta}\\
\nabla_{\blambda}\mL_{\eta}
\end{pmatrix}^T\begin{pmatrix}
\Delta\bz\\
\Delta\blambda
\end{pmatrix} + \begin{pmatrix}
\nabla_{\bz}\mL_{\eta}\\
\nabla_{\blambda}\mL_{\eta}
\end{pmatrix}^T\begin{pmatrix}
\tDelta\bz - \Delta\bz\\
\tDelta\blambda - \Delta\blambda
\end{pmatrix}.
\end{equation}
For the first term in \eqref{pequ:13},
\begin{align}\label{nnpequ:1}
\I_1 & \coloneqq \begin{pmatrix}
\nabla_{\bz}\mL_{\eta}\\
\nabla_{\blambda}\mL_{\eta}
\end{pmatrix}^T\begin{pmatrix}
\Delta\bz\\
\Delta\blambda
\end{pmatrix} \nonumber\\
& \stackrel{\mathclap{\eqref{equ:der:aug:L}}}{=} \; \begin{pmatrix}
\Delta\bz\\
\Delta\blambda
\end{pmatrix}^T\begin{pmatrix}
I + \eta_2H & \eta_1G^T\\
\eta_2G & I
\end{pmatrix}\begin{pmatrix}
\nabla_{\bz}\mL\\
\nabla_{\blambda}\mL
\end{pmatrix} \nonumber\\
& \stackrel{\mathclap{\eqref{equ:Newton}}}{=} - \begin{pmatrix}
\Delta\bz\\
\Delta\blambda
\end{pmatrix}^T\begin{pmatrix}
I + \eta_2H & \eta_1G^T\\
\eta_2G & I
\end{pmatrix}\begin{pmatrix}
\hH & G^T\\
G & \0
\end{pmatrix}\begin{pmatrix}
\Delta\bz\\
\Delta\blambda
\end{pmatrix} \nonumber\\
& = - \begin{pmatrix}
\Delta\bz\\
\Delta\blambda
\end{pmatrix}^T\begin{pmatrix}
(I+\eta_2H)\hH + \eta_1G^TG & (I+\eta_2H)G^T\\
G(I + \eta_2\hH) & \eta_2GG^T
\end{pmatrix}\begin{pmatrix}
\Delta\bz\\
\Delta\blambda
\end{pmatrix} \nonumber\\
&\stackrel{\mathclap{\eqref{equ:Newton}}}{=} - (\Delta\bz)^T\cbr{(I+\eta_2H)\hH + \frac{\eta_1}{2}G^TG}\Delta\bz -\frac{\eta_1}{2}\|\nabla_{\blambda}\mL\|^2 - \eta_2\|\hH\Delta\bz + \nabla_{\bz}\mL\|^2 \nonumber\\
&\quad \quad - (\Delta\blambda)^TG\{2I + \eta_2(\hH + H)\}\Delta\bz \nonumber\\
& = -\frac{\eta_1}{2}\|\nabla_{\blambda}\mL\|^2 - \frac{\eta_2}{2}\|\nabla_{\bz}\mL\|^2 - (\Delta\bz)^T\cbr{(I+\eta_2H)\hH + \frac{\eta_1}{2}G^TG}\Delta\bz \nonumber\\
&\quad \quad - (\Delta\blambda)^TG\{2I + \eta_2(\hH + H)\}\Delta\bz + \rbr{\frac{\eta_2}{2}\|\nabla_{\bz}\mL\|^2 - \eta_2\|\hH\Delta\bz + \nabla_{\bz}\mL\|^2}.
\end{align}
For the last term in the above equation,
\begin{align}\label{nnpequ:2}
\frac{\eta_2}{2}\|\nabla_{\bz}\mL\|^2 & - \eta_2\|\hH\Delta\bz + \nabla_{\bz}\mL\|^2 \nonumber\\
& = -\eta_2\|\hH\Delta\bz\|^2 - 2\eta_2(\Delta\bz)^T\hH\nabla_{\bz}\mL - \frac{\eta_2}{2}\|\nabla_{\bz}\mL\|^2 \nonumber\\
& \stackrel{\mathclap{\eqref{equ:Newton}}}{=} -\eta_2\|\hH\Delta\bz\|^2 + 2\eta_2(\Delta\bz)^T\hH(\hH\Delta\bz + G^T\Delta\blambda) - \frac{\eta_2}{2}\|\hH\Delta\bz + G^T\Delta\blambda\|^2 \nonumber\\
& = \eta_2\|\hH\Delta\bz\|^2 + \eta_2(\Delta\bz)^T\hH G^T\Delta\blambda - \frac{\eta_2}{2}\|\hH\Delta\bz\|^2 - \frac{\eta_2}{2}\|G^T\Delta\blambda\|^2 \nonumber\\
& \leq \eta_2\|\hH\Delta\bz\|^2 + \eta_2(\Delta\bz)^T\hH G^T\Delta\blambda - \frac{\eta_2}{2}\|G^T\Delta\blambda\|^2.
\end{align}
Combining the above two displays and supposing $\eta_1\geq \eta_2$ at the moment,	
\begin{align}\label{pequ:14}
\I_1 & \stackrel{\mathclap{\eqref{nnpequ:1}}}{\leq}  - \frac{\eta_2}{2}\|\nabla\mL\|^2 - (\Delta\bz)^T\cbr{(I+\eta_2H)\hH + \frac{\eta_1}{2}G^TG}\Delta\bz - (\Delta\blambda)^TG\{2I + \eta_2(\hH + H)\}\Delta\bz \nonumber \\[3pt]
& \quad + \rbr{\frac{\eta_2}{2}\|\nabla_{\bz}\mL\|^2 - \eta_2\|\hH\Delta\bz + \nabla_{\bz}\mL\|^2} \nonumber\\[3pt]
& \stackrel{\mathclap{\eqref{nnpequ:2}}}{\leq} - \frac{\eta_2}{2}\|\nabla\mL\|^2 - (\Delta\bz)^T\cbr{(I+\eta_2H)\hH + \frac{\eta_1}{2}G^TG}\Delta\bz - (\Delta\blambda)^TG\{2I + \eta_2(\hH + H)\}\Delta\bz \nonumber \\[3pt]
& \quad + \eta_2\|\hH\Delta\bz\|^2 + \eta_2(\Delta\bz)^T\hH G^T\Delta\blambda - \frac{\eta_2}{2}\|G^T\Delta\blambda\|^2 \nonumber \\[3pt]
&=  - \frac{\eta_2}{2}\|\nabla\mL\|^2 - (\Delta\bz)^T\cbr{(I+\eta_2(H-\hH))\hH + \frac{\eta_1}{2}G^TG}\Delta\bz - (\Delta\blambda)^TG(2I + \eta_2 H)\Delta\bz - \frac{\eta_2}{2}\|G^T\Delta\blambda\|^2 \nonumber\\[3pt]
& \leq  - \frac{\eta_2}{2}\|\nabla\mL\|^2 + 2\eta_2\Upsilon^2\|\Delta\bz\|^2 + 2\|\Delta\blambda\|\|G\Delta\bz\| + \eta_2\Upsilon\|G^T\Delta\blambda\|\|\Delta\bz\|   - \frac{\eta_2}{2}\|G^T\Delta\blambda\|^2\nonumber\\[3pt]
& \quad - (\Delta\bz)^T\cbr{\hH + \frac{\eta_1}{2}G^TG}\Delta\bz \nonumber\\[3pt]
& \leq - \frac{\eta_2}{2}\|\nabla\mL\|^2 + 3\eta_2\Upsilon^2\|\Delta\bz\|^2 + 2\|\Delta\blambda\|\|G\Delta\bz\| - \frac{\eta_2}{4}\|G^T\Delta\blambda\|^2- (\Delta\bz)^T\cbr{\hH + \frac{\eta_1}{2}G^TG}\Delta\bz \nonumber\\[3pt]
&\stackrel{\mathclap{\text{Lemma } \ref{lem:0}}}{\leq}\;\;\; - \frac{\eta_2}{2}\|\nabla\mL\|^2 + 3\eta_2\Upsilon^2\|\Delta\bz\|^2 + 2\|\Delta\blambda\|\|G\Delta\bz\| - \frac{\eta_2\gamma_G}{4}\|\Delta\blambda\|^2 - (\Delta\bz)^T\cbr{\hH + \frac{\eta_1}{2}G^TG}\Delta\bz \nonumber\\[3pt]
& \leq - \frac{\eta_2}{2}\|\nabla\mL\|^2- \frac{\eta_2\gamma_{G}}{8}\|\Delta\blambda\|^2 + 3\eta_2\Upsilon^2\|\Delta\bz\|^2  - (\Delta\bz)^T\cbr{\hH + \rbr{\frac{\eta_1}{2} - \frac{8}{\eta_2\gamma_G}}G^TG}\Delta\bz,
\end{align}	
where the fourth inequality uses Assumption \ref{ass:3} and \eqref{equ:upper} so that $\max\{\|\hH\|, \|H\|\}\leq \Upsilon$; the fifth and seventh inequalities use Young's inequalities:
\begin{align*}
\eta_2\Upsilon\|G^T\Delta\blambda\|\|\Delta\bz\| \leq & \frac{\eta_2}{4}\|G^T\Delta\blambda\|^2 + \eta_2\Upsilon^2\|\Delta\bz\|^2,\\
2\|\Delta\blambda\|\|G\Delta\bz\| \leq & \frac{\eta_2\gamma_{G}}{8}\|\Delta\blambda\|^2 + \frac{8}{\eta_2\gamma_{G}}\|G\Delta\bz\|^2.
\end{align*}
To simplify the last two terms, we suppose $\eta_1\geq 16/(\eta_2\gamma_{G})$, and decompose $\Delta\bz = \Delta\bu + \Delta\bv$, where $\Delta\bu\in \text{span}(G^T)$ so that $\Delta\bu = G^T\bar{\Delta}\bu$ for some $\bar{\Delta}\bu$, and $\Delta\bv\in \text{kernel}(G)$ so that $G\Delta\bv = \0$. Then,~we have $\|\Delta\bz\|^2 = \|\Delta\bu\|^2 + \|\Delta\bv\|^2$ and
\begin{align}\label{pequ:16}
3\eta_2\Upsilon^2 & \|\Delta\bz\|^2  - (\Delta\bz)^T\cbr{\hH + \rbr{\frac{\eta_1}{2} - \frac{8}{\eta_2\gamma_G}}G^TG}\Delta\bz \nonumber\\
& = 3\eta_2\Upsilon^2\|\Delta\bz\|^2 - (\Delta\bv)^T\hH\Delta\bv - 2(\Delta\bv)^T\hH\Delta\bu - (\Delta\bu)^T\hH\Delta\bu - \rbr{\frac{\eta_1}{2} - \frac{8}{\eta_2\gamma_G}}\|GG^T\bar{\Delta}\bu\|^2\nonumber\\
& \leq (3\eta_2\Upsilon^2 - \gamma_{RH})\|\Delta\bz\|^2 + 2\Upsilon\|\Delta\bv\|\|\Delta\bu\| + (\gamma_{RH} + \Upsilon)\|\Delta\bu\|^2 - \rbr{\frac{\eta_1}{2} - \frac{8}{\eta_2\gamma_G}}\gamma_{G}\|G^T\bar{\Delta}\bu\|^2 \nonumber\\
&\leq \rbr{3\eta_2\Upsilon^2 - \frac{\gamma_{RH}}{2}}\|\Delta\bz\|^2 + \rbr{  \gamma_{RH}  + \Upsilon + \frac{2\Upsilon^2}{\gamma_{RH}} + \frac{8}{\eta_2} - \frac{\eta_1\gamma_{G}}{2} }\|\Delta\bu\|^2,
\end{align}
where the second inequality uses Assumptions \ref{ass:1}, \ref{ass:3}, and Lemma \ref{lem:0}, and the third inequality~uses Young's inequality
\begin{equation*}
2\Upsilon\|\Delta\bv\|\|\Delta\bu\| \leq \frac{\gamma_{RH}}{2}\|\Delta\bv\|^2 + \frac{2\Upsilon^2}{\gamma_{RH}}\|\Delta\bu\|^2 \leq \frac{\gamma_{RH}}{2}\|\Delta\bz\|^2 + \frac{2\Upsilon^2}{\gamma_{RH}}\|\Delta\bu\|^2.
\end{equation*}
To make \eqref{pequ:16} negative, we let
\begin{equation}\label{pequ:eta:2}
3\eta_2\Upsilon^2 \leq \frac{\gamma_{RH}}{4} \Longleftrightarrow \eta_2\leq \frac{\gamma_{RH}}{12\Upsilon^2}.
\end{equation}
Furthermore, without loss of generality, we suppose $\Upsilon/2 \geq 1 \geq \max\{\gamma_{RH}, \gamma_G\}$, and have
\begin{equation*}
\gamma_{RH} + \Upsilon + \frac{2\Upsilon^2}{\gamma_{RH}} + \frac{8}{\eta_2} \leq\frac{3\Upsilon}{2} + \frac{2\Upsilon^2}{\gamma_{RH}} + \frac{8}{\eta_2} \leq \frac{3\Upsilon^2}{\gamma_{RH}}+ \frac{8}{\eta_2} \stackrel{\eqref{pequ:eta:2}}{\leq} \frac{1}{4\eta_2} + \frac{8}{\eta_2} \leq \frac{8.5}{\eta_2}.
\end{equation*}
Thus, we let 
\begin{equation}\label{pequ:eta:1}
\eta_1 \geq \frac{17}{\eta_2\gamma_{G}},
\end{equation}
which implies $\eta_1\geq\eta_2$ and $\eta_1\geq 16/(\eta_2\gamma_{G})$ as required in \eqref{pequ:14} and \eqref{pequ:16}. Thus,~under~\eqref{pequ:eta:1}~and~\eqref{pequ:eta:2}, the inequality \eqref{pequ:16} leads to
\begin{equation*}
3\eta_2\Upsilon^2\|\Delta\bz\|^2  - (\Delta\bz)^T\cbr{\hH + \rbr{\frac{\eta_1}{2} - \frac{8}{\eta_2\gamma_G}}G^TG}\Delta\bz \leq -\frac{\gamma_{RH}}{4}\|\Delta\bz\|^2  \stackrel{\eqref{pequ:eta:2}}{\leq} -\frac{\eta_2\gamma_{G}}{8}\|\Delta\bz\|^2.
\end{equation*}
Combining the above display with \eqref{pequ:14}, we obtain
\begin{equation}\label{pequ:15}
\I_1 =  \begin{pmatrix}
\nabla_{\bz}\mL_{\eta}\\
\nabla_{\blambda}\mL_{\eta}
\end{pmatrix}^T\begin{pmatrix}
\Delta\bz\\
\Delta\blambda
\end{pmatrix} \leq  -\frac{\eta_2}{2}\|\nabla\mL\|^2 - \frac{\eta_2\gamma_{G}}{8}\nbr{\begin{pmatrix}
\Delta\bz\\
\Delta\blambda
\end{pmatrix}}^2.
\end{equation}
For the second term in \eqref{pequ:13},
\begin{align*}
\I_2  &\coloneqq  \begin{pmatrix}
\nabla_{\bz}\mL_{\eta}\\
\nabla_{\blambda}\mL_{\eta}
\end{pmatrix}^T\begin{pmatrix}
\tDelta\bz - \Delta\bz\\
\tDelta\blambda - \Delta\blambda
\end{pmatrix} \stackrel{\substack{\eqref{equ:der:aug:L}\\ \eqref{equ:Newton}}}{=}- \begin{pmatrix}
\tDelta\bz - \Delta\bz\\
\tDelta\blambda - \Delta\blambda
\end{pmatrix}^T\begin{pmatrix}
(I+\eta_2H)\hH + \eta_1G^TG & (I+\eta_2H)G^T\\
G(I + \eta_2\hH) & \eta_2GG^T
\end{pmatrix}\begin{pmatrix}
\Delta\bz\\
\Delta\blambda
\end{pmatrix}\\
\stackrel{\eqref{equ:error:cond}}{\leq}& \delta\|(\Delta\bz, \Delta\blambda)\|^2\cbr{\Upsilon + \eta_2\Upsilon^2 + (\eta_1+\eta_2)\Upsilon^2 + 2(1+\eta_2\Upsilon)\Upsilon} = \delta\|(\Delta\bz, \Delta\blambda)\|^2\cbr{3\Upsilon + 4\eta_2\Upsilon^2 + \eta_1\Upsilon^2},
\end{align*}
where the third inequality also uses Assumption \ref{ass:3} and \eqref{equ:upper} so that $\max\{\|G\|, \|\hH\|,\|H\|\}\leq \Upsilon$. We use $\Upsilon/2 \geq 1 \geq \max\{\gamma_{RH}, \gamma_G\}$ and know that
\begin{equation}\label{npequ:1}
3\Upsilon + 4\eta_2\Upsilon^2 + \eta_1\Upsilon^2 \stackrel{\eqref{pequ:eta:2}}{\leq} 3\Upsilon + \frac{\gamma_{RH}}{3} + \eta_1\Upsilon^2\leq \frac{9.5\Upsilon}{3} + \eta_1\Upsilon^2 \leq 1.1\eta_1\Upsilon^2,
\end{equation}
where the last inequality uses	$\eta_1\stackrel{\eqref{pequ:eta:1}}{\geq} \frac{17}{\eta_2\gamma_G} > \frac{17\times 12\Upsilon^2}{\gamma_{RH}\gamma_G}$, so $9.5/3\leq 0.1\eta_1\Upsilon$. By the above two displays,
\begin{equation}\label{pequ:28}
\I_2  =\begin{pmatrix}
\nabla_{\bz}\mL_{\eta}\\
\nabla_{\blambda}\mL_{\eta}
\end{pmatrix}^T\begin{pmatrix}
\tDelta\bz - \Delta\bz\\
\tDelta\blambda - \Delta\blambda
\end{pmatrix} \leq 1.1\eta_1\delta\Upsilon^2\|(\Delta\bz, \Delta\blambda)\|^2.
\end{equation}
Combining \eqref{pequ:28} with \eqref{pequ:15} and \eqref{pequ:13},
\begin{equation}\label{npequ:2}
\begin{pmatrix}
\nabla_{\bz}\mL_{\eta}\\
\nabla_{\blambda}\mL_{\eta}
\end{pmatrix}^T\begin{pmatrix}
\tDelta\bz\\
\tDelta\blambda
\end{pmatrix} = \I_1+\I_2 \leq -\frac{\eta_2}{2}\nbr{\begin{pmatrix}
\nabla_{\bz}\mL\\
\nabla_{\blambda}\mL
\end{pmatrix}}^2 - \rbr{\frac{\eta_2\gamma_{G}}{8} - 1.1\eta_1\delta\Upsilon^2} \nbr{\begin{pmatrix}
\Delta\bz\\
\Delta\blambda
\end{pmatrix}}^2\leq - \frac{\eta_2}{2}\|\nabla\mL\|^2,
\end{equation}
where the last inequality holds if $\delta \leq \frac{\eta_2\gamma_{G}}{9\eta_1\Upsilon^2}$. This completes the proof.

\subsection{Proof of Theorem \ref{thm:5}}\label{pf:thm:5}

It suffices to prove the first part of the statement. The~\mbox{second} part holds immediately by Theorem \ref{thm:3} and specializes \eqref{equ:error:cond} with $\delta = C\rho^b$. By compactness of iterates and continuous differentiability of $\{g_k, f_k\}$ in Assumption \ref{ass:4}, we know $\sup_{\mZ\times \Lambda}\|\nabla^2\mL_\eta(\bz, \blambda)\| \leq \Upsilon_\eta$ for some $\Upsilon_\eta>0$ independent of $\tau$. We apply the Taylor expansion and obtain
\begin{align*}
\mL_{\eta}^{\tau+1}  &\leq \mL_{\eta}^\tau + \alpha_\tau\begin{pmatrix}
\nabla_{\bz}\mL_{\eta}^\tau\\
\nabla_{\blambda} \mL_{\eta}^\tau
\end{pmatrix}^T\begin{pmatrix}
\tDelta\bz^\tau\\
\tDelta\blambda^\tau
\end{pmatrix} + \frac{\Upsilon_\eta\alpha_\tau^2}{2}\nbr{\begin{pmatrix}
\tDelta\bz^\tau\\
\tDelta\blambda^\tau
\end{pmatrix}}^2\\
&\stackrel{\mathclap{\eqref{equ:error:cond}}}{\leq} \mL_{\eta}^\tau + \alpha_\tau\begin{pmatrix}
\nabla_{\bz}\mL_{\eta}^\tau\\
\nabla_{\blambda} \mL_{\eta}^\tau
\end{pmatrix}^T\begin{pmatrix}
\tDelta\bz^\tau\\
\tDelta\blambda^\tau
\end{pmatrix} + \frac{\Upsilon_\eta\alpha_\tau^2(1+\delta)^2}{2}\|(\Delta\bz^\tau,\Delta\blambda^\tau)\|^2\\
&\stackrel{\mathclap{\eqref{equ:Newton}, \text{ Lemma } \ref{lem:2}}}{\leq} \;\;\; \mL_{\eta}^\tau+ \alpha_\tau\begin{pmatrix}
\nabla_{\bz}\mL_{\eta}^\tau\\
\nabla_{\blambda} \mL_{\eta}^\tau
\end{pmatrix}^T\begin{pmatrix}
\tDelta\bz^\tau\\
\tDelta\blambda^\tau
\end{pmatrix} + 2\Upsilon_\eta\Upsilon^2 \alpha_\tau^2\|\nabla\mL^\tau\|^2\\
&\stackrel{\mathclap{\text{Theorem } \ref{thm:4}}}{\leq}\;\;\; \mL_{\eta}^\tau+ \alpha_\tau\begin{pmatrix}
\nabla_{\bz}\mL_{\eta}^\tau\\
\nabla_{\blambda} \mL_{\eta}^\tau
\end{pmatrix}^T\begin{pmatrix}
\tDelta\bz^\tau\\
\tDelta\blambda^\tau
\end{pmatrix} - \frac{4\Upsilon_\eta\Upsilon^2}{\eta_2} \alpha_\tau^2\begin{pmatrix}
\nabla_{\bz}\mL_{\eta}^\tau\\
\nabla_{\blambda} \mL_{\eta}^\tau
\end{pmatrix}^T\begin{pmatrix}
\tDelta\bz^\tau\\
\tDelta\blambda^\tau
\end{pmatrix}\\
& = \mL_{\eta}^\tau+ \alpha_\tau\begin{pmatrix}
\nabla_{\bz}\mL_{\eta}^\tau\\
\nabla_{\blambda} \mL_{\eta}^\tau
\end{pmatrix}^T\begin{pmatrix}
\tDelta\bz^\tau\\
\tDelta\blambda^\tau
\end{pmatrix}\cdot \cbr{1 - \frac{4\Upsilon_\eta\Upsilon^2}{\eta_2}\alpha_\tau},
\end{align*}
where the third inequality also uses $\delta\leq 1$. Thus, as long as
\begin{equation*}
1 - \frac{4\Upsilon_\eta\Upsilon^2}{\eta_2}\alpha_\tau \geq \beta \Longleftrightarrow \alpha_\tau \leq \frac{(1-\beta)\eta_2}{4\Upsilon_\eta\Upsilon^2},
\end{equation*}
the Armijo condition \eqref{equ:Armijo} is satisfied. Since the right hand side is independent of $\tau$, we know $\alpha_\tau \geq \bar{\alpha}$ for some $\bar{\alpha} >0$ when doing, for example, a backtracking line search. By \eqref{equ:Armijo} and Theorem~\ref{thm:4},
\begin{equation*}
\mL_{\eta}^{\tau+1} \leq \mL_{\eta}^\tau -\frac{\eta_2\alpha_\tau\beta}{2}\|\nabla\mL^\tau\|^2 \leq\mL_{\eta}^\tau - \frac{\eta_2\bar{\alpha}\beta}{2}\|\nabla\mL^\tau\|^2.
\end{equation*}
Summing over $\tau$, $\sum_{\tau=0}^{\infty} \|\nabla\mL^\tau\|^2 \leq \frac{2}{\eta_2\bar{\alpha}\beta}(\mL_{\eta}^0 - \min_{\mZ\times\Lambda}\mL_{\eta}(\bz, \blambda))<\infty$. This completes the proof.

\section{Proofs of results in \cref{sec:6}}

\subsection{Proof of Theorem \ref{thm:6}}\label{pf:thm:6}

By \eqref{equ:Armijo}, it suffices to show for large $\tau$ that
\begin{equation}\label{pequ:s}
\mL_{\eta}(\bz^\tau + \tDelta\bz^\tau, \blambda^\tau + \tDelta\blambda^\tau) \leq \mL_{\eta}^\tau + \beta\begin{pmatrix}
\nabla_{\bz}\mL_{\eta}^\tau\\
\nabla_{\blambda}\mL_{\eta}^\tau
\end{pmatrix}^T\begin{pmatrix}
\tDelta\bz^\tau\\
\tDelta\blambda^\tau
\end{pmatrix}.
\end{equation}
By the thrice continuous differentiability of $\{g_k, f_k\}$, we know $\nabla^2\mL_\eta$ is continuous. Thus, we have the following Taylor expansion
\begin{multline}\label{equ:taylor}
\mL_{\eta}(\bz^\tau + \tDelta\bz^\tau, \blambda^\tau + \tDelta\blambda^\tau) \\
\leq\mL_{\eta}^\tau + \begin{pmatrix}
\nabla_{\bz}\mL_{\eta}^\tau\\
\nabla_{\blambda}\mL_{\eta}^\tau
\end{pmatrix}^T\begin{pmatrix}
\tDelta\bz^\tau\\
\tDelta\blambda^\tau
\end{pmatrix} + \frac{1}{2}\begin{pmatrix}
\tDelta\bz^\tau\\
\tDelta\blambda^\tau
\end{pmatrix}^T\nabla^2\mL_{\eta}^\tau\begin{pmatrix}
\tDelta\bz^\tau\\
\tDelta\blambda^\tau
\end{pmatrix} + o(\|(\tDelta\bz^\tau, \tDelta\blambda^\tau)\|^2).
\end{multline}
By direct calculation, we have that
\begin{subequations}
\begin{align*}
\nabla_{\bz}^2\mL_{\eta} = & H + \eta_1G^TG + \eta_2H^2 + \eta_1\LD \nabla_{\bz}G, f\RD + \eta_2\LD \nabla_{\bz}H, \nabla_{\bz}\mL\RD, \\
\nabla_{\bz\blambda}\mL_{\eta} = & G^T + \eta_2HG^T + \eta_2\LD\nabla_{\blambda}H, \nabla_{\bz}\mL\RD, \quad\quad
\nabla_{\blambda}^2\mL_{\eta} =  \eta_2GG^T, 
\end{align*}
\end{subequations}
where for a function $a(x):\mR^{n}\rightarrow \mR^{m_1\times m_2}$ and a vector $b\in \mR^{m_1}$, we let $\LD \nabla a(x), b\RD \coloneqq \nabla^T(a^T(x)b) = \sum_{j=1}^{m_1}b_j\nabla^T a_j(x) \in \mR^{m_2\times n}$ for $a^T(x) = (a_1(x), \ldots, a_{m_1}(x))$. Thus, we define
\begin{equation*}
\H^\tau\coloneqq\begin{pmatrix}
H^\tau + \eta_1(G^{\tau})^TG^{\tau} + \eta_2 (H^{\tau})^2 & (I + \eta_2H^\tau)(G^\tau)^T\\
G^\tau(I + \eta_2H^\tau) & \eta_2G^\tau(G^\tau)^T
\end{pmatrix},
\end{equation*}
apply $\|\nabla\mL^\tau\| = \|(\nabla_{\bz}\mL^\tau, f^\tau)\|\rightarrow 0$, and get
\begin{equation}\label{equ:Hess:diff}
\nbr{\nabla^2\mL_{\eta}^\tau - \H^\tau} = o(1).
\end{equation}
Combining \eqref{pequ:s}, \eqref{equ:taylor} and \eqref{equ:Hess:diff}, it suffices to show that
\begin{equation}\label{pequ:18}
(1- \beta) \begin{pmatrix}
\nabla_{\bz}\mL_{\eta}^\tau\\
\nabla_{\blambda}\mL_{\eta}^\tau
\end{pmatrix}^T\begin{pmatrix}
\tDelta\bz^\tau\\
\tDelta\blambda^\tau
\end{pmatrix} + \frac{1}{2}\begin{pmatrix}
\tDelta\bz^\tau\\
\tDelta\blambda^\tau
\end{pmatrix}^T\H^\tau\begin{pmatrix}
\tDelta\bz^\tau\\
\tDelta\blambda^\tau
\end{pmatrix} + o(\|(\tDelta\bz^\tau, \tDelta\blambda^\tau)\|^2) \leq 0.
\end{equation}
We observe that
\begin{align*}
\begin{pmatrix}
\nabla_{\bz}\mL_{\eta}^\tau\\
\nabla_{\blambda}\mL_{\eta}^\tau
\end{pmatrix}^T\begin{pmatrix}
\tDelta\bz^\tau\\
\tDelta\blambda^\tau
\end{pmatrix} &+ \begin{pmatrix}
\tDelta\bz^\tau\\
\tDelta\blambda^\tau
\end{pmatrix}^T\H^\tau\begin{pmatrix}
\tDelta\bz^\tau\\
\tDelta\blambda^\tau
\end{pmatrix}\\
& \stackrel{\mathclap{\substack{\eqref{equ:der:aug:L}\\\eqref{equ:Newton}}}}{=} \; \begin{pmatrix}
\tDelta\bz^\tau\\
\tDelta\blambda^\tau
\end{pmatrix}^T\cbr{\H^\tau - \begin{pmatrix}
\hH^\tau + \eta_1(G^{\tau})^TG^{\tau} + \eta_2 H^\tau \hH^{\tau} & (I + \eta_2H^\tau)(G^\tau)^T\\
G^\tau(I + \eta_2\hH^\tau) & \eta_2G^\tau(G^\tau)^T
\end{pmatrix} }\begin{pmatrix}
\Delta\bz^\tau\\
\Delta\blambda^\tau
\end{pmatrix}\\
& \quad + \begin{pmatrix}
\tDelta\bz^\tau\\
\tDelta\blambda^\tau
\end{pmatrix}^T\H^\tau\begin{pmatrix}
\tDelta\bz^\tau - \Delta\bz^\tau\\
\tDelta\blambda^\tau - \Delta\blambda^\tau
\end{pmatrix}\eqqcolon \I_3 + \I_4.
\end{align*}
For term $\I_3$, we let $\Delta H^\tau = H^\tau - \hH^\tau$ and have
\begin{align}\label{pequ:19}
\I_3 = & \begin{pmatrix}
\tDelta\bz^\tau\\
\tDelta\blambda^\tau
\end{pmatrix}^T\begin{pmatrix}
(I + \eta_2H^\tau)\Delta H^\tau  & \0\\
\eta_2G^\tau \Delta H^\tau & \0
\end{pmatrix}\begin{pmatrix}
\Delta\bz^\tau\\
\Delta\blambda^\tau
\end{pmatrix} 
\stackrel{\eqref{equ:upper}}{\leq} \|(\tDelta\bz^\tau, \tDelta\blambda^\tau)\| \cdot(1+2\eta_2\Upsilon )\|\Delta H^\tau\Delta\bz^\tau\| \nonumber\\[3pt]
\leq & 2(1+2\eta_2\Upsilon)\|(\Delta\bz^\tau, \Delta\blambda^\tau)\| o(\|\Delta\bz^\tau\|) = o(\|(\Delta\bz^\tau, \Delta\blambda^\tau)\|^2),
\end{align}	
where the third inequality uses \eqref{equ:error:cond} and Assumption \ref{ass:5}.	For term $\I_4$, we apply Assumptions \ref{ass:3} and \eqref{equ:upper}, and have
\begin{multline*}
\I_4  \leq \|(\tDelta\bz^\tau, \tDelta\blambda^\tau)\|\cdot \|\H^\tau\|\cdot\|(\Delta\bz^\tau - \tDelta\bz^\tau, \Delta\blambda^\tau - \tDelta\blambda^\tau)\|\\[3pt]
\stackrel{\substack{\eqref{equ:error:cond}, \eqref{equ:upper}}}{\leq}  2\delta\|(\Delta\bz^\tau, \Delta\blambda^\tau)\|^2 \cbr{\Upsilon + \eta_2\Upsilon^2 + (\eta_1+\eta_2)\Upsilon^2 + 2(1+\eta_2\Upsilon)\Upsilon}\stackrel{\eqref{npequ:1}}{\leq} 2.2\delta\eta_1\Upsilon^2\|(\Delta\bz^\tau, \Delta\blambda^\tau)\|^2.
\end{multline*}	
Combining the above display with \eqref{pequ:18}, \eqref{pequ:19}, and using  $o(\|(\tDelta\bz^\tau, \tDelta\blambda^\tau)\|) \stackrel{\eqref{equ:error:cond}}{=} o(\|(\Delta\bz^\tau, \Delta\blambda^\tau)\|)$, it suffices to show that
\begin{equation*}
(\frac{1}{2} - \beta)\begin{pmatrix}
\nabla_{\bz}\mL_{\eta}^\tau\\
\nabla_{\blambda}\mL_{\eta}^\tau
\end{pmatrix}^T\begin{pmatrix}
\tDelta\bz^\tau\\
\tDelta\blambda^\tau
\end{pmatrix} + 1.1 \delta\eta_1\Upsilon^2\|(\Delta\bz^\tau, \Delta\blambda^\tau)\|^2 + o(\|(\Delta\bz^\tau, \Delta\blambda^\tau)\|^2)\leq 0.
\end{equation*}	
By \eqref{pequ:15}, \eqref{pequ:28}, \eqref{npequ:2} in the proof of Theorem \ref{thm:4}, the above inequality holds if	
\begin{equation*}
(\frac{1}{2} - \beta)\rbr{\frac{\eta_2\gamma_G}{8} - 1.1\delta\eta_1\Upsilon^2}\geq 1.1\delta\eta_1\Upsilon^2\Longleftarrow \delta \leq \frac{1/2-\beta}{3/2-\beta}\cdot\frac{\eta_2\gamma_{G}}{9\eta_1\Upsilon^2}.
\end{equation*}
This completes the proof.

\subsection{Proof of Theorem \ref{thm:7}}\label{pf:thm:7}

It suffices to show that the Newton system of $\P_{\mu}^i(\bd_i^\tau)$ at $\mD_i(\bz^\tau, \blambda^\tau)$ is the same as \eqref{pro:4} with $\hH_k^\tau = H_k^\tau$. We suppress the iteration index $\tau$. The Newton system of $\P_{\mu}^i(\bd_i^\tau)$ can be expressed~as
\begin{equation}\label{pequ:21}
\begin{pmatrix}
H^{(i)} & (G^{(i)})^T\\
G^{(i)} & \0
\end{pmatrix}\begin{pmatrix}
\Delta\bz^{(i)}\\
\Delta\blambda^{(i)}
\end{pmatrix} = -\begin{pmatrix}
\nabla_{\tbz_i}\mL^{(i)}\\
\nabla_{\tblambda_i}\mL^{(i)}
\end{pmatrix},
\end{equation}
where $\mL^{(i)}$ is the Lagrangian function of $\P_{\mu}^i(\bd_i^\tau)$, and $H^{(i)} = \nabla_{\tbz_i}^2\mL^{(i)}$ and $G^{(i)} = \nabla_{\tblambda_i\tbz_i}\mL^{(i)}$. By direct calculation and using the setup of $\bd_i$ of procedure (a), we have 
\begin{subequations}\label{pequ:22}
\begin{align}
H^{(i)} = &\diag(H_{m_1}, \ldots, H_{m_2-1}, Q_{m_2} + \mu I), \label{pequ:22:a}\\[3pt]
G^{(i)} = & \left(\begin{smallmatrix}
I\\
-A_{m_1} & -B_{m_1} & I\\
& & -A_{m_1+1} & -B_{m_1 + 1} & I \\
& & & & \ddots & \ddots & \ddots\\
& & & & & & -A_{m_2-1} & -B_{m_2-1} & I
\end{smallmatrix}\right), \label{pequ:22:b}
\end{align}
\end{subequations}
and
\begin{equation}\label{pequ:23}
\nabla_{\tbz_i}\mL^{(i)} =  (\nabla_{\bz_{m_1}}\mL; \ldots; \nabla_{\bz_{m_2-1}}\mL; \nabla_{\bx_{m_2}}\mL), \quad\quad 
\nabla_{\tblambda_i}\mL^{(i)} = (\0; \nabla_{\blambda_{m_1+1}}\mL;\ldots; \nabla_{\blambda_{m_2}}\mL).
\end{equation}
Plugging \eqref{pequ:22} and \eqref{pequ:23} into \eqref{pequ:21}, we observe that \eqref{pequ:21} is the same as $\LP_{\mu}^i(\bd_i)$ in \eqref{pro:4} with $\bd_i = (\0;\0;\0;\0)$ and $\hH_k = H_k$. Thus,
\begin{equation}\label{pequ:100}
(\Delta\bz^{(i)}, \Delta\blambda^{(i)}) = (\tbwi(\bd_i), \tbzetai(\bd_i)).
\end{equation}
Moreover, we denote by $(\bz_s^{\tau+1}, \blambda_s^{\tau+1})$ and $(\bz_f^{\tau+1}, \blambda_f^{\tau+1})$ the next iterate generated by the one-Newton-step Schwarz scheme and generated by FOTD, respectively. We have	
\begin{align*}
(\bz_s^{\tau+1}, \blambda_s^{\tau+1}) & = \mC\rbr{\cbr{ \mD_i(\bz^\tau, \blambda^\tau) + (\Delta\bz^{(i)}, \Delta\blambda^{(i)})}_i }
=\mC\rbr{\cbr{\mD_i(\bz^\tau, \blambda^\tau)}_i} + \mC\rbr{\cbr{(\Delta\bz^{(i)}, \Delta\blambda^{(i)})}_i }\\
& \stackrel{\mathclap{\eqref{pequ:100}}}{=} \;\; \mC\rbr{\cbr{\mD_i(\bz^\tau, \blambda^\tau)}_i} + \mC\rbr{\cbr{(\tbwi(\bd_i), \tbzetai(\bd_i)) }_i }
= \mC\rbr{\cbr{\mD_i(\bz^\tau, \blambda^\tau)}_i} + (\tDelta\bz^\tau, \tDelta\blambda^\tau)\\
&=  (\bz^\tau, \blambda^\tau) + (\tDelta\bz^\tau, \tDelta\blambda^\tau) = (\bz_f^{\tau+1}, \blambda_f^{\tau+1}),
\end{align*}
where the first, fourth and last equalities are due to the definitions of the Schwarz and the FOTD procedures; the second and fifth equalities are due to Definition \ref{def:1}. This completes the proof.

\subsection{Proof of Lemma \ref{lem:3}}\label{pf:lem:3}

Our proof relies on the KKT inverse structure in \cite[Lemma 2]{Na2023Superconvergence}. We only show \eqref{equ:result:b}, while \eqref{equ:result:c} holds by recalling that the last subproblem does not have~boundary variables at the terminal stage $N$. For subproblem $i\in[M-2]$, we let $(\tbzi, \tblambdai) = \mD_i(\tz, \tlambda)$. Let $C, \rho$ be the constants in Theorem \ref{thm:3}, and let $C_1 = 9C\Upsilon^2/\gamma_{G}$. Then, under the assumptions and the setup of $b$ in \eqref{equ:b:cond}, $\delta = \gamma_{G}C_1\rho^b/(9\Upsilon^2)$ satisfies \eqref{cond:2}. Suppose $\tau$ is large enough so that $\alpha_\tau = 1$. Borrowing the notation in \eqref{pequ:21}-\eqref{pequ:23}, we let $\hH^{(i)}$, $G^{(i)}$, $\nabla\mL^{(i)}$ be the Hessian, Jacobian, and KKT residual vector of Problem \eqref{pro:4} at the $\tau$-th iterate $(\tbz_i^{\tau}, \tblambda_i^{\tau})$. We consider the FOTD update:
\begin{multline}\label{pequ:24}
\begin{pmatrix}
\tbz_i^{\tau+1} - \tbzi\\
\tblambda_i^{\tau+1} - \tblambdai
\end{pmatrix} = \begin{pmatrix}
\tbz_i^{\tau} - \tbzi\\
\tblambda_i^{\tau} - \tblambdai
\end{pmatrix} + \begin{pmatrix}
\tbwi(\bd_i)\\
\tbzetai(\bd_i)
\end{pmatrix} \\
\stackrel{\substack{\eqref{pro:4} \\ \text{Theorem } \ref{thm:7}}}{=} \begin{pmatrix}
\hH^{(i)} & (G^{(i)})^T\\
G^{(i)} & \0
\end{pmatrix}^{-1}\bigg\{ \begin{pmatrix}
\hH^{(i)} & (G^{(i)})^T\\
G^{(i)} & \0
\end{pmatrix}\begin{pmatrix}
\tbz_i^{\tau} - \tbzi\\
\tblambda_i^{\tau} - \tblambdai
\end{pmatrix} - \begin{pmatrix}
\nabla_{\tbz_i}\mL^{(i)}\\
\nabla_{\tblambda_i}\mL^{(i)}
\end{pmatrix} \bigg\}.
\end{multline}	
We define the KKT residual evaluated at the truncated full horizon solution $(\tbzi, \tblambdai)$~as
\begin{equation}\label{pequ:200}
\begin{aligned}
\nabla_{\tbz_i}\mL^{(i), \star} = & (\nabla_{\bz_{m_1}}\mL^\star; \ldots; \nabla_{\bz_{m_2-1}}\mL^\star; \nabla_{\bx_{m_2}}\tilde{\mL}^\star),\\
\nabla_{\tblambda_i}\mL^{(i), \star} = & (\tx_{m_1} - \bx_{m_1}^\tau;\nabla_{\blambda_{m_1+1}}\mL^\star; \ldots; \nabla_{\blambda_{m_2}}\mL^\star),
\end{aligned}
\end{equation}
where $\nabla_{\bz_{m_1:m_2-1}}\mL^\star$ (similar for $\nabla_{\blambda_{m_1+1:m_2}}\mL^\star$) replaces the evaluation point $(\tbz_i^\tau, \tblambda_i^\tau)$~of the components $\nabla_{\bz_{m_1:m_2-1}}\mL$ of $\nabla_{\tbz_i}\mL^{(i)}$ (cf. \eqref{pequ:23}) with $(\tbzi, \tblambdai)$, and 
\begin{equation}\label{pequ:29}
\nabla_{\bx_{m_2}}\tilde{\mL}^\star \coloneqq \nabla_{\bx_{m_2}}g_{m_2}(\tx_{m_2}, \bu_{m_2}^\tau) + \tlambda_{m_2} - A_{m_2}^T(\tx_{m_2}, \bu_{m_2}^\tau)\blambda_{m_2+1}^\tau + \mu(\tx_{m_2} - \bx_{m_2}^\tau).
\end{equation}
Clearly, if we change the evaluation point from $(\tbzi, \tblambdai)$ back to $(\tbz_i^\tau, \tblambda_i^\tau)$ in \eqref{pequ:200} and \eqref{pequ:29}, then we get the vectors $\nabla_{\tbz_i}\mL^{(i)}$ and $\nabla_{\tblambda_i}\mL^{(i)}$ in \eqref{pequ:24}. Moreover, for $0\leq \phi\leq 1$, we let 
\begin{align*}
\tu_{m_2}(\phi) = & \tu_{m_2} + \phi(\bu_{m_2}^\tau - \tu_{m_2}), \quad\quad\quad \tlambda_{m_2+1}(\phi) = \tlambda_{m_2+1} + \phi(\blambda_{m_2+1}^\tau - \tlambda_{m_2+1}),\\[3pt]
\tbzi(\phi) = & \tbzi + \phi(\tbz_i^\tau - \tbzi), \hskip2.6cm
\tblambdai(\phi) = \tblambdai + \phi(\tblambda_i^\tau - \tblambdai),
\end{align*}
and let $H^{(i)}(\phi)$, $G^{(i)}(\phi)$ be $H^{(i)}$, $G^{(i)}$ (cf. \eqref{pequ:22}) evaluated at $(\tbzi(\phi), \tblambdai(\phi))$. Then, \eqref{pequ:24}~implies
\begin{align*}
\begin{pmatrix}
\tbz_i^{\tau+1} - \tbzi\\
\tblambda_i^{\tau+1} - \tblambdai
\end{pmatrix} = &-\begin{pmatrix}
\hH^{(i)} & (G^{(i)})^T\\
G^{(i)} & \0
\end{pmatrix}^{-1}\begin{pmatrix}
\nabla_{\tbz_i}\mL^{(i), \star}\\
\nabla_{\tblambda_i}\mL^{(i), \star}
\end{pmatrix} \\
&\quad  + \begin{pmatrix}
\hH^{(i)} & (G^{(i)})^T\\
G^{(i)} & \0
\end{pmatrix}^{-1} \cbr{\begin{pmatrix}
\hH^{(i)} & (G^{(i)})^T\\
G^{(i)} & \0
\end{pmatrix}\begin{pmatrix}
\tbz_i^{\tau} - \tbzi\\
\tblambda_i^{\tau} - \tblambdai
\end{pmatrix} - \begin{pmatrix}
\nabla_{\tbz_i}\mL^{(i)} - \nabla_{\tbz_i}\mL^{(i), \star}\\
\nabla_{\tblambda_i}\mL^{(i)} - \nabla_{\tblambda_i}\mL^{(i), \star}
\end{pmatrix} }\\
=& -\begin{pmatrix}
\hH^{(i)} & (G^{(i)})^T\\
G^{(i)} & \0
\end{pmatrix}^{-1}\begin{pmatrix}
\nabla_{\tbz_i}\mL^{(i), \star}\\
\nabla_{\tblambda_i}\mL^{(i), \star}
\end{pmatrix}\\
& \quad + \begin{pmatrix}
\hH^{(i)} & (G^{(i)})^T\\
G^{(i)} & \0
\end{pmatrix}^{-1}\int_{0}^1 \begin{pmatrix}
\hH^{(i)} - H^{(i)}(\phi) & (G^{(i)})^T - (G^{(i)}(\phi))^T\\
G^{(i)} - G^{(i)}(\phi) & \0
\end{pmatrix}\begin{pmatrix}
\tbz_i^{\tau} - \tbzi\\
\tblambda_i^{\tau} - \tblambdai
\end{pmatrix} d\phi\\
\eqqcolon& \K_i^{-1}\J_1 + \K_i^{-1}\J_2.
\end{align*}
To establish the stagewise error recursion, it suffices to establish the blockwise bound for the KKT inverse $\K_i^{-1}$ and the component-wise bound for vectors $\J_1$ and $\J_2$. The KKT inverse structure is given by \cite[Lemma 2]{Na2023Superconvergence} (the conditions are satisfied by Corollary \ref{cor:1}). We now deal with $\J_1$~and~$\J_2$.

\vskip4pt
\noindent\textbf{Term $\J_1$}. By Theorem \ref{thm:1}(i) (or checking the KKT conditions \eqref{equ:full:KKT} in the appendix), we know $\nabla_{\bz_{m_1:m_2-1}}\mL^\star = \0$ and $\nabla_{\blambda_{m_1+1:m_2}}\mL^\star = \0$. Thus, only the last component of $\nabla_{\tbz_i}\mL^{(i), \star}$ and the first component of $\nabla_{\tblambda_i}\mL^{(i), \star}$ are nonzero. The first component of $\nabla_{\tblambda_i}\mL^{(i), \star}$ is trivially bounded by $\|\tx_{m_1} - \bx_{m_1}^\tau\|$. For the last component of $\nabla_{\tbz_i}\mL^{(i), \star}$, we have
\begin{align*}
\nabla_{\bx_{m_2}}\tilde{\mL}^{\star} \;\;& \stackrel{\mathclap{\eqref{pequ:29}}}{=}\;\; \cbr{\nabla_{\bx_{m_2}}g_{m_2}(\tx_{m_2}, \bu_{m_2}^\tau) + \tlambda_{m_2} - A_{m_2}^T(\tx_{m_2}, \bu_{m_2}^\tau)\blambda_{m_2+1}^\tau } + \mu(\tx_{m_2} - \bx_{m_2}^\tau) \nonumber\\[3pt]
&\quad\quad -  \underbrace{\cbr{\nabla_{\bx_{m_2}}g_{m_2}(\tx_{m_2}, \tu_{m_2}) + \tlambda_{m_2} - A_{m_2}^T(\tx_{m_2}, \tu_{m_2})\tlambda_{m_2+1}}}_{\text{this is }\0 \text{ by KKT conditions (cf. \eqref{equ:full:KKT})}} \nonumber\\
& = \int_{0}^{1}\big(S_{m_2}^T(\tx_{m_2}, \tu_{m_2}(\phi), \tlambda_{m_2+1}(\phi))\;\;  -A_{m_2}^T(\tx_{m_2}, \tu_{m_2}(\phi))\big)\begin{pmatrix}
\bu_{m_2}^\tau - \tu_{m_2}\\
\blambda_{m_2+1}^\tau - \tlambda_{m_2+1}
\end{pmatrix}d\phi \nonumber\\
&\quad \quad+  \mu(\tx_{m_2} - \bx_{m_2}^\tau)\nonumber\\
& \stackrel{\mathclap{\eqref{equ:upper}}}{\leq}  (2\Upsilon+\mu)\|(\bz_{m_2}^\tau - \tz_{m_2}; \blambda_{m_2+1}^\tau - \tlambda_{m_2+1})\|.
\end{align*}
Thus, we have ($\preceq$ means component-wise $\leq$)
\begin{equation}\label{pequ:J_1}
\J_1 \preceq {\scriptsize \left(\begin{array}{c}
0\\
\vdots\\
0\\
(2\Upsilon+\mu)\nbr{\left(\begin{smallmatrix}
\bz_{m_2}^\tau - \tz_{m_2}\\
\blambda_{m_2+1}^\tau - \tlambda_{m_2+1}
\end{smallmatrix}\right)}\\  \hdashline[2pt/2pt]
\|\bx_{m_1}^\tau - \tx_{m_1}\|\\
0\\
\vdots\\
0
\end{array}\right)},\quad\quad \J_2\asymp {\scriptsize \left(\begin{array}{c}
o(\Psi_{m_1}^\tau) + O((\Psi_{m_1+1}^\tau)^2)\\
\vdots\\
o(\Psi_{m_2-1}^{\tau}) + O((\Psi_{m_2}^\tau)^2)\\
o(\Psi_{m_2}^\tau)\\  \hdashline[2pt/2pt]
0\\
O((\Psi_{m_1}^\tau)^2)\\
\vdots\\
O((\Psi_{m_2-1}^\tau)^2)
\end{array}\right)}.
\end{equation}
\noindent\textbf{Term $\J_2$}. Applying Assumption \ref{ass:5} so that $\|\hH^{(i)} - H^{(i)}\| = o(1)$, and using the Lipschitz continuity of $H^{(i)}$ and $\{A_k, B_k\}$ assumed by Assumption \ref{ass:6}, we immediately obtain \eqref{pequ:J_1} for $\J_2$ ($\asymp$ means component-wise $=$).

Finally, we apply \cite[Lemma 2]{Na2023Superconvergence} and know that, there exists a constant $\tilde{C}>0$ independent of $\tau$ and algorithmic parameters $\beta,\eta_1,\eta_2$ ($\rho$ is the same as Theorem \ref{thm:3}), such that $\forall k\in[m_1, m_2]$,
\begin{align}\label{pequ:400}
\|\tbz_{i, k}^{\tau+1} - \tz_k\| \leq & \tilde{C}\cbr{ \rho^{k-m_1}\|\bx_{m_1}^\tau - \tx_{m_1}\| + \rho^{m_2-k}\nbr{\begin{pmatrix}
\bz_{m_2}^\tau - \tz_{m_2}\\
\blambda_{m_2+1}^\tau - \tlambda_{m_2+1}
\end{pmatrix}} }  + \tilde{C}\sum_{j=m_1}^{m_2}\rho^{|k-j|}\cdot o(\Psi^\tau) \nonumber\\
= & o(\Psi^\tau) + \tilde{C}\cbr{ \rho^{k-m_1}\|\bx_{m_1}^\tau - \tx_{m_1}\| + \rho^{m_2-k}\nbr{\begin{pmatrix}
\bz_{m_2}^\tau - \tz_{m_2}\\
\blambda_{m_2+1}^\tau - \tlambda_{m_2+1}
\end{pmatrix}} },
\end{align}
where the second equality is due to $\sum_{j=m_1}^{m_2}\rho^{|k-j|} \leq 2\sum_{j=0}^{\infty}\rho^j<\infty$. The inequality \eqref{pequ:400} holds for $\|\tblambda_{i, k}^\tau - \tlambda_k\|$ as well. Using $\Psi_k^{\tau+1} \leq \|\tbz_{i, k}^{\tau+1} - \tz_k\| + \|\tblambda_{i, k}^\tau - \tlambda_k\|$, we rescale $C_1$ by $C_1\leftarrow \max\{C_1, 2\tilde{C}\}$~and complete the proof.

\end{APPENDICES}


\bibliographystyle{informs2014} 
\bibliography{ref} 

\begin{flushright}
\scriptsize \framebox{\parbox{\textwidth}{Government License: The submitted manuscript has been created by UChicago Argonne, LLC, Operator of Argonne National Laboratory (``Argonne"). Argonne, a U.S. Department of Energy Office of Science laboratory, is operated under Contract No. DE-AC02-06CH11357.  The U.S. Government retains for itself, and others acting on its behalf, a paid-up nonexclusive, irrevocable worldwide license in said article to reproduce, prepare derivative works, distribute copies to the public, and perform publicly and display publicly, by or on behalf of the Government. The Department of Energy will provide public access to these results of federally sponsored research in accordance with the DOE Public Access Plan. http://energy.gov/downloads/doe-public-access-plan. }}
\normalsize
\end{flushright}

\end{document}